\documentclass[11pt, leqno]{amsart}
\usepackage{latexsym}
\usepackage{amsmath}
\usepackage{amssymb}
\usepackage{amsfonts}
\usepackage{esint}
\usepackage{graphicx}
\usepackage{accents}
\usepackage{color}
\definecolor{Blue}{rgb}{0.,0.,1.}
\definecolor{Red}{rgb}{1.,0.,0.}
\newcounter{smallarabics}
\newenvironment{arabicenumerate}
{\begin{list}{{\normalfont\textrm{(\arabic{smallarabics})}}}
  {\usecounter{smallarabics}\setlength{\item[--]indent}{0cm}
   \setlength{\leftmargin}{5ex}\setlength{\labelwidth}{4ex}
   \setlength{\topsep}{0.75\parsep}\setlength{\partopsep}{0ex}
   \setlength{\item[--]sep}{0ex}}}
{\end{list}}

\newcounter{smallroman}

\newcommand{\ben}{\begin{arabicenumerate}}  
\newcommand{\een}{\end{arabicenumerate}}

\newtheorem{theoreme}{Theorem }[section]
\newtheorem{theorem}[theoreme]{Theorem}
\newtheorem{proposition}[theoreme]{Proposition}
\newtheorem{lemma}[theoreme]{Lemma}
\newtheorem{definition}[theoreme]{Definition}
\newtheorem{corollary}[theoreme]{Corollary}
\newtheorem{remark}[theoreme]{Remark}
\newtheorem{example}[theoreme]{Example}

\newcommand{\beq}{\begin{equation}}
\newcommand{\eeq}{\end{equation}}

\newcommand{\bex}{\begin{example}}
\newcommand{\eex}{\end{example}}
\def\bel{\begin{lemma}}
\def\eel{\end{lemma}}
\def\bet{\begin{theoreme}}
\def\eet{\end{theoreme}}
\def\bed{\begin{definition}}
\def\eed{\end{definition}}
\def\ber{\begin{remark}}
\def\eer{\end{remark}}

\def\rag{\rangle}
\def\lag{\langle}

\def\braket#1#2{\langle{#1}|{#2}\rangle}

\def\jap#1{\langle {#1} \rangle}

\def\slim{\mbox{\rm s-}\!\lim}

\def\nin{\notin}

\def\Ker{\mbox{\rm Ker\,}}

\def\supp{\mbox{\rm supp\! }}

\def\nin{\notin}

\def\qed{\hfill \raisebox{0.5ex}{\framebox[1.6ex]{
                                       \rule[0ex]{0ex}{0.3ex} }}}

\def\build#1_#2^#3{\mathrel{\mathop{\kern 0pt#1}\limits_{#2}^{#3}}}

\def\rr{{\mathbb R}}
\def\zz{{\mathbb Z}}
\def\cc{{\mathbb C}}
\def\nn{{\mathbb N}}

\def\slim{{\rm s-}\lim}

\def\bar{\overline}

\def\coinf{C_0^\infty}

\def\cV{{\mathcal V}}

\def\qed{$\Box$\medskip}
\def \p{ \partial}

\def\12{\frac{1}{2}}

\def\supp{{\rm supp}}

\def\e{{\rm e}}

\def\bbbone{{\mathchoice {\rm 1\mskip-4mu l} {\rm 1\mskip-4mu l}
{\rm 1\mskip-4.5mu l} {\rm 1\mskip-5mu l}}}
\def\one{\bbbone}
\def\cH{{\mathcal H}}

\def\Ker{{\rm Ker}}

\def\C{{\mathcal C}}
\def\cT{{\mathcal T}}

\def\N{{\mathcal N}}

\def\bep{\begin{proposition}}
\def\eep{\end{proposition}}

\newcommand{\mat}[4]{\left(\begin{array}{cc}#1 &#2  \\ #3 &#4 \end{array}\right)}

\usepackage{amsmath}
\usepackage{amssymb}
\usepackage{a4wide}
\usepackage{graphics}
\usepackage{epsfig}
\usepackage{enumerate}
\newcounter{hypo}

\def\C{{\mathbb C}}

\def\N{{\mathbb N}} 
\def\R{{\mathbb R}} 

\def\Z{{\mathbb Z}}

\def\CC{\mathcal {C}}
\def\CA{\mathcal {A}}
\def\CB{\mathcal {B}}
\def\CD{\mathcal {D}}
\def\CE{\mathcal {E}}

\def\CH{\mathcal {H}}

\def\CL{\mathcal {L}}
\def\CM{\mathcal {M}}

\def\CO{\mathcal {O}}
\def\CP{\mathcal {P}}

\def\CS{\mathcal {S}}
\def\CT{\mathcal {T}}
\def\CU{\mathcal {U}}
\def\CV{\mathcal {V}}

\def\one{{\mathchoice {\rm 1\mskip-4mu l} {\rm 1\mskip-4mu l} {\rm 1\mskip-4.5mu l} {\rm 1\mskip-5mu l}}}

\def\supp{\mathop {\rm supp}\nolimits}

\def\<{\langle}
\def\>{\rangle}

\def\slim{\mathop{\text{s--lim}}}

\def\bel{\begin{lemma}}

\makeatletter
 \@addtoreset{equation}{section}
 \makeatother

\def\4{\frac{1}{4}}

\def\cE{\mathcal{E}}

\def\Dom{{\rm Dom}}
\def\i{{\rm i}}
\def\h{\langle h \rangle}

\def\jh{\langle h\rangle}

\def\rr{{\mathbb R}}
\def\zz{{\mathbb Z}}
\def\cc{{\mathbb C}}
\def\nn{{\mathbb N}}

\def\slim{{\rm s-}\lim}

\def\bar{\overline}

\def\coinf{C_0^\infty}

\def\cV{{\mathcal V}}

\def\qed{$\Box$\medskip}
\def \p{ \partial}

\def\12{\frac{1}{2}}

\def\supp{{\rm supp}}

\def\e{{\rm e}}

\def\bbbone{{\mathchoice {\rm 1\mskip-4mu l} {\rm 1\mskip-4mu l}
{\rm 1\mskip-4.5mu l} {\rm 1\mskip-5mu l}}}
\def\one{\bbbone}
\def\cH{{\mathcal H}}

\def\C{{\cc}}
\def\cT{{\mathcal T}}

\def\N{\nn}

\def\bep{\begin{proposition}}
\def\eep{\end{proposition}}

\newcommand{\cb}{\mathcal{B}}

\newcommand{\ce}{\mathcal{E}}
\newcommand{\ch}{\mathcal{H}}

\def\rag{\rangle}
\def\lag{\langle}

\def\braket#1#2{\langle{#1}|{#2}\rangle}

\def\jap#1{\langle {#1} \rangle}

\def\slim{\mbox{\rm s-}\!\lim}

\def\nin{\notin}

\def\Ker{\mbox{\rm Ker\,}}

\def\supp{\mbox{\rm supp\! }}

\def\nin{\notin}

\def\qed{\hfill \raisebox{0.5ex}{\framebox[1.6ex]{
                                       \rule[0ex]{0ex}{0.3ex} }}}

\def\build#1_#2^#3{\mathrel{\mathop{\kern 0pt#1}\limits_{#2}^{#3}}}
\author{V. Georgescu}
\address{D\'epartement de Math\'ematiques,
Universit\'e de Cergy-Pontoise,
95302 Cergy-Pontoise Cedex, France }
\email{Vladimir.Georgescu@math.cnrs.fr}
\author{C. G\'erard}
\address{D\'epartement de Math\'ematiques, Universit\'e de Paris XI,
  91405 Orsay Cedex France} 
\email{christian.gerard@math.u-psud.fr}
\author{D. H\"{a}fner}
\address{Universit\'e de Grenoble 1, Institut Fourier, UMR 5582
  CNRS, BP 74 38402 Saint-Martin d'H\`eres France} 
\email{Dietrich.Hafner@ujf-grenoble.fr}
\begin{document}

\title[Asymptotic completeness for superradiant Klein-Gordon
equations]{Asymptotic completeness for superradiant Klein-Gordon equations and applications to the De Sitter Kerr metric}


\begin{abstract}
We show asymptotic completeness for a class of superradiant
  Klein-Gordon equations. Our results are applied to the Klein-Gordon
  equation on the De Sitter Kerr metric with small angular momentum of
  the black hole. For this equation we obtain
  asymptotic completeness for fixed angular momentum of the field.
\end{abstract}

\maketitle
%\tableofcontents
\section{Introduction}
\subsection{Introduction}
{\em Asymptotic completeness} is one of the fundamental properties one might
want to show for a Hamiltonian describing the dynamics of a physical system. Roughly speaking it states that
the Hamiltonian of the system is equivalent to a free Hamiltonian
for which the dynamics is well understood. The
dynamics that we want to understand behaves then at large times like
this free dynamics modulo possible eigenvalues. In the case when the Hamiltonian is
selfadjoint with respect to some suitable Hilbert space inner product
an enormous amount of literature has been dedicated to this
question. The question is much less studied in the case when the
Hamiltonian is not selfadjoint. This situation occurs for example for
the Klein-Gordon equation when the field is coupled to a (strong)
electric field. This system has been studied by Kako in a short range
case (see \cite{Ka}) and by C. G\'erard in the long range case (see
\cite{Ge12}). In this situation the Hamiltonian, although not
selfadjoint on a Hilbert space, is selfadjoint on a so called {\em Krein
space}. In a previous paper \cite{GGH1} we addressed the question of boundary values of the
resolvent for selfadjoint operators on Krein spaces. Applications to
propagation estimates for the Klein-Gordon equation are given in \cite{GGH2}. 

The Klein Gordon
equation can be written in a quite general setting in the form
\begin{equation}
\label{KG}
(\partial_t^2-2ik\partial_t+h)u=0, \ u:\ \rr\to \cH
\end{equation}
with selfadjoint operators $k$ and $h$. However if $h$ is not positive
the natural conserved energy for \eqref{KG}
\[\Vert \partial_tu\Vert^2+(hu|u)\]
is not positive and in general no positive conserved energy is
available. This happens in particular when the equation is associated to
a lorentzian manifold with no global time-like Killing vector
field. In this situation natural positive energies can grow in time, and we will loosely 
 speak about {\em superradiance}, the most famous example being  the (De Sitter) Kerr metric which
describes rotating black holes. 
T
his example doesn't enter into the
framework of our previous papers because the Hamiltonian can no longer
be realized as a selfadjoint operator on a Krein space whose 
topology is given by some natural positive (but not conserved) energy. The problem
comes from the fact that the operator $k$ has different ``limit
operators'' in the different ends of the manifold. In the one
dimensional case scattering results for this situation have been obtained
by Bachelot in \cite{Ba04}.

Asymptotic completeness for wave equations on Lorentzian manifolds
 has been studied for a long time since the works of Dimock
and Kay in the 1980's, see e.g.  \cite{DiKa86}. The
main motivation came from the {\em Hawking effect}. Such results are a
necessary step  to give mathematically rigorous descriptions of the Hawking effect, 
see Bachelot \cite{Ba99} and H\"afner \cite{Ha09}. The most complete
scattering results exist in the
Schwarzschild metric, see e.g.  Bachelot \cite{Ba94}. Asymptotic completeness has
also been shown on the Kerr metric for non superradiant modes of the
Klein-Gordon (see H\"afner \cite{Ha03})     
and for the Dirac equation (for which no superradiance occurs), see
H\"afner-Nicolas \cite{HaNi04}.  In this setting asymptotic
completeness can be understood as an existence and uniqueness result
for the characteristic Cauchy problem in energy space at null
infinity, see \cite{HaNi04} for details. As
far as we are aware asymptotic completeness has not been
addressed in the setting of superradiant equations on the (De Sitter)
Kerr background. Note however that scattering
results have been obtained by Dafermos, Holzegel, Rodnianski in the
difficult nonlinear setting of the Einstein equations supposing
exponential decay for the scattering data on the future event horizon
and on future null infinity, see \cite{DaHoRo13}. Also there has been enormous
progress in the last years on a somewhat related question which is the question of decay of the local energy for the wave
equation on the (De Sitter) Kerr metric. We mention in this context the
papers of Andersson-Blue \cite{AnBl09}, Dyatlov \cite{Dy11_02},
Dafermos-Rodnianski \cite{DaRo13_01},
Finster-Kamran-Smoller-Yau \cite{FKSY}, \cite{FKSYER}, Tataru-Tohaneanu \cite{TaTo11} and Vasy \cite{Va13} as well as
references therein for an overview. Let us make some comments on
the similarities and differences between asymptotic completeness results and decay of the
local energy: 
\begin{itemize}
\item[--] for a hyperbolic equation like the wave equation, the  essential
  ingredients for asymptotic completeness are {\em minimal velocity
  estimates} stating that the energy in cones inside the light cone
  goes to zero. No precise rate is required, but no loss of
  derivatives is permitted in the estimates. 
\item[--] {\em energy estimates} are necessary for asymptotic
  completeness. One needs to estimate the energy at null infinity by
  the energy on a $t=0$ slice and the energy at the time $t=0$ slice
  by the energy at null infinity. Leaving aside the question of loss
  of derivatives the first estimate can probably be deduced from the
  local energy decay estimates. For the other direction however a new
  argument is needed. Indeed most of the local energy decay estimates
  use the redshift, that becomes a blueshift in the inverse sens of
  time, see \cite{DaHoRo13}.
\item[--] the choice of coordinates has probably to be different. Whereas
  coordinates that extend smoothly across the event horizon are well adapted
  to the question of decay of local energy, they don't seem to be well
  adapted for showing asymptotic completeness results. 
\end{itemize}
We refer to \cite{Ni13} for a more detailed discussion on the link
between local energy decay and asymptotic completeness results. 

We show in this paper asymptotic completeness results for the
superradiant Klein-Gordon equation in a quite general setting. Our abstract Klein-Gordon operators
have to be understood as  operators acting on $\rr_{t}\times \Sigma$, where $\Sigma$ is a manifold with two
ends, both being asymptotically hyperbolic. 
We also impose the
existence of ``limit operators'' in the ends which can be realized as
selfadjoint operators on a Hilbert space. 
In this setting the non real
spectrum of the Hamiltonian consists of a finite number of complex
eigenvalues with finite multiplicity, we can define a smooth
functional calculus for the Hamiltonian and the truncated
resolvent can be extended meromorphically across  the real axis. We
show propagation estimates for initial data which in energy are
supported outside so called singular points. These singular points are
closely related to real resonances but unlike the selfadjoint setting
they may be singular points which are not real resonances. 

From the
propagation estimates it follows in particular that the evolution is
uniformly bounded for data supported in energy outside the singular
points. The same holds true for high energy data for which no
superradiance appears. 

We apply then our results  to the De Sitter Kerr
metric with small angular momentum. We show asymptotic completeness
for fixed but arbitrary angular momentum $n$ of the field. For angular momentum $n\neq 0$ the absence of real resonances
follows from the results of Dyatlov, see \cite{Dy11_01}. However some
additional work is required to show the absence of singular
points. In a subsequent paper we will discuss the De Sitter Kerr case
in more detail.     
\subsection{Abstract Klein-Gordon equation}
Let $(\CH, (\cdot|\cdot))$ be a Hilbert space, $(h,D(h))$ be a selfadjoint operator on
$\CH$ and $k\in \CB(\<h\>^{-1/2}\CH;\, \<h\>^{1/2}\CH)$ be another selfadjoint operator.  We consider the following abstract Klein-Gordon equation~:
\begin{eqnarray}
\label{WE}
\left\{\begin{array}{rcl} (\partial_t^2-2ik\partial_t+h)u&=&0,\\
u\vert_{t=0}&=&u_0,\\ \partial_tu\vert_{t=0}&=&u_1. \end{array}\right.
\end{eqnarray}
The operator $h$ will be in general not positive. As a consequence the
conserved {\it energy}
\begin{equation*}
\<u,u\>_0:=\Vert\partial_tu\Vert_{\CH}^2+(hu|u)_{\CH}
\end{equation*}
will in general not be positive. The hyperbolic character of the
equation is expressed by the condition
\begin{equation*}
h_0:=h+k^2\ge 0.
\end{equation*}
Another important conserved quantity is the {\it charge} which is given by
\begin{equation*}
q(u,v)=\<u|v\>=(u_0|v_1)+(u_1|v_0)
\end{equation*}
%Note that $\omega=i^{-1}q$ is a complex symplectic form
%(i.e. sesquilinear, non-degenerate, anti-hermitian). 
Setting $\Psi_0=(u_0,-iu_1)$ we can rewrite \eqref{WE} as a first order equation
\begin{eqnarray}
\label{W1O} \left\{\begin{array}{rcl} (\partial_t-iH)\Psi&=&0,\\
\Psi\vert_{t=0}&=&\Psi_0,\end{array}\right.
\end{eqnarray}
where  
\begin{eqnarray*}
H=\left(\begin{array}{cc} 0 & \one \\ h &
    2k \end{array}\right).
\end{eqnarray*}
Observe that 
\begin{equation}\label{eq:inverform}
(H- z)^{-1}=p( z)^{-1}\mat{ z-2k}{1}{h}{ z},
\end{equation}
where 
\begin{equation*}
p(z)=h+z(2k-z)=h_0-(k-z)^2\in B(\<h_0\>^{-1/2}\CH,\<h_0\>^{1/2}\CH),\,
z\in \C.
\end{equation*}
The map $z\mapsto p(z)$ is called a \emph{quadratic pencil} and plays an important role in
the study of $H$.
\subsection{Results for the Klein-Gordon equation on the De Sitter Kerr
  metric}
Let $(\CM,g)$ be  the De Sitter Kerr spacetime and $u$ a solution of
the Klein-Gordon equation
\begin{equation}
\label{WEDK}
\left\{\begin{array}{rcl} (\Box_g+m^2)u&=&0,\\
u|_{t=0}&=&u_0,\\
\partial_tu|_{t=0}&=&u_1,\end{array}\right.
\end{equation}
where $t$ is the Boyer-Lindquist time. We define the Cauchy surface
$\Sigma=\{t=0\}$. We write \eqref{WEDK} in the
form \eqref{WE} and associate the homogeneous energy space $\dot{\CE}$, which is
the completion of $C_0^{\infty}(\Sigma)\times C_0^{\infty}(\Sigma)$ for
the norm
\[\Vert u\Vert_{\dot{\CE}}=\Vert
u_1-ku_0\Vert_{\CH}^2+((h+k^2)u_0|u_0)_{\CH},\]
where $\CH=L^2(\Sigma; dVol)$ for a suitable measure $dVol$.
We can take $dVol=dxd\omega$, where $x$ is a Regge-Wheeler type coordinate if
we multiply $u$ with a suitable function. 
We then write the equation in the form \eqref{W1O} and
obtain  a Hamiltonian $\dot{H}$ (formally $\dot{H}=H$) acting
on $\dot{\CE}$. We also consider  the wave equations 
\begin{equation}
\label{WEpm}
\left\{\begin{array}{rcl} (\partial_t^2-2\Omega_{l/r}\partial_{\varphi}\partial_t-\partial_x^2+\Omega^2_{l/r}\partial_{\varphi}^2)u^{l/r}&=&0,\\
u^{l/r}|_{t=0}&=&u^{l/r}_0,\\
\partial_tu^{l/r}|_{t=0}&=&u^{l/r}_1,\end{array}\right.
\end{equation}
where $\Omega_{l/r}$ is the angular velocity of the  black hole resp. cosmological horizon. Let $H_{l/r}$, $\dot{\CE}_{l/r}$  be the
associated Hamiltonians and 
homogeneous energy spaces.  We denote by $\dot{\CE}^{n}$, $\dot{\CE}^{n}_{l/r}$ the subspaces of angular momentum $n$.  Let $i_{l/r}$ be smooth cut-off functions
equal to one at $\mp\infty$, equal to zero at $\pm\infty$. Let $a$ be the angular momentum per unit mass of the De Sitter Kerr spacetime. The main result in the De Sitter Kerr
setting is the following
\begin{theorem}
There exists $a_0>0$ such that for all $|a|<a_0$ and 
$n\in \Z\setminus\{0\}$ we have the following.
\begin{enumerate} 
\item There exists a dense subspace $\CE^{fin,n}_{l/r}$ of
  $\dot{\CE}^n_{l/r}$ such that for all $v_{l/r}\in \CE^{fin,n}_{l/r}$ the
  limits
\begin{equation*}
W^{l/r}_{\pm}v_{l/r}:=\lim_{t\rightarrow\pm \infty}e^{it\dot{H}}i_{l/r}e^{-it\dot{H}_{l/r}}v_{l/r}
\end{equation*}
exist. The operators $W^{l/r}_{\pm}$ extend to  bounded operators
$W^{l/r}_{\pm}\in \CB(\dot{\CE}^n_{l/r};\dot{\CE}^n)$.
\item For all $u\in\dot{\CE}^n$ the limit
\begin{equation*}
\Omega^{l/r}_{\pm}u:=\lim_{t\rightarrow\pm\infty}e^{it\dot{H}_{l/r}}i_{l/r}e^{-it\dot{H}}u
\end{equation*}
exists in $\dot{\CE}^n_{l/r}$.
\end{enumerate}
$i),\, ii)$ also hold for $n=0$ if the mass $m$ of the field is strictly positive.
\end{theorem}  
\subsection{Plan of the paper}
\begin{itemize}
\item[--] in Sect. \ref{SecBG} we collect some results on general
  abstract Klein-Gordon equations. Also
  we give some basic resolvent estimates. It turns out that already in
  this abstract setting superradiance can only occur at low
  frequencies as expressed in the estimates of Lemma
  \ref{pr:est}. This fact is already known for the Kerr metric, see
  e.g. Dafermos Rodnianski \cite{DaRo13_01}, but not in this spectral formulation. We also study {\em gauge
  transformations} in Sect. \ref{SecGT}. These gauge transformations 
  correspond to choices of different Killing fields in a more geometric
  language.
\item[--] in Sect. \ref{SecME} we recall some facts on meromorphic
  Fredholm theory. We show that a meromorphic extension of the
  truncated resolvent of $h$ gives a meromorphic extension of the
  weighted resolvent of $H$.
\item[--] in Sect. \ref{SecTE} we describe the abstract setting for a
  Klein-Gordon operator on a manifold with two ends. Our assumptions
  assure that the asymptotic Hamiltonians in the ends are selfadjoint.
  Gluing the resolvents of the asymptotic operators together gives the resolvent
  for $H$ using the Fredholm theory of Sect. \ref{SecME}. We obtain
  in this way resolvent estimates for the Hamiltonian $H$ which are
  sufficient to construct  a smooth functional calculus for $H$.
\item[--] in Sect. \ref{SecAC1} we prove  propagation estimates which
  are needed for the proof of the asymptotic completeness result. We
  also introduce the notion of {\em singular points}. Singular points are
  obstacles for uniform boundedness of the evolution and therefore
  also for asymptotic completeness. A  useful criterion for the absence of
  singular points is given. 
\item[--] in Sect. \ref{secbound1} we show uniform boundedness of the evolution for data which is
  spectrally supported outside the singular points.
\item[--] asymptotic completeness is shown in the abstract setting in
  Sect. \ref{secAsympc1}. The scattering space corresponds to data which are  supported in energy
  outside the singular points. 
\item[--] in Sect. \ref{SecGS} we introduce the geometric setting. Our
  operators fulfill the hypotheses of the abstract setting. We
  obtain meromorphic extensions of the weighted resolvents using a
  result of Mazzeo-Melrose \cite{MM87}.
\item[--] We apply our general result to the geometric setting in Section
  \ref{secACgeo} and obtain an asymptotic completeness result in
  the geometric setting. 
\item[--] in Sect. \ref{SecKGDK} we describe the Klein-Gordon equation on
  the De Sitter Kerr metric.
\item[--] in Sect. \ref{SecAC3} the main results are formulated in the
  De Sitter Kerr setting. Two types of results are established: comparison to spherically
  symmetric asymptotic dynamics on the same energy space and
  comparison to asymptotic profiles. These asymptotic profiles give rise
  to energy spaces which are bigger than the original ones. The wave
  operators can therefore only be defined as limits on dense
  subspaces. They then extend by continuity to the whole energy space
  for the profiles. Inverse wave operators exist on the whole energy
  space as limits.
\item[--] The proofs of the theorems in the De Sitter Kerr setting are
  given in Sect. \ref{prDSKerr}. We apply our earlier abstract
  theorems. To obtain the meromorphic extensions of the different
  truncated resolvents it is crucial that the cosmological constant is
  strictly positive. The absence of real resonances and complex eigenvalues
  follows from the work of Dyatlov \cite{Dy11_01} for a compactly
  supported cutoff resolvent. An
  hypo-ellipticity argument  enables us to use an exponential weight. Our
  earlier results in \cite{GGH2} and the general criterion in Section
  \ref{SecAC1} enable us to show the absence of singular points.  
\end{itemize}
As usual for asymptotic completeness results and to simplify the
exposition we consider only the limit $t\rightarrow \infty$. All the
results in this paper also hold in the limit $t\rightarrow -\infty$
and the proofs are the same. 
\begin{center}

\vspace{1cm}

{\bf \large Acknowledgements}
\end{center}
The authors thank J.-F. Bony for fruitful discussions. This work was
partially supported by the ANR project AARG. DH thanks the
MSRI in Berkeley for hospitality during his stay in the
fall of 2013.

\section{Background on abstract Klein-Gordon operators}
\label{SecBG}

\subsection{Notations}
\begin{itemize}
\item[--]  if $X,Y$ are sets and $f:X\to Y$,  we write $f:X\tilde\to Y$  if $f$ is bijective. We use the same notation if $X, Y$ are topological spaces  if $f$ is an isomorphism.

 \item[--] if $\cH$ is a Banach space we denote $\cH^*$ its adjoint space, the
set of continuous anti-linear functionals on $\cH$ equipped with the
natural Banach space structure.  Thus the canonical anti-duality $\lag
u, w \rag$, where $u\in\cH$ and $w\in\cH^{*}$, is anti-linear in $u$
and linear in $w$. In general we denote by $\braket{\cdot}{\cdot}$
hermitian forms on $\cH$, again anti-linear in the first argument and
linear in the second one, but if $\cH$ is a Hilbert space its scalar
product is denoted by $(\cdot|\cdot)$.
\item[--] $\CB(\CH)$ is the space  of bounded operators on $\CH$ and
$\CB_{\infty}(\CH)$ the subspace of compact operators.

\item[--] if $S$ is a closed densely defined operator on a Banach space, then
$D(S),\, \rho(S)$, $\sigma(S)$ are its domain, resolvent set and
spectrum.  We use the notation $\jap{S}=(1+S^2)^{1/2}$ if $S$ is an
operator for which this expression has a meaning, in particular if $S$
is a real number.

\item[--] if $S$ is a self-adjoint operator on a Hilbert space then $S>0$ means
$S\geq0$ and $\Ker S=\{0\}$. 

\item[--] if $A, B$ are two selfadjoint operators or real numbers, (possibly depending on some parameters), we write $A\lesssim B$ if $A\leq C B$ for some  constant $C>0$, (uniformly with respect to the parameters).

\end{itemize}

\subsection{Scales of Hilbert spaces}

Let $\cH$ be a Hilbert space identified with its adjoint space
$\cH^{*}=\cH$ via the Riesz isomorphism. If $h$ is a selfadjoint
operator on $\cH$ we associate to it the {\em non-homogeneous
  Sobolev spaces}
\[
 \langle h\rangle^{-s}\cH:= \Dom |h|^{s}, \   \langle
 h\rangle^{s}\cH:= ( \langle h\rangle^{-s}\cH)^{*}, \ s\geq 0. 
\]
The spaces $ \langle h\rangle^{-s}\cH$ are equipped with the graph
norm $\| \langle h\rangle^{s}u\|$.  We keep the notation
\[
(u| v), \ u\in \langle h\rangle^{-s}\cH, \ v\in \langle
h\rangle^{s}\cH,
\]
to denote the duality bracket between $\jh^{-s}\cH$ and
$\jh^{s}\cH$.  \def\jh{\langle h\rangle}

If $\Ker h=\{0\}$ then we also define the {\em homogeneous Sobolev
  space} $|h|^{s}\cH$ equal to the completion of $\Dom |h|^{-s}$ for
the norm $\||h|^{-s}u\|$.  The notation $\langle h\rangle^{s}\cH$ or
$|h|^{s}\cH$ is convenient but somewhat ambiguous because usually
$a\cH$ denotes the image of $\cH$ under the linear operator $a$. We
refer to \cite[Subsect. 2.1]{GGH2} for a complete discussion of this
question. 

Let us mention some properties of the scales of spaces defined above:
\[
\begin{array}{l}
 \langle h\rangle^{-s}\cH \subset \langle h\rangle^{-t}\cH, 
\hbox{  if }t\leq s, \  \langle h\rangle^{-s}\cH\subset |h|^{-s}\cH, 
|h|^{s}\cH\subset \h^{s}\cH \hbox{ if }s\geq 0, \\[2mm]
\h^{0}\cH= |h|^{0}\cH= \cH, \ \h^{s}\cH= (\h^{-s}\cH)^{*}, 
\ |h|^{s}\cH= (|h|^{-s}\cH)^{*},\\[2mm]
0\in \rho(h)\Leftrightarrow \ \h^{s}\cH= |h|^{s}\cH
\hbox{ for some } s\neq 0 
\Leftrightarrow \h^{s}\cH= |h|^{s}\cH\hbox{ for all }s, \\[2mm]
\text{the operator } |h|^{s} \text{ is unitary from } 
|h|^{-t}\cH \text{ to } |h|^{s-t}\cH \text{ for all } s,t\in \rr. 
\end{array}
\]

\subsection{Quadratic pencils}
\label{qp}
Let $\CH$ be a Hilbert space, $h$ a selfadjoint operator on $\CH$, and
$k\in B(\CH)$ a \emph{bounded} symmetric operator. Then $h_0=h+k^2$ is
a self-adjoint operator on $\ch$ with the same domain as $h$ hence
$\<h\>^s\CH=\<h_0\>^s\cH$  for  $s\in[-1,1]$. Thus the operators $h$ and
$h_0$ define the same scale of Sobolev spaces for $s\in[-1,1]$ that we
shall denote:
\[
\ch^s := \< h\>^{-s}\cH = \< h_0\>^{-s}\cH
\quad\text{if } -1 \leq s \leq 1.
\]
We define the {\it quadratic pencil}
\begin{equation*}
p(z)=h+z(2k-z)=h_0-(k-z)^2, \ z\in \cc.
\end{equation*}
 A priori these are operators on $\cH$
with domain $\cH^1$ and we clearly have $p(z)^*=p(\bar{z})$ as
operators on $\cH$. Moreover, for each $s\in[0,1]$ they extend to
operators in $B(\cH^s,\ch^{s-1})$ and, for example, the relation
$p(z)^*=p(\bar{z})$ holds as operators $\cH^{\12} \to \cH^{-\12}$.
From this it is easy to deduce the following lemma (see \cite[Lemma
8.1]{GGH1} for a more general result).

\begin{lemma}\label{lm:kg}
  The following conditions are equivalent:
\[
  \begin{array}{ll}
 (1) \ p( z):\cH^1\tilde\to\cH
& (2) \ p(\bar{ z}):\cH^1\tilde\to\cH\\[2mm]
 (3) \ p( z):\cH^{-\12}\tilde\to\cH^{\12} \ 
& (4) \ p(\bar{ z}):\cH^{-\12}\tilde\to\cH^{\12}\\[2mm]
 (5) \ p( z):\cH\tilde\to\cH^{-1} \
& (6) \ p(\bar{ z}):\cH\tilde\to\cH^{-1}.
\end{array}
\]
In particular, the set 
\begin{equation}\label{eq:rhohk}
\rho(h,k):=\{z\in\cc\mid p(z):\cH^{-\12}\tilde\to\cH^{\12}\}
=\{z\in\cc\mid p(z):\cH^{1}\tilde\to\cH\}
\end{equation}
is invariant under conjugation.
\end{lemma}

The next result is easy to prove in the present context; one can find
a proof under more general conditions in \cite[Lemma 8.2]{GGH1}.

\begin{proposition}\label{pr2.2}
  If $h$ is bounded below then there exists $c_{0}>0$ such that
 \[
\{ z\ : |{\rm Im} z| > |{\rm Re}z| + c_{0}\}\subset \rho(h,k).
\]
\end{proposition}

We shall prove now some estimates on $p(z)^{-1}$ for $z\in\rho(h,k)$. Note that
they are valid under much more general assumptions on $h$ and $k$ then those
imposed in this paper.

\begin{lemma}\label{lm:est1}
Assume that $h+c\geq 0$ for some $0\leq c$ and let $b>1$. If $z\in\rho(h,k)$ then 
\begin{equation}\label{eq:est1}
\|p(z)^{-1}\|\leq\frac{b}{|z\,{\rm Im}z|} \quad\text{if}\quad 
|z|^2\geq \frac{bc}{b-1}.
\end{equation}
\end{lemma}

\proof

We abbreviate $p=p(z)$ and $\mu={\rm Im z}$. The main point is the identity
\begin{equation}\label{eq:id}
{\rm Im}\frac{z}{\mu p} =\frac{1}{p^*}(h+|z|^2)\frac{1}{p} 
\end{equation}
which is rather obvious:
\[
\frac{z}{p}-\frac{\bar{z}}{p^*}= \frac{1}{p^*}\big( zp^*-\bar{z}p \big)\frac{1}{p} 
=(z-\bar{z})\frac{1}{p^*}\big( h+|z|^2 \big)\frac{1}{p}   .
\]
Then the relation \eqref{eq:id} gives
\( (|z|^2-c) \frac{1}{p^*}\frac{1}{p} \leq {\rm Im}\frac{z}{\mu p} \) hence
\[
|\mu| (|z|^2-c) \|p^{-1}u\|^2 \leq \big| {\rm Im} (u|zp^{-1}u) \big| \leq
|z| \|u\| \|p^{-1}u\|
\]
hence $|\mu| (|z|^2-c) \|p^{-1}u\|^2 \leq |z| \|u\|$ which is more than required.
\qed

\begin{lemma}\label{lm:est2}
Assume $h_0\geq 0$ and $k^2 \leq \alpha h_0+\beta$ with $\alpha<1$.
Then $h$ is bounded from below and if $h+c\geq 0$ and $\varepsilon >0$ then there is 
a number $C$ such that for $z\in\rho(h,k)$ and $|z|\geq \sqrt{c}+\varepsilon$
\begin{equation}\label{eq:est2}
\|h_0^{\frac{1}{2}} p(z)^{-1}\|\leq C |{\rm Im}z|^{-1}.
\end{equation}
\end{lemma}

\proof

If we set $q=p^{-1}$ then \eqref{eq:id} implies 
$q^* \big( (1-\alpha) h_0 +|z|^2 \big)  q  \leq \beta q^* q + \mu^{-1} {\rm Im}\,z q $
hence
\[
(1-\alpha) \|h_0^{\frac{1}{2}} p^{-1} u \|^2 \leq \beta \|p^{-1}u\|^2 +| z/\mu | \|u\| \|p^{-1}u\| \leq
{\textstyle\frac{\beta b^2}{|z\mu|^2}} \|u\|^2 + {\textstyle\frac{b}{\mu^2}} \|u\|^2
\]
if $|z|^2\geq \frac{bc}{b-1}$. This estimate is more precise than \eqref{eq:est2}. Note that 
we may take $c=\beta$. 
\qed

\subsection{Spaces}

The operators $h,k,h_0$ and the spaces $\ch^s$ are as in the preceding
subsection, in particular $k$ is a bounded operator in $\cH$, but now
we shall impose much stronger conditions. 

From now on \emph{we always assume} that the following  condition is
satisfied:
\begin{equation}
\tag{A1}  \label{Hyp1} h_0:=h+k^2 > 0 . 
\end{equation}
Then the homogeneous scale $h_0^s\cH$ associated to $h_0$ is well
defined. Note that, if $h$ is injective, the spaces $|h|^{-1}\cH$ and
$h_0^{-1}\cH$ are quite different in general, although
$\<h\>^{-1}\cH=\<h_0\>^{-1}\cH$. 

We shall require $k$ to behave well with respect to the
homogeneous $h_0$-scale:
\begin{equation}
\tag{A2} \label{A2}
\left\{ \begin{array}{l}  
k\in \CB(h_0^{-1/2}\CH);   \\[2mm]         %h_0^{1/2}kh_0^{-1/2}\in \CB(\CH), 
\text{if } z\nin \R \text{ then }  
(k-z)^{-1}\in \CB(h_0^{-1/2}\CH)
\text{ and} \\[1mm] 
\Vert
(k-z)^{-1}\Vert_{\CB(h_0^{-1/2}\CH)}\lesssim |{\rm Im} z|^{-n} 
\text{ for some } n>0;    \\[3mm]
\exists m>0 \text{ such that if } |z|\ge m\Vert
  k\Vert_{\CB(\CH)} \text{ then}\\[1mm] 
\Vert (k-z)^{-1}\Vert_{\CB(h_0^{-1/2}\CH)}\lesssim
\left| |z|-\Vert k\Vert_{\CB(\CH)} \right|^{-1}    .
\end{array}\right.
\end{equation}
The next comments will clarify the meaning of these conditions.
Recall that $h_0^{-s}\cH$ and $h_0^{s}\cH$ are adjoints to each other
but they are not comparable and neither are they comparable with $\cH$. 
The first assumption 
%could be formally written $h_0^{1/2}kh_0^{-1/2}\in \CB(\CH)$  and 
says that the operator $k$ leaves $D(h_0^{1/2})$ invariant and that
its restriction to $D(h_0^{1/2})$ extends to a bounded operator, say $\bar{k}$, in
$h_0^{-1/2}\cH$. The rest of the assumption concerns the resolvent of
$\bar{k}$ in this space. In order not to overcharge the notation we
kept the notation $k$ for $\bar{k}$. 

% ?????????????????????????
%Alternatively, one may express \eqref{A2} by requiring bounds on
%expressions of the form $(h_0+\varepsilon)^{1/2} S
%(h_0+\varepsilon)^{-1/2}$ uniform in $\varepsilon$.
%??????????????????????????

The preceding assumptions allow us to get a new estimate on the quadratic pencil $p$.

\begin{lemma}\label{lm:est3}
Under the conditions \eqref{Hyp1} and \eqref{A2}, there are numbers
$C,M>0$ such that 
\begin{equation} \label{eq:est3}
\|h_0^{\frac{1}{2}} p(z)^{-1}(k-z)u\| \leq C |{\rm Im} z|^{-1} 
\|h_0^{\frac{1}{2}} u \| \quad\text{if } 
|z| \geq M \|k\|_{\cb(\ch)}  . %\text{ and } {\rm Im} z\neq0.
\end{equation}
\end{lemma}

\proof

We abbreviate $p=p(z)$ and $m=z-k$, so that $m^*=\bar{z}-k$ and $p=h_0-m^2$.  
We have:
\[
\frac{z}{m} h_0 \frac{1}{p} m - m^* \frac{1}{p^*} h_0 \frac{\bar{z}}{m^*}=
m^* \frac{1}{p^*} \left( 
p^* \frac{z}{|m|^2} h_0 -h_0 \frac{\bar{z}}{|m|^2} p
\right) \frac{1}{p} m .
\]
If we replace here $p$ by $h_0-m^2$ and then develop and rearrange the terms,  
we get:
\[
m^* \frac{1}{p^*} \left(
(z-\bar{z}) h_0 \frac{1}{|m|^2}h_0 + h_0 \frac{\bar{z}m}{m^*} -
\frac{zm^*}{m} h_0 
\right) \frac{1}{p} m .
\]
Since $\frac{m^*}{m}=1-\frac{z-\bar{z}}{m}$ and $\frac{z}{m}=1+\frac{k}{m}$,
a simple computation gives:
\[
h_0 \frac{\bar{z}m}{m^*} - \frac{zm^*}{m} h_0  =
(z-\bar{z}) \left(
h_0 + h_0 \frac{k}{m^*} + \frac{k}{m} h_0 
\right) .
\]
To conclude, we have proved, with $\mu={\rm Im} z$, 
\[
\frac{1}{\mu} {\rm Im} \left( \frac{z}{m}h_0\frac{1}{p}m \right) =
m^* \frac{1}{p^*} \left(
h_0 + 2{\rm Re} \big(\frac{k}{m} h_0 \big) +h_0 \frac{1}{|m|^2}h_0
\right) \frac{1}{p} m .
\]
We may also write this as follows:
\[
\frac{1}{\mu} {\rm Im} (u | zm^{-1} h_0 p^{-1}mu) =
\|h_0^{\frac{1}{2}} p^{-1} m u \|^2 +
2{\rm Re} (p^{-1}mu | km^{-1} h_0 p^{-1} mu ) 
+ \|m^{-1} h_0 p^{-1} m u \|^2 .
\]
Since the last term is positive, we get 
\begin{align*}{rl}
\|h_0^{\frac{1}{2}} p^{-1} m u \|^2   \leq& 
\frac{1}{\mu} {\rm Im} (u | zm^{-1} h_0 p^{-1}mu) 
- 2{\rm Re} (p^{-1}mu | km^{-1} h_0 p^{-1} mu )   \\
 =&
\frac{1}{\mu} {\rm Im} (h_0^{\frac{1}{2}}u | 
h_0^{-\frac{1}{2}}zm^{-1} h_0^{\frac{1}{2}} 
\cdot h_0^{\frac{1}{2}} p^{-1}mu)   \\ 
& - 2{\rm Re} (h_0^{\frac{1}{2}} p^{-1}mu | 
h_0^{-\frac{1}{2}}km^{-1} h_0^{\frac{1}{2}} 
\cdot h_0^{\frac{1}{2}} p^{-1} mu ) .
\end{align*}
Set $a(z)=\|h_0^{-\frac{1}{2}}zm^{-1} h_0^{\frac{1}{2}}\|$ and 
$b(z)=2\|h_0^{-\frac{1}{2}}km^{-1} h_0^{\frac{1}{2}}\|$. Since $zm^{-1}=1+km^{-1}$,
assumption \eqref{A2} implies the boundedness of  $a(z)$ for large $z$
and $b(z)\to0$ if $z\to\infty$. Finally, we have
\[
\big( 1-b(z) \big) \|h_0^{\frac{1}{2}} p^{-1} m u \| \leq
a(z)|\mu|^{-1} \|h_0^{\frac{1}{2}} u\|, 
\]
which proves the lemma.
\qed

For easier reference later on, we summarize in the next proposition a particular case
of the estimates we got in Lemmas \ref{lm:est1}, \ref{lm:est2} and \ref{lm:est3}.

\begin{proposition}\label{pr:est}
Assume that the conditions \eqref{Hyp1} and \eqref{A2} are satisfied and let
$\varepsilon>0$. Then there are numbers $C,M>0$ such that:

 \begin{align}
\label{basicp1}
\|p^{-1}(z)\|  & \leq C |z|^{-1} |{\rm Im}z|^{-1}
 \text{ if } |z| \ge(1+\epsilon)\Vert k\Vert_{\CB(\CH)} ,
\\
\label{basicp2}
 \|h_0^{1/2} p^{-1}(z)\| &\leq C |{\rm Im}z|^{-1} 
 \text{ if } |z| \ge(1+\epsilon)\Vert k\Vert_{\CB(\CH)} , 
\\
\label{basicp3}
 \Vert h_0^{1/2}p^{-1}(z)(k-z)u\Vert  & \leq C |{\rm Im}z|^{-1} \Vert h_0^{1/2}u\Vert 
 \text{ if } |z|  \ge M\Vert k\Vert_{\CB(\CH)} .
\end{align}
\end{proposition}

\bigskip

Sometimes it is useful to consider also the homogeneous $h$-scale. The
following assumption will be convenient in such situations. 
\begin{equation}
\tag{A3} \label{A3}  
h\ge c k^2 \text{ for some real } c>0  .
\end{equation}
This means that $h$ is positive and $\Vert ku\Vert\leq c^{-1/2} \Vert h^{1/2}u\Vert$ 
for all $u\in D(h^{1/2}) $.

\begin{lemma}\label{lm:A3}
  If $c>0$ is real then $h\geq ck^2 \Leftrightarrow h\geq
  \frac{c}{1+c} h_0$. Thus \eqref{A3} is satisfied if and only if
  there is $b>0$ real such that $bh_0\leq h\leq h_0$. If \eqref{Hyp1}
  and \eqref{A3} hold then $h>0$.
\end{lemma}
\proof Note that $h=h_0-k^2\leq h_0$. On the other hand, if $c>0$ is
real then we clearly have
\begin{equation}\label{eq:ch0}
h\geq ck^2 \Leftrightarrow h_0\geq (1+c)k^2
\Leftrightarrow h\geq \frac{c}{1+c} h_0
\end{equation}
and this implies the assertions of the lemma. \qed

\begin{corollary}\label{co:A3}
If \eqref{Hyp1} and \eqref{A3} are satisfied then $h_0^s\CH=h^s\CH$
for all $-1/2\leq s\leq 1/2$. 
\end{corollary} 
\proof Indeed, from $bh_0\leq h\leq h_0$ we get 
$b^\theta h_0^\theta\leq h^\theta\leq h_0^\theta$ if
$0\leq\theta\leq1$. \qed

\subsubsection{Inhomogeneous energy spaces}

The {\it inhomogeneous} energy space is the vector space 
\[
\CE=\CH^{1/2}\oplus\CH,
\]
equipped with the natural direct sum topology which makes it a
Hilbertizable space. For consistency with the norm that we introduce later in
the homogeneous case we take 
\begin{equation}\label{eq:norm}
\Vert(\begin{smallmatrix}u_0\\u_1\end{smallmatrix})\Vert_{\mathcal E}^2
=\Vert u_1-ku_0\Vert^2+((h_0+1)u_0|u_0)
\end{equation}
but of course we could replace here $k$ by zero. It is convenient, as
explained in \cite{GGH1}, to identify its adjoint space $\CE^*$ with
$\CH\oplus\CH^{-1/2}$
the anti-duality being given by 
\begin{equation}
\label{charge}
\<u,v\>=(u_0|v_1-kv_0)+(u_1-ku_0|v_0) \quad \text{if}\quad
u=\left(\begin{array}{c} u_0 \\ u_1\end{array} \right)\in \CE,\quad
v=\left(\begin{array}{c} v_0 \\ v_1\end{array} \right)\in \CE^*
\end{equation}
usually called  the {\it charge}.
Observe that $\CE\subset\CE^*$ continuously and densely.  We identify
$\CE^{**}=\CE$ as in the Hilbert space case by setting
$\<v,u\>=\overline{\<u,v\>}$.

In what follows it will often be convenient to use the operator
\begin{equation}
\label{phik}
\Phi(k)=\left(\begin{array}{cc} 1 & 0 \\ k & 1 \end{array}\right). 
%\quad\text{hence}\quad
%\Phi(k) \left(\begin{array}{c} u_0 \\ u_1\end{array} \right) =
%\left(\begin{array}{c} u_0 \\ u_1 +k u_0\end{array} \right).
\end{equation}
Note that $\Phi(k):\CE\tilde\to\CE$ and $\Phi(k):\CE^*\tilde\to\CE^*$ 
with $\Phi^{-1}(k)=\Phi(-k)$ and we may write
\begin{equation}
\CE=\Phi(k)(\<h_0\>^{-1/2}\CH\oplus\CH) \quad\text{and}\quad
    \CE^*=\Phi(k)(\CH\oplus\<h_0\>^{1/2}\CH), 
\end{equation}
which explains the choice of the norm in \eqref{eq:norm} and makes the 
connection with \eqref{eq:dote}.

\subsubsection{Homogeneous energy spaces}

We define the {\it homogeneous energy space}  $\dot\CE$ as the completion of 
$\CE$ under the norm  defined by:
\begin{equation} \label{eq:hnorm}
\Vert(\begin{smallmatrix}u_0\\u_1\end{smallmatrix})\Vert_{\dot{\CE}}^2
:=\Vert u_1-ku_0\Vert^2+(h_0u_0|u_0).
\end{equation}
The completion is the set of couples 
$u=(\begin{smallmatrix}u_0\\u_1\end{smallmatrix})$ with
$u_0\in h_0^{-1/2}\CH, u_1\in (1+h_0^{-1/2})\CH$, and such that  $u_1 - ku_0\in\CH$.
We shall realize  its adjoint space $\dot\CE^*$ with the help of the charge anti-duality 
defined as in \eqref{charge}.  Observe that, since $k\in\CB(h_0^{-1/2}\CH)$ by  
\eqref{A2}, we also have  
\begin{equation}\label{eq:dote}
\dot{\CE}=\Phi(k)(h_0^{-1/2}\CH\oplus\CH),\quad
\dot{\CE}^*=\Phi(k)(\CH\oplus h_0^{1/2}\CH).
\end{equation}
If the assumption \eqref{A3} is satisfied then we also define the
$h$-homogeneous energy spaces
\begin{equation*}
\dot{\tilde{\CE}}=h^{-1/2}\CH\oplus\CH,\quad
\dot{\tilde{\CE}}^*=\CH\oplus h^{1/2}\CH.
\end{equation*}
Here the direct sums are in the Hilbert space sense and 
the identification of $\dot{\tilde{\CE}}^*$ with the space adjoint to 
$\dot{\tilde{\CE}}$ is done with the help of the sesquilinear form 
defined as in \eqref{charge} but with $k=0$. 

\begin{lemma}
Assume \eqref{Hyp1}-\eqref{A3}. Then $\dot{\CE}=\dot{\tilde{\CE}}$ and
the norms $\Vert\cdot\Vert_{\dot{\CE}}$ and
$\Vert\cdot\Vert_{\dot{\tilde{\CE}}}$ are equivalent. 
\end{lemma}
\proof
We have to prove that 
$\Vert u_1-ku_0 \Vert^2+(h_0u_0|u_0)\simeq \Vert u_1\Vert^2+(hu_0|u_0)$.
But this is obvious by \eqref{A3} and Lemma \ref{lm:A3}.
\qed

\subsubsection{Conserved quantities}
On $\CE$ we introduce for $\ell\in \R$ the following sesquilinear
forms :
\begin{equation}\label{eq:ell}
\<u|v\>_{\ell}:=(u_1-\ell u_0|v_1-\ell v_0)+(p(\ell)u_0|v_0),
\end{equation}
where $p(\ell)=h_0-(k-\ell)^2.$ We have seen in the introduction that
these forms are formally conserved by the evolution, but in general
not positive definite. 
\begin{lemma}
\label{lem2.3.3}
For all $\ell\in \R$, $\<\cdot|\cdot\>_{\ell}$ is continuous with
respect to the norm $\Vert \cdot \Vert_{\CE}$.
\end{lemma}
\proof
Due to the polarization identity it suffices to show 
$|\<u|u\>_\ell | \lesssim \Vert u\Vert_{\CE}^2$ 
for all $u\in\CE$. Since
\begin{equation}\label{eq:polar}
\<u|u\>_\ell = \|u_1-\ell u_0 \|^2 + \|h_0^{1/2}u_0\|^2 - \|(k-\ell)u_0\|^2 
\end{equation}
and $k$   is bounded, this is obvious.
\qed

\begin{lemma}
\label{lem2.3.4}
For all $\ell\in \R,\, \<\cdot|\cdot\>_{\ell}$ is continuous with
respect to the norm $\Vert \cdot \Vert_{\dot{\CE}}$ if and only if
\begin{equation}
\label{2.3.4}
h_0\gtrsim (k-\ell)^2 .
\end{equation}
\end{lemma}
\proof
We have $\Vert u \Vert_{\dot\CE}^2 = \|u_1- k u_0 \|^2 + \|h_0^{1/2}u_0\|^2$ 
and we have to decide when $|\<u|u\>_\ell | \lesssim \Vert u\Vert_{\dot\CE}^2$.
By \eqref{eq:polar} this holds  if and only if 
$\big| \|u_1-\ell u_0 \|^2 - \|(k-\ell)u_0\|^2 \big| \lesssim \|u_1- k u_0 \|^2 + \|h_0^{1/2}u_0\|^2$.
If this holds and we take $u_1=ku_0$ we get \eqref{2.3.4}. The converse is obvious.
\qed

\subsection{Energy Klein-Gordon operators}
Let
\begin{equation} \label{eq:hathk}
\hat{H}=\left(\begin{array}{cc} 0 & 1 \\ h & 2k \end{array} \right)=\Phi(k)\hat{K}\Phi^{-1}(k)
\quad\text{where}\quad
\hat{K} =\left(\begin{array}{cc} k & 1 \\ h_0 & k \end{array} \right) .
\end{equation}
The {\em energy Klein-Gordon operators }will be various realizations of $\hat{H}$. 
The operator $\hat{K}$ is a {\em charge Klein-Gordon operator} and will only play a technical role.

\subsubsection{Klein-Gordon operator on the inhomogeneous energy  space}
The {\em inhomogeneous Klein-Gordon operator} is the operator $H$
induced by $\hat{H}$ on $\CE$. This means that its domain is
\begin{equation*}
D(H):=\{u \in \CE;\, \hat{H} u\in \CE\}  = \CH^{1} \oplus \CH^{1/2},
\end{equation*}
and for $u \in D(H)$ we have $Hu=\hat{H} u$. For the second equality above, see
\cite[Sect. 5.2]{GGH2}.  We also recall \cite[Proposition 5.3]{GGH2}:

\begin{proposition}
\label{propspecH}
Assume \eqref{Hyp1} and \eqref{A2}.
\begin{enumerate}
\item[--] One has $\rho(H)=\rho(h,k)$.
\item[--] In particular, if $\rho(h,k)\neq \emptyset$ then $H$ is a closed
  densely defined operator in $\CE$ and its spectrum is invariant
  under complex conjugation.
\item[--] If $z\in \rho(h,k)$, then
\begin{equation}
\label{resH}
R(z):=(H-z)^{-1}=p^{-1}(z)\left(\begin{array}{cc} z-2k & 1 \\ h &
    z \end{array}\right).
\end{equation}
\end{enumerate}
\end{proposition}

We may similarly define the operator $K$ induced by $\hat K$ in $\CE$ and one may
easily check, under the same conditions \eqref{Hyp1} and \eqref{A2}, that 
$\Phi(k): \CH^{1} \oplus \CH^{1/2} \tilde\to \CH^{1} \oplus \CH^{1/2}$ with inverse $\Phi(-k)$,
hence $H$ and $K$ have the same domain and $H=\Phi(k)K\Phi(-k)$. This implies
\begin{equation}
\label{resK}
(K-z)^{-1}=\left(\begin{array}{cc}
    p^{-1}(z)(z-k) & p^{-1}(z) \\ 1 + (z-k)p^{-1}(z)(z-k) &  (z-k)p^{-1}(z) 
\end{array}\right)     .
\end{equation}

\subsubsection{Klein-Gordon operator on the homogeneous energy
  space}

The {\it homogeneous Klein-Gordon operator} is the operator
$\dot{H}$ induced by $\hat{H}$ on $\dot{\CE}$. This means that its
domain is
\begin{equation*}
D(\dot{H})=\{u\in \dot{\CE};\, \hat{H}u\in \dot{\CE}\}
\end{equation*}
and for $u\in D(\dot{H})$ we have $\dot{H}u=\hat{H}u$. 
The proofs will involve the homogeneous  operator $\dot{K}$ associated 
to the auxiliary operator $\hat{K}$ and acting in the space $h_0^{-1/2}\CH \oplus \cH$ 
with domain
\[
D(\dot{K})=\{v\in h_0^{-1/2}\CH \oplus \cH ;\, \hat{K}v\in h_0^{-1/2}\CH \oplus \cH \}.
\]
From the relations \eqref{eq:dote} and \eqref{eq:hathk} we see that $\Phi(k)$ 
induces an isomorphism of $h_0^{-1/2}\CH\oplus\CH$ with $\dot\CE$ whose inverse
is $\Phi(-k)$. Clearly then $\dot{H}=\Phi(k)\dot{K}\Phi(-k)$.  

\begin{lemma}\label{lm:DH}
Under the conditions \eqref{Hyp1} and \eqref{A2} we have:
\begin{equation*}
D(\dot{H})=\Phi(k) 
\left( \big(h_0^{-1/2}\CH\cap h_0^{-1}\CH \big)\oplus\CH^{1/2} \right).
\end{equation*}
\end{lemma} 

\proof 

From the preceding comments we see that the assertion of the lemma
is equivalent to
\begin{equation}\label{eq:DK}
D(\dot{K})=
 \big(h_0^{-1/2}\CH\cap h_0^{-1}\CH \big)\oplus\CH^{1/2}  
\end{equation}
Since $\hat{K}v= (\begin{smallmatrix} kv_0 + v_1 \\ h_0v_0 +
  kv_1\end{smallmatrix})$, if $v$ belongs to the right hand side
above then $kv_0 + v_1\in h_0^{-1/2}\CH$ and $h_0v_0 + kv_1 \in
\CH$, thus $\hat{K}v\in h_0^{-1/2}\CH \oplus \cH $, hence $v\in
D(\dot K)$. Reciprocally, if $v\in D(\dot K)$ then
\[
v_0\in h_0^{-1/2}\CH,\ v_1\in\CH,\ kv_0 + v_1\in h_0^{-1/2}\CH, 
\  h_0v_0 + kv_1 \in \CH .
\]
We have to show $v_0 \in h_0^{-1/2}\CH\cap h_0^{-1}\CH $ and 
$v_1 \in \CH^{1/2}$, which follow from 
$v_0 \in h_0^{-1}\CH$ and $v_1 \in h_0^{-1/2}\CH$. The last 
relation is a consequence of $kv_0 + v_1\in h_0^{-1/2}\CH$ 
because $k\in\CB(h_0^{-1/2}\CH)$. 
Since $k$ is bounded on $\CH$ we finally get 
$h_0v_0 \in \CH - kv_1\subset\CH$ hence 
$v_0 \in h_0^{-1}\CH $.
\qed

\begin{lemma}\label{lm:p-1}
Assume that the conditions \eqref{Hyp1} and \eqref{A2} are satisfied and let 
$z\in\rho(h,k)\setminus\R$. Then the maps $p(z)^{-1}$ and $p(z)^{-1}h_0$ induce 
continuous operators $h_0^{-1/2}\CH\to h_0^{-1/2}\CH\cap h_0^{-1}\CH$ and 
$h_0^{-1/2}\CH\to\CH^{1/2}$ respectively.
\end{lemma}
\proof
We set $m=z-k$ and, to simplify the writing, we do not specify $z$ unless this is really 
necessary, e.g. we write $p$ for $p(z)$ and $p=h_0-m^2$. 
From \eqref{A2} it follows that $m$ induces
bounded invertible operators in all the spaces $\CH^{s}$ with
$-1/2\leq s\leq 1/2$ and in the space $h_0^{-1/2}\CH$ (in all $h_0^s\CH$
with $-1/2\leq s \leq 1/2$, in fact). 
Since $h_0$ extends to a unitary operator $h_0^{-1/2}\CH\to
h_0^{1/2}\CH$ and $h_0^{1/2}\CH$ is a dense subspace of
$\CH^{-1/2}$, the operator $p^{-1}h_0$ extends to a bounded map
$p^{-1}h_0:h_0^{-1/2}\CH\to\CH^{1/2}$. Then we write
$$p^{-1}=p^{-1}(m^2-h_0+h_0)m^{-2}=p^{-1}h_0m^{-2} -m^{-2} $$ 
from which it follows that  $p^{-1}$ extends to an operator in
$\CB(h_0^{-1/2}\CH)$. We still have to prove that 
$p^{-1}$ sends $h_0^{-1/2}\CH$ into $h_0^{-1}\CH$.
For this we note that
$$h_0p^{-1}=(h_0-m^2+m^2)p^{-1}=1+m^2p^{-1} $$
and thus, by what we just proved, we see that 
$h_0p^{-1} h_0^{-1/2}\CH\subset h_0^{-1/2}\CH$ 
hence $p^{-1}$ sends $h_0^{-1/2}\CH$ into $h_0^{-3/2}\CH\cap h_{0}^{-1/2}\CH\subset h_{0}^{-1}\CH$, 
which clearly proves the assertion.
\qed

\begin{proposition} \label{lemmaresHdot}
If \eqref{Hyp1}, \eqref{A2} are true then $\rho(\dot{H})\setminus
\R=\rho(h,k)\setminus \R$ and for $z$ in this set
\begin{equation} \label{resdotH}
\dot{R}(z):=(\dot{H}-z)^{-1}=\Phi(k)\left(\begin{array}{cc}
    p^{-1}(z)(z-k) & p^{-1}(z) \\ 1 + (z-k)p^{-1}(z)(z-k) &
    (z-k)p^{-1}(z) \end{array}\right)\Phi(-k).
\end{equation}
\end{proposition}
\proof 

As in the proof of Lemma \ref{lm:DH} we prove the corresponding
statement for the operator $\dot{K}$. Fix $z\in\rho(h,k)\setminus
\R$ and adopt the notations of the proof of Lemma \ref{lm:p-1}.  We
show that $z\in\rho(\dot K)$ and that $(\dot K-z)^{-1}$ is just the
matrix in \eqref{resdotH} or in \eqref{resK}:
\begin{equation} \label{resdotH1}
(\dot K-z)^{-1} =\left(\begin{array}{cc}
    p^{-1}m & p^{-1} \\ 1 + mp^{-1}m & mp^{-1} \end{array}\right) .
%=\left(\begin{array}{cc}
%    \big( 1-p^{-1}h_0 \big)m^{-1} & p^{-1} \\ 
%m^{-1}p^{-1}h_0m^{-1} & mp^{-1} \end{array}\right)
\end{equation}
We denote $S$ the matrix in the right hand side of \eqref{resdotH1} 
and first show that $S$ sends $h_0^{-1/2}\CH\oplus\CH$ into 
$D(\dot K)$ as defined in \eqref{eq:DK}. Thus, if $v_0\in h_0^{-1/2}\CH$
and $v_1\in\CH$ we must prove that
\begin{equation}\label{eq:sv}
p^{-1}mv_0+p^{-1}v_1\in h_0^{-1/2}\CH\cap h_0^{-1}\CH   \quad\text{and}\quad
(1 + mp^{-1}m)v_0 +  mp^{-1} v_1\in \CH^{1/2} .
\end{equation}
From Lemma \ref{lm:kg} we get $p^{-1}v_1\in\CH^{1}\subset
h_0^{-1/2}\CH\cap h_0^{-1}\CH$ hence also $mp^{-1} v_1\in
\CH^{1/2}$.  Thus it remains to treat the terms involving $v_0$.
From Lemma \ref{lm:p-1} we get $p^{-1}mv_0\in h_0^{-1/2}\CH\cap
h_0^{-1}\CH$.  On the other hand, since $p=h_0-m^2$, Lemma
\ref{lm:kg} and a simple computation give
\begin{equation}\label{eq:pm}
1 + mp^{-1}m=m^{-1}p^{-1}h_0m^{-1} 
%\quad\text{and}\quad p^{-1}m= \big( 1-p^{-1}h_0 \big)m^{-1}
\end{equation}
in the sense of bounded operators $\CH^{-1/2}\to\CH^{1/2}$. 
From this relation and Lemma \ref{lm:p-1} we get $(1 + mp^{-1}m)v_0
\in \CH^{1/2} $.

Thus $S: h_0^{-1/2}\CH\oplus\CH \to D(\dot K)$ and a straightforward
computation gives $(\dot K-z)Sv=v$ for all $v\in
h_0^{-1/2}\CH\oplus\CH$. On the other hand, if $u\in D(\dot K)$ and
$v=(\dot K -z)u$ then it is easy to show that $u=Sv$. This finishes
the proof of the relation $S=(\dot K-z)^{-1}$, i.e. of
\eqref{resdotH1}. 

It remains to be shown that
$\rho(\dot{K})\setminus\R\subset\rho(h,k)$. Assume that $z\nin\R$
and 
\[
\dot{K}-z:\big(h_0^{-1/2}\CH\cap h_0^{-1}\CH \big)\oplus\CH^{1/2} 
\to  h_0^{-1/2}\CH\oplus\CH
\]
is bijective. Then for any $v$ of the form 
$v=(\begin{smallmatrix} 0 \\ v_1\end{smallmatrix})\in
h_0^{-1/2}\CH\oplus\CH $ 
there is a unique $u=(\begin{smallmatrix} u_0 \\
  u_1\end{smallmatrix})$ with $u_0\in h_0^{-1/2}\CH\cap h_0^{-1}\CH
$ and $u_1 \in \CH^{1/2}$ such that $-mu_0+u_1=0$ and
$h_0u_0-mu_1=v_1$. Then $u_0=m^{-1}u_1\in\CH^{1/2}$ and also $u_0\in
h_0^{-1}\CH$ hence $u\in\CH^1$ and we have
$pu_0=(h_0-m^2)u_0=v_1$. Thus $p:\CH^1\to\CH$ is surjective. It is
also injective because if $pu_0=0$ for some $u_0\in\CH^1$ then 
$u=(\begin{smallmatrix} u_0 \\  mu_0\end{smallmatrix})\in
D(\dot{K})$ and $(\dot{K}-z)u=0$ hence $u=0$. 
\qed

We point out a simple relation between $H$ and $\dot H$ 
(a similar statement holds for $K$ and $\dot K$). 
Recall that $\CE\subset\dot\CE$ continuously and densely. 

\begin{lemma}\label{lm:closure}
$\dot{H}$ coincides with the closure of $H$ in $\dot\CE$. 
\end{lemma}
\proof

If $z\in\rho(h,k)\setminus\R$  then $z$ belongs to $\rho(H)\cap \rho(\dot{H})$ and the resolvents $R=(H-z)^{-1}$ and $\dot R=(\dot H-z)^{-1}$
are bounded operators in $\CE$ and $\dot\CE$ respectively.  Moreover, $\dot R$ 
is clearly a continuous extension of $R$ to $\dot\CE$, so it is the closure of $R$
in $\dot\CE$. By thinking in terms of graphs one easily sees that $\dot H-z=\dot{R}^{-1}$
is the closure of $H-z=R^{-1}$ in $\dot\CE$. 
\qed

We will often consider the case where \eqref{A3} is fulfilled.  In
this case we have that
\begin{equation*}
D(\dot{H})=\big( h^{-1/2}\CH\cap h^{-1}\CH \big)\oplus\CH^{1/2}
\end{equation*}
and $\dot{H}$ is selfadjoint (see e.g. \cite[Lemma
2.1.1]{Ha03}). Note also that in this case the resolvent of $\dot{H}$
is given by:  (see \cite[Proposition 5.7]{GGH2}
\begin{equation}
\label{resdotHself}
\dot{R}(z)=\left(\begin{array}{cc} z^{-1}p^{-1}(z)h-z^{-1} & p^{-1}(z) \\ p^{-1}(z)h &
    zp^{-1}(z) \end{array}\right).
\end{equation} 
Moreover, if we assume \eqref{A3}, then $\Vert
\dot{R}(z)\Vert_{\CB(\dot{\CE})}\le |{\rm Im}z|^{-1}.$ Using
\cite[Proposition 5.10]{GGH2} we obtain the following resolvent
estimate for $H$:

\begin{proposition}
\label{propbasicreest}
Assume \eqref{Hyp1}-\eqref{A2}. Then:
\begin{equation}
\label{1.16}
\Vert R(z)\Vert_{\CB(\CE)}\lesssim(1+|z|^{-1})\Vert
\dot{R}(z)\Vert_{\CB(\dot{\CE})}+|z|^{-1}.
\end{equation}
Assume in addition \eqref{A3}. Then:
\begin{equation}
\Vert R(z)\Vert_{\CB(\CE)}\lesssim(1+|z|^{-1})|\mathrm{Im} z|^{-1}.    %+|z|^{-1}.
\end{equation}
\end{proposition}

\subsubsection{Gauge transformations}
\label{SecGT}
Let us recall that our starting point was the Klein-Gordon equation
\begin{equation}
\label{KGcomp}
(\partial_t-ik)^2u+h_0u=0.
\end{equation}
If $u$ is solution of \eqref{KGcomp} and $\ell\in \rr$, then $v=e^{-it\ell}u$ solves :
\begin{equation}
\label{KGcompl}
(\partial_t-i(k-\ell))^2v+h_0v=0.
\end{equation}
Let us formulate this in terms of generators: if 
\[\Phi(\ell)H\Phi^{-1}(\ell)=:H_{\ell}+\ell,\]
then
\[H_{\ell}=\left(\begin{array}{cc} 0 & 1 \\
p(\ell) & 2(k-\ell)\end{array}\right),\, p(\ell)=h_0-(k-\ell)^2.\]
It follows that if there exists $\ell\in \R$ such that \eqref{A3} is
fulfilled with $h$ replaced by $p(\ell)$ and $k$ by $k-\ell$, then
$\dot{H}$ is selfadjoint on the homogeneous energy space
\[
\dot{\CE}=\Phi(\ell)(p(\ell)^{-1/2}\CH\oplus\CH).
\]

\subsection{Existence of the dynamics}

From \cite[Corollary 8.6]{GGH1} we obtain :
\begin{lemma}
$H$ is the generator of a $C_0$-group $e^{-itH}$ on $\CE$.
\end{lemma}

Now we show that  $e^{-itH}$ extends to a $C_0$-group on $\dot{\CE}$.

\begin{lemma}
$\dot{H}$ is the generator of a $C_0-$ group on $\dot{\CE}$ and for each real $t$
the operator $e^{-it\dot{H}}$ coincides with the continuous extension of 
$e^{-itH}$ to $\dot\CE$.
\end{lemma}
\proof

We start by proving that for some constants $C,\omega>0$ we have 
\begin{equation}
\label{2.6.1}
\Vert e^{-itH}\varphi\Vert_{\dot{\CE}}\le Ce^{\omega |t|}\Vert\varphi\Vert_{\dot{\CE}}
\quad \forall \varphi\in \CE .
\end{equation}
Let first $\varphi\in D(H)$. We compute by using \eqref{eq:hnorm} for
$u=(u_0,u_1)=e^{-itH}\varphi$
\begin{eqnarray*}
\frac{d}{dt}\Vert u\Vert_{\dot{\CE}}^2&=&2 {\rm Re}
(ihu_0+iku_1|u_1-ku_0)+2{\rm Re}(h_0u_0|iu_1)\\
&=&([ik,h]u_0|u_0)\lesssim (h_0u_0|u_0)\lesssim \Vert
u\Vert^2_{\dot{\CE}}.
\end{eqnarray*}
The inequality \eqref{2.6.1} then follows for $\varphi\in D(H)$ by the
Gronwall's lemma and for $\varphi\in \CE$ by density. From
\eqref{2.6.1} we see that $e^{-itH}$ extends to a continuous operator $V_t$
on $\dot{\CE}$ such that $\|V_t\|_{\dot\CE} \leq C e^{\omega |t|}$. This clearly 
implies that $V_t$ is a $C_0$-group on $\dot\ce$ and from Nelson's invariant domain theorem 
it follows that its generator is the closure of $H$ in $\dot\ce$, which by 
Lemma \ref{lm:closure} is just $\dot{H}$.
\qed

\section{Meromorphic extensions}
\label{SecME}
In this Section we discuss various facts related to meromorphic extensions of quadratic pencils.
\subsection{Background and definitions}
\label{SecMEBG}

\begin{definition}
Let $\CH$ be a Hilbert space. For $z_0\in \C$, let $\CU$ be a neighborhood of $z_0$, and let 
$F:\CU\setminus\{z_0\}\rightarrow B(\CH)$ be a holomorphic
function. We say that $F$ is {\em finite meromorphic} at $z_0$ if the Laurent expansion  of $F$ at $z_0$ has the form 
\[ F(z)=\sum_{n=m}^{+\infty}(z-z_0)^nA_n,\quad m>-\infty,\]
the operators $A_m,...,A_{-1}$ being of finite rank, if $m<0$. If, in addition, $A_0$ is a Fredholm operator, then $F$ is called Fredholm at $z_0$.
\end{definition}

We will need the following fact, cf. \cite[Proposition 4.1.4]{GoLe09}:

\begin{proposition}
\label{Prop5}
Let $\CD\subset\C$ be a connected open set, let $Z\subset\CD$ be a 
discrete and closed subset of $\CD$, and let $F:\CD\setminus  Z \rightarrow B(\CH)$ 
be a holomorphic function. Assume that
\begin{itemize}
\item[--] $F$ is finite meromorphic and Fredholm at each point  of $\CD$;
\item[--] there exists $z_0\in \CD\setminus Z$ such that $F(z_0)$ is invertible.
\end{itemize}
Then there exists a discrete closed subset $Z'$ of $\CD$ such that 
$Z\subset Z'$ and:
\begin{itemize}
\item[--] $F(z)$ is invertible for each $z\in \CD\setminus Z'$;
\item[--] $F^{-1}:\CD\setminus Z'\rightarrow\CL(\CH)$ is finite
  meromorphic and Fredholm at each point of $\CD$.
\end{itemize} 
\end{proposition}

\subsection{Meromorphic extensions of weighted resolvents}

Let $w$ be a positive selfadjoint operator on $\CH$ with bounded inverse
$w^{-1}$. One should think of $w$ as a {\em weight function}. $w$ and
$w^{-1}$ will act on $\CE,\, \dot{\CE}$ by $w(u_0,u_1)=(wu_0,wu_1)$ etc. 
In this subsection we will require \eqref{A3}. 

We need the following hypotheses:
\begin{equation}
\tag{ME1}\label{HypME}
\left\{\begin{array}{cl} a) & 
wkw\in \CB(\CH).\\
b) & [k,w]=0\\
c) & h^{-1/2}[h,w^{-\epsilon}]w^{\epsilon/2} \in\CB(\CH) \, \forall \,  0<\epsilon\leq 1 \\
%,\, [h,w^{-\epsilon}]w^{\epsilon/2}h^{-1/2}\,  [h,w^{-\epsilon}]h^{-1/2}\in\CB(\CH) , 
d) & \text{if } \epsilon>0
\text{ then } \Vert w^{-\epsilon}u\Vert\lesssim \Vert h^{1/2}u\Vert \quad
\forall u\in h^{-1/2}\CH\\
e) & w^{-1}\<h\>^{-1}\in \CB_{\infty}(\CH) . 
\end{array}\right.
\end{equation}
Note that part $d)$ of \eqref{HypME} is a Hardy type inequality and it implies the boundedness
of the operators $w^{-\epsilon}h^{-1/2}$ and $h^{-1/2}w^{-\epsilon}$. Later on we shall 
see that these two operators are compact if \eqref{HypME} is satisfied (see the proof of
Lemma \ref{merextp-1}). 

Observe that from part $c)$ we also get 
$w^{\epsilon/2}[h,w^{-\epsilon}]h^{-1/2} \in\CB(\CH)$. Moreover, we shall have
$w^{-\epsilon}\<h\>^{-\tilde{\epsilon}}\in \CB_{\infty}(\CH)$ for all 
$\epsilon,\tilde{\epsilon}>0$. Indeed, $w^{-z}\<h\>^{-z}\in \CB_{\infty}(\CH)$ is an
analytic function of $z$ in the region ${\rm Re}\, z>0$ and by $e)$ above this is a
compact operator for ${\rm Re}\, z \geq1$, hence for any $z$.

We also need the assumption
\begin{equation}
\tag{ME2}\label{ME2}
\left\{\begin{array}{cl} & 
\forall \, \epsilon>0 \,  \exists \, \delta_{\epsilon}>0  \text{ such that }
w^{-\epsilon}(h-z^2)^{-1}w^{-\epsilon} 
\text{ extends from } {\rm Im}z>0 \\
&   
 \text{to } {\rm Imz}>-\delta_{\epsilon} 
\text{ as a finite meromorphic function with values in } \CB_{\infty}(\CH). 
\end{array}\right.
\end{equation}

\begin{lemma}
\label{merextp-1}
Assume \eqref{Hyp1}-\eqref{A3}, \eqref{HypME}-\eqref{ME2} and let
$0<\epsilon\leq 1$. Then the operators
\begin{align*}
& {\mathrm(i)} \ \  \ w^{-\epsilon}p^{-1}(z)w^{-\epsilon}, \\
& {\mathrm(ii)} \  \ (h+1)^{1/2}w^{-\epsilon}p^{-1}(z)w^{-\epsilon}, \\
& {\mathrm(iii)} \    w^{-\epsilon}p^{-1}(z)hw^{-\epsilon}h^{-1/2},  \\
& {\mathrm(iv)} \    h^{1/2}w^{-\epsilon}p^{-1}(z)(z-2k)w^{-\epsilon}h^{-1/2},
\end{align*}
extend from $\{{\rm Im\,}z>0$ to ${\rm  Im\,}z>-\delta_{\epsilon/2}\}$ as finite meromorphic  functions with values in $\CB_{\infty}(\CH)$. 
\end{lemma} 
\proof
The resolvent identity yields
\[
w^{-\epsilon}p^{-1}(z)w^{-\epsilon}=
w^{-\epsilon}(h-z^2)^{-1}w^{-\epsilon}
(1+2z w^\epsilon k w^\epsilon 
\cdot w^{-\epsilon}(h-z^2)^{-1}w^{-\epsilon})^{-1}.
\]
Applying Proposition \ref{Prop5} to
\[
F(z)= 1+2z w^\epsilon k w^\epsilon \cdot w^{-\epsilon}(h-z^2)^{-1}w^{-\epsilon}
\]
proves $(i)$. We now write
\begin{eqnarray*}
hw^{-\epsilon}p^{-1}(z)w^{-\epsilon}&=&w^{-\epsilon}hp^{-1}(z)w^{-\epsilon}+[h,w^{-\epsilon}]w^{\epsilon/2}w^{-\epsilon/2}p^{-1}(z)w^{-\epsilon}\nonumber\\
&=&w^{-2\epsilon}+z(z-2k)w^{-\epsilon}p^{-1}(z)w^{-\epsilon}+[h,w^{-\epsilon}]w^{\epsilon/2}w^{-\epsilon/2}p^{-1}(z)w^{-\epsilon}.
\end{eqnarray*}
This allows us to compute the second operator
\begin{align*}
 &(h+1)^{1/2}w^{-\epsilon}p^{-1}(z)w^{-\epsilon} \\=&
(h+1)^{-1/2}w^{-2\epsilon}+(h+1)^{-1/2}z(z-2k)
\cdot w^{-\epsilon}p^{-1}(z)w^{-\epsilon}\\
&+ (h+1)^{-1/2} [h,w^{-\epsilon}]w^{\epsilon/2}
\cdot w^{-\epsilon/2}p^{-1}(z)w^{-\epsilon}+(h+1)^{-1/2}w^{-\epsilon}p^{-1}(z)w^{-\epsilon} .
\end{align*}

By using $(i)$ and hypotheses $c), e)$ of \eqref{HypME} we get $(ii)$.

Let us now prove $(iii)$. Let $\chi\in C_0^{\infty}(\R),\, \chi=
1$ in a neighborhood of $0$. We write
\begin{align*}
w^{-\epsilon}p^{-1}(z)hw^{-\epsilon}h^{-1/2}
=&w^{-\epsilon}p^{-1}(z)hw^{-\epsilon}h^{-1/2}(1-\chi(h))\\
&+w^{-\epsilon}p^{-1}(z)hw^{-\epsilon}h^{-1/2}\chi(h)=:T_1+T_2.
\end{align*}
We have
\[
T_1=w^{-2\epsilon}h^{-1/2}(1-\chi(h))+
w^{-\epsilon} p^{-1}(z) w^{-\epsilon} \cdot z(z-2k)h^{-1/2}(1-\chi(h)).
\]
The first term is compact by a comment after hypothesis \eqref{HypME}, 
the second is compact outside the poles of
$w^{-\epsilon}p^{-1}(z)w^{-\epsilon}$ by part $(i)$.  We have
\[
T_2 = w^{-\epsilon}p^{-1}(z)w^{-\epsilon/2}  \cdot 
w^{\epsilon/2}[h,w^{-\epsilon}]h^{-1/2}\chi(h)
+ w^{-\epsilon}p^{-1}(z)w^{-\epsilon} \cdot h^{1/2}\chi(h).
\]
By the same comment we see that both terms here extend to finite 
meromorphic  functions  with values in $B_{\infty}(\CH)$ in 
$\{{\rm Im}\, z>-\delta_{\epsilon/2}\}$. Thus $(iii)$ is proved. 

Note that since $h=p+z(z-2k)$ we have 
\[
w^{-\epsilon}p^{-1}(z)hw^{-\epsilon}h^{-1/2} = 
w^{-2\epsilon}h^{-1/2} + 
w^{-\epsilon}p^{-1}(z)w^{-\epsilon/2} \cdot z(z-2k) w^{-\epsilon/2}h^{-1/2} .
\]
The left hand side here is a compact operator by what we just proved 
and the last term is also compact by $(i)$ and because $
w^{-\epsilon/2}h^{-1/2}$ is bounded by hypotheses $d)$ of \eqref{HypME}. Since 
$0<\epsilon\leq 1$ is arbitrary, we see that 
$w^{-\epsilon}h^{-1/2}$ and $h^{-1/2}w^{-\epsilon}$ are compact operators 
if $0<\epsilon\leq1$. 

Finally, we prove $(iv)$. We have
\begin{align*}
&h^{1/2}w^{-\epsilon}p^{-1}(z)(z-2k)w^{-\epsilon}h^{-1/2} \\
=& h^{-1/2}[h,w^{-\epsilon}]w^{\epsilon/2} \cdot 
w^{-\epsilon/2}p^{-1}(z)w^{-\epsilon/2} \cdot
(z-2k)w^{-\epsilon/2}h^{-1/2}\\
&+ h^{-1/2}w^{-\epsilon/2} \cdot
w^{-\epsilon/2}hp^{-1}(z)w^{-\epsilon/2}  \cdot
(z-2k) w^{-\epsilon/2}h^{-1/2}.
\end{align*}
For the first term of the right hand side we use $c)$ of \eqref{HypME} as well as $(i)$ and 
the boundedness of $w^{-\epsilon/2}h^{-1/2}$. For the last term we note that it is equal to
\begin{align*}
& h^{-1/2}w^{-\epsilon}  (z-2k) w^{-\epsilon}h^{-1/2} \\
 +& h^{-1/2}w^{-\epsilon/2} z(z-2k) \cdot w^{-\epsilon/2}
p^{-1}(z)w^{-\epsilon/2}  
\cdot  (z-2k) w^{-\epsilon/2}h^{-1/2} .
\end{align*}
The first term is a holomorphic function with values in $B_{\infty}(\CH)$ because
$w^{-\epsilon}h^{-1/2}$,  $h^{-1/2}w^{-\epsilon}\in B_{\infty}(\CH)$ . The last 
line is treated as before. This proves $(iv)$.
\qed

Using this lemma we obtain a meromorphic extension of the  truncated resolvent of $H$.

\begin{proposition}
\label{merexttrresH}
Assume the hypotheses of Lemma \ref{merextp-1} and let $\epsilon>0$. 
Then $w^{-\epsilon}R(z)w^{-\epsilon}$ and $w^{-\epsilon}\dot{R}(z)w^{-\epsilon}$ 
extend finite meromorphically to $\{{\rm Im}z>-\delta_{\epsilon/2}\}$ as a operator valued
functions with values in $\CB_{\infty}(\CE)$ and $\CB_{\infty}(\dot{\CE})$
respectively.
\end{proposition}
\proof

We first prove the assertion concerning $R(z)$. Using \eqref{resH} we see that
\begin{eqnarray*}
w^{-\epsilon}R(z)w^{-\epsilon}&=&w^{-\epsilon} p(z)^{-1}\left(\begin{array}{cc}
    z-2k & 1 \\ h & z \end{array}\right)w^{-\epsilon}\\
&=&\left(\begin{array}{cc} 0 & 0\\ w^{-2\epsilon} &
    0 \end{array}\right)+w^{-\epsilon}p^{-1}w^{-\epsilon}\left(\begin{array}{cc}
    z-2k & 1 \\ z(z - 2k) & z\end{array}\right).
\end{eqnarray*}
We then use $(i),\, (ii)$ of Lemma \ref{merextp-1} as well as assumption 
\eqref{HypME}{\it e)}.

Let us now treat $\dot{R}(z)$. Recall that under hypothesis \eqref{A3} we
have
\[\dot{R}(z)=\left(\begin{array}{cc} z^{-1}p^{-1}(z)h-z^{-1} & p^{-1}(z)\\ p^{-1}(z)h &
  zp^{-1}(z)\end{array}\right).\]
Using that $w^{-\epsilon}h^{-1/2}$ is bounded by hypothesis \eqref{HypME}{\it d)} we can write
\[w^{-\epsilon}\dot{R}(z)w^{-\epsilon}=w^{-\epsilon}p^{-1}(z)\left(\begin{array}{cc} z-2k & 1 \\ h &
  z\end{array}\right)w^{-\epsilon}.\]
We therefore have to show that
\begin{align*}
&h^{1/2}w^{-\epsilon}p^{-1}(z)(z-2k)w^{-\epsilon}h^{-1/2},\,
h^{1/2}w^{-\epsilon}p^{-1}(z)w^{-\epsilon},\\ 
&w^{-\epsilon}p^{-1}(z)hw^{-\epsilon}h^{-1/2},\, w^{-\epsilon}p^{-1}(z)zw^{-\epsilon}
\end{align*}
extend finite meromorphically with values in $\CB_{\infty}(\CH)$. This follows from Lemma \ref{merextp-1}.
\qed

\section{Klein-Gordon operators with ``two ends''}
\label{SecTE}
In this section we discuss an abstract framework corresponding to Klein-Gordon operators on manifolds with ``two ends''. The essential condition is that the asymptotic Hamiltonians in both ends are selfadjoint for a positive energy norm, modulo some gauge transformation.
\subsection{Assumptions}
\label{asympham}

We assume that there exists a selfadjoint operator $x$ on $\CH$ with 
$\sigma(x)=\sigma_{ac}(x)=\R$
such that $w$ is a smooth function of $x$, such that $k$ commutes with $x$
and such that $h_0$ is {\em local} in $x$ in the following sense: 

if 
$\chi_1,\chi_2\in C^{\infty}(\R)$ are bounded together with all their derivatives 
and if $\supp\chi_1\cap\supp\chi_2=\emptyset$, then $\chi_1(x)h_0\chi_2(x)=0$.
To summarize we assume:
\begin{equation}
\tag{TE1}\label{TE1}
\left\{\begin{array}{l}
{[}x,k{]}=0,\\
w=w(x) \quad\text{with } w\in C^{\infty}(\R),\\
h_0 \text{ is local in } x .
\end{array}\right.
\end{equation}
Let $\chi_{\pm}\in C^{\infty}(\R)$ such that  $\chi_+^2+\chi_-^2=1$,
$0\le\chi_{\pm}\le 1$, and
\begin{eqnarray*}
\supp \chi_-&\subset&(-\infty,1),\quad\chi_-=
1\quad \text{on}\quad(-\infty,0],\\
\supp\chi_+&\subset&(-1,\infty),\quad \chi_+=
1\quad \text {on}\quad[0,\infty). 
\end{eqnarray*}
If $\epsilon>0$ is a real number, let 
\[
i_{\pm}=\chi_{\pm}(\epsilon^{-1}x),\quad
j_{\pm}=\chi_{\pm}(\epsilon^{-1}x\mp 3).
\]
We then have
\[
j_{\pm}i_{\pm}=j_{\pm},\quad i_+j_-=i_-j_+=0.
\]
Let $\tilde{\chi}\in C_0^{\infty}((-2,2))$ with $\tilde{\chi}= 1$ on $[-1,1]$.
We set $\tilde{i}=\tilde{\chi}(R^{-1}x)$.  Let
\beq\label{tutu}
k_{\pm}:=k\mp\ell j_{\mp}^2,\ 
 h_{\pm}:=h_0-k_{\pm}^2.
\eeq
We also set:
\beq\label{titi}
\tilde{h}_-:=h_-+2\ell k_--\ell^2=h_0-(\ell-k_-)^2.
\eeq
We require
\begin{equation}
\tag{TE2}\label{TE2}
\text{there are } \ell\in \R,\, R>0 
\text{ such that } (h_+,k_+)
\text{ and }(\tilde{h}_-,k_--\ell)
\text{ satisfy \eqref{A3}.}
\end{equation} 
We also set:
\(
p_{\pm}(z):=h_{\pm}+z(2k_{\pm}-z).
\)
Note that $\tilde{h}_-=p_-(\ell).$

\subsection{Asymptotic Hamiltonians}\label{Harlem Désir}

We introduce the homogeneous energy spaces
\[
\dot{\CE}_+:=h_+^{-1/2}\CH\oplus\CH,\,
\dot{\CE}_-:=\Phi(\ell)(\tilde{h}_-^{-1/2}\CH\oplus\CH).
\]
Then the operators 
\begin{equation}
\dot{H}_{\pm}=\left(\begin{array}{cc} 0 & 1 \\ 
h_{\pm} & 2k_{\pm}\end{array}\right)
\end{equation}
are selfadjoint with domains 
\begin{eqnarray*}   %\label{domdotH+}
D(\dot{H}_+)&=&h_+^{-1/2}\CH\cap h_+^{-1}\CH\oplus\<h_+\>^{-1/2}\CH,\\
\label{domdotH-}
D(\dot{H}_-)&=&\Phi(\ell)((\tilde{h}_-^{-1/2}\CH\cap\tilde{h}_-^{-1}\CH)\oplus\<\tilde{h}_-\>^{-1/2}\CH).
\end{eqnarray*}
We denote \( \dot{R}_{\pm}(z):=(\dot{H}_{\pm}-z)^{-1} \). 

We will also need the following assumption for $\ell$ as in assumption \eqref{TE2}:
\begin{equation}
\tag{TE3}\label{TE3}
\left\{\begin{array}{ll} a)& wi_+ki_+w,\,  wi_-(k-\ell)i_-w\in\CB(\CH),\\ [1mm]
b) & {[}h,i_{\pm}]=\tilde{i}{[}h,i_{\pm}{]}\tilde{i},\\ [1mm]
c) & (h_+,k_+,w)\text{ and } (\tilde{h}_-,k_- -\ell,w)
\text{ fulfill }    \eqref{HypME}, \eqref{ME2},
\\ [1mm]
d) &  h_{\pm}^{1/2}i_{\pm}h_{\pm}^{-1/2},\,
  h^{1/2}_0i_{\pm}h_{0}^{-1/2}\in\CB(\CH),\\ [1mm]
e) & \text{the operators }w[h,i_{\pm}]wh_{\pm}^{-1/2},\, w[h,i_{\pm}]wh_{0}^{-1/2},\,
[h,i_{\pm}]h_{\pm}^{-1/2}, \\ [1mm]
& [h,i_{\pm}]h_{0}^{-1/2},\, h_0^{-1/2}[w^{-1},h_0]w \text{ are bounded on } \CH,\\ [1mm]
f) & \text{if } \epsilon>0
\text{ then } \Vert w^{-\epsilon}u\Vert\lesssim \Vert h_0^{1/2}u\Vert \quad
\forall u\in h_0^{-1/2}\CH .
\end{array}\right.
\end{equation}
As a direct consequence of Proposition \ref{merexttrresH} we obtain
\begin{proposition}
\label{prop4.1}
For any $\epsilon>0$ the functions $w^{-\epsilon}R_{\pm}(z)w^{-\epsilon}$ and
$w^{-\epsilon}\dot{R}(z)w^{-\epsilon}$ extend finite meromorphically to 
$\{{\rm Im}z>-\delta_{\epsilon/2}\}$  with values 
in $\CB_{\infty}(\CE_{\pm})$ and  $\CB(\dot{\CE}_{\pm}$) respectively. 
\end{proposition}

\subsection{Construction of the resolvent}

We will need the following lemma:
\begin{lemma}
\label{lemma4.3.1}
The linear maps
\begin{eqnarray*}
i_{\pm}:\dot{\CE}\rightarrow\dot{\CE},\, i_{\pm}:\dot{\CE}_{\pm}\rightarrow
\dot{\CE}_{\pm},\, i_{\pm}:\dot{\CE}_{\pm}\rightarrow\dot{\CE},\,
i_{\pm}:\dot{\CE}\rightarrow\dot{\CE}_{\pm}
\end{eqnarray*}
are bounded..
\end{lemma}
\proof

First note that the condition \eqref{TE3}{\it d)} and the relation
$[k,i_{\pm}]=0$ give the continuity of the maps
$i_{\pm}:\dot{\CE}\rightarrow\dot{\CE}$ and 
$i_{\pm}:\dot{\CE}_{\pm}\rightarrow\dot{\CE}_{\pm}$. 
Then note that
\[
i_+(h_++k^2)i_+=i_+h_0i_+,\quad i_-(\tilde{h}_-+(k-\ell)^2)i_-=i_-h_0i_-.
\]
This together with \eqref{TE3}{\it d)} gives the continuity of
$i_{\pm}:\dot{\CE}\rightarrow\dot{\CE}_{\pm}$.
We then claim that               
\begin{eqnarray}
\label{TE4.3.1}
i_+(h_0+k^2)i_+\lesssim i_+h_+i_+,  \\
\label{TE4.3.2}
i_-(h_0+(k-\ell)^2)i_-\lesssim i_-\tilde{h}_-i_-.
\end{eqnarray}
Indeed, \eqref{TE4.3.1} follows from
\[
i_+(h_0+k^2)i_+\lesssim i_+h_+i_++w^{-2}i_+^2\lesssim i_+h_+i_+.
\]
Here we have used \eqref{TE3}{\it c)}. Then \eqref{TE4.3.2} follows from
\[
i_-(h_0+(k-\ell)^2)i_-\lesssim
i_-\tilde{h}_-i_-+i_-(k-\ell)^2i_-\lesssim
i_-\tilde{h}_-i_-+w^{-2}i_-^2\lesssim i_-\tilde{h}_-i_-.
\]
Finally, \eqref{TE4.3.1}, \eqref{TE4.3.2} ans \eqref{TE3}{\it d)} give the continuity of
$i_+:\dot{\CE}_+\rightarrow\dot{\CE}$ and $i_-:\dot{\CE}_-\rightarrow\dot{\CE}$.
\qed

\medskip

We introduce now a new operator:
\begin{equation}\label{defdeQ}
Q(z):=i_-(\dot{H}_--z)^{-1}i_-+i_+(\dot{H}_+-z)^{-1}i_+.
\end{equation}
Thanks to  Lemma \ref{lemma4.3.1} $ Q(z)$ is well defined as a
bounded operator on $\dot{\CE}$. 
We now compute~:
\begin{eqnarray*}
(\dot{H}-z)Q(z)=1+[\dot{H},i_-](\dot{H}_--z)^{-1}i_-+[\dot{H},i_+](\dot{H}_+-z)^{-1}i_+.
\end{eqnarray*}
Note that
\begin{equation*}
[\dot{H},i_{\pm}]=\left(\begin{array}{cc} 0 & 0 \\ {[}h,i_{\pm}{]} &
    0 \end{array}\right).
\end{equation*}
Let 
\begin{equation}
\label{Kpm}
K_{\pm}(z):=\left(\begin{array}{cc} 0 & 0 \\ {[}h,i_{\pm}{]} & 0 \end{array} \right)\dot{R}_{\pm}(z)i_{\pm},\quad \tilde{K}_{\pm}(z):=i_{\pm}\dot{R}_{\pm}(z)\left(\begin{array}{cc} 0 & 0 \\ {[}h,i_{\pm}{]} & 0 \end{array} \right).
\end{equation}\label{kplusmoins}
Note that by assumptions \eqref{TE1}, \eqref{TE3} the operators
\begin{equation}
\label{HypA10}
[\dot{H},i_{\pm}]w^{\epsilon}\quad\mbox{and}\quad
i_{\pm}(1-j_{\pm})w^{\epsilon}\quad\mbox{are bounded on $\dot{\CE}$ for all $\epsilon>0$}.
\end{equation}
We put
\begin{eqnarray*}
A(z)=K_-(z)(1-j_-)+K_+(z)(1-j_+): \C^+\rightarrow {\mathcal B}(\dot{\CE})
\end{eqnarray*}
Using Proposition \ref{prop4.1} and Lemma \ref{lemma4.3.1}
we see that $A(z)$ extends meromorphically to ${\rm
Im}  z>-\delta$ with values in $\CB_{\infty}(\dot{\CE})$ for some
$\delta>0$. As $\dot{H}_{\pm}$ are selfadjoint it follows that 
\[\Vert A(z)\Vert_{{\mathcal B}({\dot{\CE}})}\le 1/2\]
for ${\rm Im}z$ sufficiently large. Thus $(1+A(z))^{-1}$ exists for ${\rm Im}z$
large enough. By Proposition \ref{Prop5} there exists a closed
discrete subset $Z^+$ of the half plane $\{{\rm Im}z>-\delta\}$ such
that $(1+A(z))^{-1}$ exists if ${\rm Im}z>-\delta$ and $z\nin Z^+$
and $(1+A(z))^{-1}$ is finitely meromorphic in $\{{\rm
  Im}z>-\delta\}$ and analytic in $\{{\rm Im}>-\delta\}\setminus
Z^+$. Let
\[
K(z)=K_-(z)+K_+(z).
\]
Now observe that we have $j_{a}K_{b}=0$ for $a=\pm,\, b=\pm$
by assumption \eqref{TE3}{\it b)}.  We therefore have
\begin{align*}
& 1+K(z) =(1+K_-(z)j_-+K_+(z)j_+)(1+K_-(z)(1-j_-)+K_+(z)(1-j_+)),\\
& (1+K_-(z)j_-+K_+(z)j_+)^{-1} =1-K_-(z)j_--K_+(z)j_+.
\end{align*}
We can now construct the resolvent of $\dot{H}$ by
\begin{align}
\label{RB} 
R_{\dot{H}} (z) & := Q(z)(1+K(z))^{-1} \nonumber \\
&=Q(z) \big(1+K_-(z)(1-j_-)+K_+(z)(1-j_+) \big)^{-1}
\big( 1-K_-(z)j_--K_+j_+ \big). 
\end{align}
The same considerations are valid in the lower half plane, we obtain
a set of poles $Z^-$. The set $(Z^-\cap\C^-)\cup (Z^+\cap\C^+)$ is
clearly finite. 
\begin{proposition}
\label{ResH}
If the conditions \eqref{Hyp1}-\eqref{A2} and
\eqref{TE1}-\eqref{TE3} are satisfied then there is a finite set
$Z\subset\C\setminus\R$ with $\bar{Z}=Z$ such that the spectra of
$H$ and $\dot{H}$ are included in $\R\cup Z$ and such that the
resolvents $R$ and $\dot{R}$ are finite meromorphic functions on
$\C\setminus \R$. Moreover, the point spectrum of $H$ coincides with
the point spectrum of $\dot{H}$ and the set $Z$ consists of
eigenvalues of finite multiplicity of $H$ and $\dot{H}$.
\end{proposition}

\begin{proof}

  From the previous arguments it follows that if we define
  $R_H(z):=R_{\dot{H}}(z)|_{\CE}$ then $R_H(z)=R(z)$ and
  $R_{\dot{H}}(z)=\dot{R}(z)$ for $z$ with sufficiently large
  (positive or negative) imaginary part.  We know by Proposition
  \ref{lemmaresHdot} that $\rho(\dot{H})\cap(\C\setminus
  \R)=\rho(h,k)\cap(\C\setminus \R)$.  Then we use Proposition
  \ref{pr:est} to see that all the poles of $\dot{H}$ in
  $\C\setminus \R$ are in a finite ball.  But in this ball
  $(1+A(z))^{-1}$ has only a finite number of poles. By using
  \eqref{RB} and an analyticity argument we see that $\dot{R}(z)$
  has only a finite number of poles in $\C\setminus \R$. From the
  analyticity properties of a resolvent family it follows then that
  the non real spectrum $Z$ of $\dot{H}$ coincides with the set of
  non real poles of its resolvent, in particular is finite.
 
  From $\rho(\dot{H})\backslash \rr=\rho(h,k)\backslash \rr$ and Lemma \ref{lm:kg} we see that the complex spectrum is
  invariant under conjugation. Note also that every eigenvector of
  $\dot{H}$ for a non-zero eigenvalue  is in $D(H)$ and thus  an eigenvector of $H$. It remains to
 show  that the complex point spectrum consists exactly of
  complex eigenvalues of $\dot{H}$ and that the corresponding
  eigenspaces are finite dimensional.  If $z_0$ is a pole of
  $\dot{R}(z)$, then on a neighborhood of $z_0$ we may write
  \(\dot{R}(z)= \sum_{n=1}^{N}(z_0-z)^{-n}S_n + S(z)\) with $S$
  holomorphic near $z_0$ and the $S_n\neq0$ of finite rank because
  the function $A(z)$ is finitely meromorphic. From this it follows
  that $z_0$ is an eigenvalue of finite multiplicity of $\dot{H}$,
  see \cite[Ch.\ VIII Sec.\ 8]{Yo} for details.
\end{proof}

From now on we shall denote $\sigma_{pp}^{\C}(\dot{H})$ the set of
non  real eigenvalues of $\dot{H}$. For $z\in
\sigma_{pp}^{\C}(\dot{H})$ the Riesz projector is defined by
\[E(z,\dot{H})=\frac{i}{2\pi}\ointctrclockwise_{\gamma}(\dot{H}-z)^{-1}dz,\]
where $\gamma$ is a small curve in $\rho(\dot{H})$ surrounding
$z$. Let
\[
\one_{pp}^{\C}(\dot{H}):=\oplus_{z\in\sigma_{pp}^{\C}(H)}E(z,\dot{H})
\quad\text{and}\quad 
\CE^{\C}_{pp}(\dot{H})=\one^{\C}_{pp}(\dot{H})\dot\CE  .
\]
Let furthermore 
\begin{equation*}
\one_{\R}(\dot{H})=\one-\one_{pp}^{\C}(\dot{H}),\quad
\CE_{\R}(\dot{H})=\one_{\R}(\dot{H})\dot{\CE}.
\end{equation*}
We clearly have
\begin{equation*}
\dot{\CE}=\CE_{\R}(\dot{H})\oplus\CE_{pp}^{\C}(\dot{H})
\end{equation*}
and both spaces are invariant under $e^{-it\dot{H}}$.

\subsection{Resolvent estimates} \label{resest}

For $R,\, \delta>0$ we put
\begin{eqnarray*}
\CU_0(R,\delta)&=&\{z\in \C\,:\, 0<\vert\mbox{Im} z\vert\le \delta,\,
\vert\mbox{Re}z\vert\le R\}.
\end{eqnarray*}
We have
\begin{lemma}
\label{lemma4.5}
Assume \eqref{Hyp1}, \eqref{A2}, \eqref{TE1}-\eqref{TE3}. Then for
each $R>0$ there are $M,\delta>0$ such that $\sigma(H)\setminus \R$
does not intersect $\CU_0(R,\delta)$ and such that for all $z\in \CU_0(R,\delta)$
\begin{align}
\label{4.5.1}
\Vert \dot{R}(z)\Vert_{\CB(\dot{\CE})} &\lesssim \vert {\rm Im}\, z\vert^{-M} , \\
\label{4.5.2}
\Vert R(z)\Vert_{\CB(\CE)} & \lesssim \left(1+|z|^{-1}\right)\vert {\rm Im}\, z\vert^{-M}
+|z|^{-1} . %\quad \forall z\in \CU_0(R,\delta) .
\end{align}
\end{lemma}

\proof

Recall that 
\begin{equation*}
\dot{R}(z)=Q(z)(1+A(z))^{-1}(1-K_-(z)j_--K_+(z)j_+).
\end{equation*}
We choose $\delta>0$ sufficiently small such that $(1+A(z))^{-1}$ has no poles in $\CU_0(R,\delta)$. We have:
\begin{equation*}
\Vert Q(z)\Vert_{B(\dot{\CE})}\lesssim |{\rm Im} z|^{-1}.
\end{equation*}
The meromorphic extension of $(1+A(z))^{-1}$ has only a finite number of real poles in 
$\overline{\CU_0(R,\delta)}$, hence we have:
\begin{equation*}
\Vert (1+A(z))^{-1}\Vert_{B(\dot{\CE)}}\lesssim |{\rm Im} z|^{-M_1},\quad M_1>0.
\end{equation*}
Noting that $[H,i_{\pm}]:\dot{\CE}\rightarrow \dot{\CE}$ is bounded by
assumption \eqref{TE3} we obtain :
\begin{equation*}
\Vert(1-K_-(z)j_--K_+(z)j_+)\Vert_{B(\dot{\CE})}\lesssim |{\rm Im} z|^{-1}.
\end{equation*}
This gives \eqref{4.5.1} with $M=M_1+2.$ \eqref{4.5.2} now follows
from Proposition \ref{propbasicreest}.
\qed
\begin{remark}
\label{remlem4.4}
If $\sigma^{\C}_{pp}(\dot{H})=\emptyset$, then we can choose $\delta$
independently of $R$.
\end{remark}
\begin{lemma}
\label{lem4.6}
Let $R\ge M\Vert k\Vert_{\CB(\CH)}$ with $M$ as in Proposition \ref{pr:est}. Then we have 
\begin{equation*}
\Vert \dot{R}(z)\Vert_{\CB(\dot{\CE})}\lesssim |{\rm Im}\, z|^{-1} \quad \text{if } |z|\geq R.
\end{equation*}
\end{lemma}

\proof 

Recall from \eqref{resdotH} that
\begin{eqnarray*}
\dot{R}(z)&:=&(\dot{H}-z)^{-1}=\Phi(k)\left(\begin{array}{cc}
    -p^{-1}(z)(k-z) & p^{-1}(z) \\ 1+(k-z)p^{-1}(z)(k-z) &
    -(k-z)p^{-1}(z) \end{array}\right)\Phi(-k).
\end{eqnarray*}
Therefore it is sufficient to show
\begin{eqnarray}
\label{4.6.1}
\Vert h_0^{1/2}p^{-1}(z)(k-z)u\Vert&\lesssim& \frac{1}{|{\rm Im}z|}\Vert
h_0^{1/2}u\Vert,\\
\label{4.6.2}
\Vert h_0^{1/2} p^{-1}(z)u\Vert &\lesssim& \frac{1}{|{\rm Im}z|}\Vert
u\Vert,\\
\label{4.6.3}
\Vert (k-z)p^{-1}(z)u\Vert&\lesssim& \frac{1}{|{\rm
    Im}z|}\Vert u\Vert,\\
\label{4.6.4}
\Vert (1+(k-z)p^{-1}(z)(k-z))u\Vert&\lesssim& \frac{1}{|{\rm Im}z|}\Vert
h_0^{1/2}u\Vert.
\end{eqnarray}
for $|z|\ge R$. \eqref{4.6.1}-\eqref{4.6.3} follow from Proposition \ref{pr:est}.  To show
\eqref{4.6.4} we use \eqref{eq:pm} and write
\begin{equation*}
1+(k-z)p^{-1}(z)(k-z)=(k-z)p^{-1}(z)h_0^{1/2}h_0^{1/2}(k-z)^{-1}h_0^{-1/2}h_0^{1/2}.
\end{equation*}
Then using \eqref{basicp2}, \eqref{A2} we obtain
\begin{equation*}
\Vert (1+(k-z)p^{-1}(z)(k-z))u\Vert\lesssim (\Vert
k\Vert_{\CB(\CH)}+|z|)\frac{1}{|{\rm Im}z|}\frac{1}{|z|-\Vert
  k\Vert_{\CB(\CH)}}\Vert h_0^{1/2}u\Vert.
\end{equation*}
We can suppose $M\ge 2$ and obtain
\[ (\Vert
k\Vert_{\CB(\CH)}+|z|)\frac{1}{|z|-\Vert
  k\Vert_{\CB(\CH)}}\lesssim 1,\]
which finishes the proof of the lemma.
\qed
  
\subsection{Smooth functional calculus}
The resolvent estimates in Lemma \ref{lemma4.5} easily allow to construct a smooth functional calculus for $\dot{H}$.
%We shall define $f(\dot{H})$ for $f\in C_0^{\infty}(\R)$. For $m\in\N$ set
%\begin{eqnarray*}
%\Vert f\Vert_m:=\sup_{\lambda \in \R,\, \alpha\le m}\vert f^{(\alpha)}(\lambda)\vert.
%\end{eqnarray*}
%Let $\chi\in C_0^{\infty}(\R)$ with $\chi(s)=1$ in $\vert s\vert\le 1.$ Set
%\begin{eqnarray*}
%\tilde{f}(x+iy):=\sum_{n=0}^{\infty}f^{(n)}(x)\frac{(iy)^n}{n!}\chi\left(\frac{y}{\delta
%  C_n}\right),
%\end{eqnarray*}
%with $\delta>0$ and $C_n=\min\{(\Vert
%f^{(n)}\Vert_{\infty})^{-1/n},1\}.$ Then the series converges uniformly and
%\begin{equation*}
%\bar{\partial}\tilde{f}(z)=\sum_{n=0}^{\infty}i^{n+1}f^{(n)}(x)\frac{y^n}{n!\delta
%  C_n}\chi'\left(\frac{y}{\delta
%    C_n}\right).
%\end{equation*}
For $f\in \coinf(\rr)$ we denote by $\tilde{f}\in \coinf(\cc)$ an almost analytic extension of $f$, satisfying 
\begin{eqnarray}
\tilde{f}\vert_{\R}&=&f,\nonumber\\
\left\vert\frac{\partial \tilde{f}(z)}{\partial \bar{z}}\right\vert
&\le& C_N \vert {\rm Im}z\vert^N  \quad N\in \nn.\nonumber 
\end{eqnarray}
%and we have
%\begin{eqnarray}
%f(t)&=&\frac{1}{2\pi i}\int_{\C}\frac{\partial \tilde{f}}{\partial \bar{z}}(z)(t-z)^{-1}dz\wedge d\bar{z},\, t\in \R.
%\end{eqnarray} 
\begin{proposition}
\label{prop16}
Assume \eqref{Hyp1}-\eqref{A2},
\eqref{TE1}-\eqref{TE3}.

(i) Let $f\in C_0^{\infty}(\R)$. Let $\tilde{f}$ be an almost analytic extension of $f$ with 
$\supp\tilde{f}\cap \sigma_{pp}^{\C}(\dot{H})=\emptyset$.  Then the integral 
\begin{eqnarray*}
f(\dot{H}):=\frac{1}{2\pi i}\int_{\C}\frac{\partial \tilde{f}}{\partial \bar{z}}(z)\dot{R}(z)dz\wedge d\bar{z}
\end{eqnarray*}
is norm convergent in $\CB(\dot{\CE})$ and is independent of the choice of $\tilde f$. 
\\
(ii) The map $C_0^{\infty}(\R)\ni f\mapsto f(\dot{H})\in \CB(\dot{\CE})$ is an  continuous algebra morphism if we equip $\coinf(\rr)$ with its canonical topology.
%such that  $f(\dot{H})^*=\bar{f}(\dot{H}^*)$. Moreover, if $\Omega$ is a real compact set
%then there are $C,m>0$ such that for $f$ with $\supp f\subset \Omega$:
%\begin{equation} \label{1.18}
%\Vert f(\dot{H})\Vert_{B(\dot{\CE})} \leq C \Vert f\Vert_{m} . 
%\end{equation}
\end{proposition}

\begin{remark} 
(i) The condition $\supp\tilde{f}\cap
  \sigma^{\C}_{pp}(\dot{H})=\emptyset$  can always be satisfied by choosing $\supp\tilde{f}$ close enough to the real axis.\\[1mm] 
(ii) If $\chi\in C^{\infty}(\R)$ with $\chi= 1$ on $\R\setminus ]-R,R[$ then we define 
\( \chi(\dot{H}):=\one_{\R}(\dot{H})-(1-\chi)(\dot{H}) \).  \\[1mm]
(iii) We define in the same way a smooth functional calculus for $H,
  H_{\pm}, \dot{H}_{\pm}$. For $\dot{H}_{\pm}$ this coincides with the
  smooth functional calculus for selfadjoint operators.
\end{remark}

\begin{proposition}
\label{prop4.9}
If $\sigma^{\C}_{pp}(\dot{H})=\emptyset$ and $\chi\in C_0^{\infty}(\R),\,
\chi= 1$ in a neighborhood of zero, then 
\begin{equation*}
\slim_{L\rightarrow \infty}\chi\left(L^{-1}\dot{H}\right)=1.
\end{equation*}
\end{proposition}
\proof
First note that we have for some $M>0$ the estimate
\begin{equation}
\label{4.5.19}
\Vert\dot{R}(z)\Vert_{\CB(\dot{\CE})}\lesssim \frac{1}{|{\rm
    Im}z|}+\frac{1}{|{\rm Im}z|^M},\quad  |{\rm Im}z|>0.
\end{equation}
Indeed we first choose $R>0$ as in Lemma \ref{lem4.6}. Then we have
\begin{equation*}
\Vert \dot{R}(z)\Vert_{\CB(\dot{\CE})}\lesssim \frac{1}{|{\rm
    Im}z|},\quad \forall z\in \C\setminus B(0,R).
\end{equation*}
By Remark \ref{remlem4.4} we can choose $\delta=R$ in Lemma
\ref{lemma4.5} and obtain
\begin{equation*}
\Vert \dot{R}(z)\Vert_{\CB(\dot{\CE})}\lesssim \frac{1}{|{\rm
    Im}z|^M},\quad \forall z\in B(0,R)\subset\CU_0(R,R).
\end{equation*}
In particular we can choose the same almost analytic extension of
$\chi$ to define $\chi\left(L^{-1}\dot{H}\right)$ for all $L>0$. 
We now show that
\begin{equation}
\label{5.4.14}
\lim_{L\rightarrow\infty}\chi\left(L^{-1}\dot{H}\right)-i_-\chi\left(\frac{\dot{H}_-}{L}\right)i_--i_+\chi\left(\frac{\dot{H}_+}{L}\right)i_+=0.
\end{equation}

We have 
\begin{eqnarray*}
\chi\left(L^{-1}\dot{H}\right)-i_-\chi\left(\frac{\dot{H}_-}{L}\right)i_--i_+\chi\left(\frac{\dot{H}_+}{L}\right)i_+=\frac{1}{2\pi
  i}\int\bar{\partial}\tilde{\chi}(z)L(\dot{R}(Lz)-Q(Lz))dz\wedge
d\bar{z}.
\end{eqnarray*}
Now recall that $\dot{R}(z)=Q(z)(1+K(z))^{-1}$, thus
\begin{equation*}
\dot{R}(Lz)-Q(Lz)=-\dot{R}(Lz)K(Lz).
\end{equation*}
Thanks to \eqref{4.5.19}  we have the estimate for $L\ge 1$
\begin{equation*}
\Vert \bar{\partial}\tilde{\chi}(z)L\dot{R}(Lz)K(Lz)\Vert\lesssim
\frac{1}{L}\rightarrow 0. 
\end{equation*}
This implies \eqref{5.4.14}. As $\dot{H}_{\pm}$ is selfadjoint in
$\dot{\CE}_{\pm}$ we find using Lemma \ref{lemma4.3.1}
\begin{equation*}
\slim_{L\rightarrow\infty}i_{\pm}\chi\left(L^{-1}\dot{H}_{\pm}\right)i_{\pm}=i_{\pm}^2.
\end{equation*}
Thus
\begin{equation*}
\slim_{L\rightarrow\infty}\chi\left(L^{-1}\dot{H}\right)=i_-^2+i_+^2=1.
\end{equation*}
\qed

\section{Propagation estimates}
\label{SecAC1}
In this section we derive resolvent and propagation estimates for $\dot{H}$, similar to those obtained for selfadjoint operators. The key ingredients are the meromorphic extension of $\dot{R}(z)$ in Sect. \ref{SecME} and the fact that the asymptotic Hamiltonians $\dot{H}_{\pm}$ are selfadjoint for their energy norms. There is however a new difficulty not present in the selfadjoint case: in addition to resolvent poles and thresholds, additional spectral singularities may appear. In the theory of selfadjoint operators on Krein spaces used in our previous works  \cite{GGH1, GGH2} these spectral singularities are known as {\em critical points}.
\subsection{Resonances and boundary values of the resolvent}\label{fai}
By the usual arguments the operator 
\begin{eqnarray*}
A_{w}(z)=w^{\epsilon}K_-(z)(1-j_-)w^{-\epsilon}+w^{\epsilon}K_+(z)(1-j_+)w^{-\epsilon}.
\end{eqnarray*} 
can also be extended meromorphically from the upper half plane to
$\{{\rm Im}z>-\delta_{\epsilon/2}\}$ with values in $\CB_{\infty}(\dot{\CE})$. By the same argument as in the construction of the resolvent we have for ${\rm Im}z$ large enough 
\[\Vert A_{w}(z)\Vert_{\CB(\dot{\CE})}\le 1/2.\]
Using Proposition \ref{Prop5} we see that $(1+A_{w}(z))^{-1}$ is meromorphic in $\{{\rm Im}z>-\delta_{\epsilon/2}\}$. Let $S_{w}$ be the set of its poles. Now we have~:
\begin{eqnarray}
\label{Resweight-x}
w^{-\epsilon}\dot{R}(z)w^{-\epsilon}
&=&w^{-\epsilon}Q(z)w^{-\epsilon}(1+A_{w}(z))^{-1}\nonumber\\
&\times&(1-w^{\epsilon}K_-(z)j_-w^{-\epsilon}-w^{\epsilon}K_+(z)j_+w^{-\epsilon}).
\end{eqnarray}
Using \eqref{Resweight-x} we see that
$w^{-\epsilon}\dot{R}(z)w^{-\epsilon}$ can be extended meromorphically from the upper half
plane to $\{{\rm Im}z>-\delta_{\epsilon/2}\}$ with values in $\CB_{\infty}(\dot{\CE})$.  The same result holds also for the resolvents of $\dot{H}_{\pm}$, by assumption  (\eqref{TE3}){\it c)}.
\begin{definition}\label{faila}
\begin{enumerate}
 \item[i)] the poles in $\{{\rm Im}z\leq 0\}$ of  the meromorphic extension of $w^{-\epsilon}\dot{R}(z)w^{-\epsilon}$ are called {\em resonances} of $\dot{H}$. 
 \item[ii)] the set of {\em real resonances} of $\dot{H}$, resp. $\dot{H}_{\pm}$ is  denoted by $\CT$, resp. $\CT_{\pm}$.
 \end{enumerate}
\end{definition} 
Note that $\CT$, $\CT_{\pm}$ are obviously closed, discrete sets.
As a consequence of the meromorphic
extensions of $w^{-\epsilon}\dot{R}(z)w^{-\epsilon}$ and  $w^{-\epsilon}\dot{R}_{\pm}(z)w^{-\epsilon}$ we obtain:
\begin{proposition}
\label{prop5.1}
Assume \eqref{Hyp1}-\eqref{A2}, \eqref{TE1}-\eqref{TE3}. Let $\epsilon>0$. 
\begin{enumerate}
\item[--] there exists $\nu>0$ such that for
all 
$\chi\in C_0^{\infty}(\R\setminus \CT)$ and all $k\in \N$  we have 
\begin{equation}
\label{5.1}
\sup_{\nu\ge \delta>0,\, \lambda\in \R}\Vert
w^{-\epsilon}\chi(\lambda)\dot{R}^k(\lambda\pm i\delta)w^{-\epsilon}\Vert_{\CB(\dot{\CE})}<\infty.
\end{equation}
\item[--] 
 for
all 
$\chi\in C_0^{\infty}(\R\setminus \CT_{\pm})$ and all $k\in \N$  we have 
\begin{equation}
\label{5.1pm}
\sup_{\delta>0,\, \lambda\in \R}\Vert
w^{-\epsilon}\chi(\lambda)\dot{R}_{\pm}^k(\lambda\pm i\delta)w^{-\epsilon}\Vert_{\CB(\dot{\CE}_{\pm})}<\infty.
\end{equation}
\end{enumerate}
\end{proposition}
We apply \cite[Thm. 4.3.1]{Ya} to obtain:
\begin{corollary}
\label{corollary5.2}
Assume \eqref{Hyp1}-\eqref{A2}, \eqref{TE1}-\eqref{TE3}. Let $\epsilon>0$ and $\CT_{\pm}$ be as in Proposition
\ref{prop5.1} $ii)$. Then we have for all $\chi\in
C_0^{\infty}(\R\setminus\CT_{\pm})$
\begin{equation}
\label{5.2}
\sup_{\Vert u\Vert_{\dot{\CE}_{\pm}}=1,\, \delta\neq 0}\int_{\R}(\Vert
w^{-\epsilon}\dot{R}_{\pm}(\lambda+i\delta)\chi(\lambda)u\Vert^2_{\dot{\CE}_{\pm}}d\lambda<\infty.
\end{equation}
\end{corollary}
Note that we can not apply directly \cite[Thm. 4.3.1]{Ya} to
$\dot{H}$, because the self-adjointness of the operator is crucial in
this theorem.  To discuss this further let is introduce a definition.
\begin{definition}
We call $\lambda\in \R$ a {\em regular point} of $\dot{H}$ if there exists
$\chi\in C_0^{\infty}(\R),\, \chi(\lambda)=1$  and $\nu>0$ such that:
\begin{equation}
\label{regpoints}
\sup_{\Vert u\Vert_{\dot{\CE}_{\pm}}=1,\, \nu>|\delta|> 0}\int_{\R}(\Vert
w^{-\epsilon}\dot{R}(\lambda+i\delta)\chi(\lambda)u\Vert^2_{\dot{\CE}_{\pm}}d\lambda<\infty.
\end{equation}

 Otherwise we call it a {\em singular point}. We denote by $\mathcal{S}$ the set of singular points of $\dot{H}$.
\end{definition}
\begin{remark}\label{remdeb}
 Denoting by $\mathcal{S}_{\pm}$ the analog of $\mathcal{S}$ for $\dot{H}_{\pm}$ we see that  $\CS_{\pm}=\CT_{\pm}$ by Kato's
theory of $H$-smoothness (see \cite{Ya}).
\end{remark}

In our situation  it is still possible to control the set of singular points. Recall that
\begin{equation*}
Q(z)=(1-\tilde{K}_-(z)-\tilde{K}_+(z))\dot{R}(z).
\end{equation*}
We then have
\begin{equation*}
w^{-\epsilon}Q(z)=(1-w^{-\epsilon}\tilde{K}_-(z)w^{\epsilon}-w^{-\epsilon}\tilde{K}_+(z)w^{\epsilon})w^{-\epsilon}\dot{R}(z).
\end{equation*}
Let 
\[\tilde{A}_{w}(z):=-w^{-\epsilon}\tilde{K}_-(z)w^{\epsilon}-w^{-\epsilon}\tilde{K}_+(z)w^{\epsilon}.\]
By the usual arguments $\tilde{A}_w$ is meromorphic  in $\{{\rm
  Im}z>-\delta_{\epsilon/2}\}$ with values in $\CB_{\infty}(\dot{\CE})$. Also $\Vert
\tilde{A}_w(z)\Vert_{\CB(\dot{\CE})}\le 1/2$ for ${\rm Im}z$
sufficiently large. We can therefore apply again Proposition
\ref{Prop5} to see that $(1+\tilde{A}_w(z))^{-1}$ is meromorphic for
$\{{\rm Im z}>-\delta_{\epsilon/2}\}$. We then have
\begin{equation}
\label{PE2.3}
w^{-\epsilon}\dot{R}(z)=(1+\tilde{A}_w(z))^{-1}w^{-\epsilon}Q(z).
\end{equation}
\begin{proposition}
\label{propPE2}
Assume \eqref{Hyp1}-\eqref{A2}, \eqref{TE1}-\eqref{TE3}. Let $\N_w$ be the set of real poles $\tilde{A}_w(z)$. Then 
\[
\mathcal{S}\subset \tilde{N}_{w}\cup \CT_{+}\cup \CT_{-}.
\]
 It follows that $\mathcal{S}$ is a closed and discrete set.

\end{proposition}
\proof
The estimate (\ref{regpoints}) with $\chi(\lambda)$ instead of $\chi(\dot{H})$ follows from
\eqref{PE2.3} and Corollary \ref{corollary5.2}. We therefore only have to show that we can replace
$\chi(\lambda)$ by $\chi(\dot{H})$. We choose $\tilde{\chi}\in C_0^{\infty}(I)$ with 
$\tilde{\chi}\chi=\chi$ and  write~:
\begin{align}
\label{estwithchi}
\Vert w^{-\epsilon}\dot{R}(\lambda\pm i\delta) \chi(\dot{H})  f \Vert^2_{\dot{\CE}}
&\lesssim \Vert w^{-\epsilon}\dot{R}(\lambda\pm i\delta)\tilde{\chi}(\lambda)\chi(\dot{H})  f\Vert^2_{\dot{\CE}}\nonumber\\
+& \Vert w^{-\epsilon}\dot{R}(\lambda\pm i\delta)(1-\tilde{\chi}(\lambda))\chi(\dot{H}) f\Vert^2_{\dot{\CE}}.
\end{align}
The estimate for the first term follows from the estimate with
$\chi(\lambda)$. Let us treat the second term. We claim
\begin{eqnarray*}
\Vert
w^{-\epsilon}\dot{R}(\lambda\pm\i\delta)(1-\tilde{\chi}(\lambda))\chi(\dot{H})
\Vert_{\CB(\dot{\CE})}\lesssim \langle\lambda\rangle^{-1},
\end{eqnarray*}
uniformly in $\delta$. In fact let
\begin{eqnarray*}
f^{\epsilon}_{\lambda}(x)=\langle\lambda\rangle\frac{1}{x-(\lambda+\i\delta)}(1-\tilde{\chi}(\lambda))\chi(x).
\end{eqnarray*}
It is sufficient to show that all the semi-norms $\Vert f^{\epsilon}_{\lambda}\Vert_m$ are uniformly bounded with respect to $\lambda,\, \delta$. Note that $g_{\lambda}(x)=(1-\tilde{\chi}(\lambda))\chi(x)$ vanishes to all orders at $x=\lambda$. If $\supp\chi\subset [-C,C]$ this is enough to ensure that $\Vert f^{\epsilon}_{\lambda}\Vert_m$ is uniformly bounded in $\lambda\in [-2C,2C]$ and $\delta>0$. For $\vert\lambda\vert\ge 2C$ we observe that 
\begin{eqnarray*}
\left\vert\langle\lambda\rangle\frac{1}{x-(\lambda+\i\delta)}\right\vert\lesssim 1
\end{eqnarray*}
with analogous estimates for the derivatives. This gives the
integrability of the second term in
\eqref{estwithchi}. 
\qed

\subsection{Propagation estimates}
As an immediate consequence of Proposition \ref{prop5.1} we obtain :
\begin{proposition}
\label{prop5.2}
Assume \eqref{Hyp1}-\eqref{A2}, \eqref{TE1}-\eqref{TE3}. Let $\epsilon>0$. 
\begin{enumerate}
\item[--]
for all $\chi\in C_0^{\infty}(\R\setminus \CT)$ and  $k\in \N$  we have 
\begin{equation}
\label{5.1a)pm}
\Vert w^{-\epsilon}e^{-it\dot{H}}\chi(\dot{H})w^{-\epsilon}\Vert_{\CB(\dot{\CE})}\lesssim
\<t\>^{-k}.
\end{equation}
\item[--] 
 for all  $\chi\in C_0^{\infty}(\R\setminus \CT_{\pm})$ and $k\in \N$  we
have 
\begin{equation}
\label{5.1a)}
\Vert w^{-\epsilon}e^{-it\dot{H}_{\pm}}\chi(\dot{H}_{\pm})w^{-\epsilon}\Vert_{\CB(\dot{\CE}_{\pm})}\lesssim
\<t\>^{-k}.
\end{equation}
\end{enumerate}
\end{proposition}
\proof

We only prove  $i)$, the proof of  $ii)$ being analogous. We have 
\begin{equation*}
w^{-\epsilon}e^{-it\dot{H}}\chi(\dot{H})w^{-\epsilon}=\frac{1}{2\pi
  i}\int\chi(\lambda)e^{-it\lambda}w^{-\epsilon}(\dot{R}(\lambda+i0)-\dot{R}(\lambda-i0))w^{-\epsilon}d\lambda.
\end{equation*}
Integration by parts gives :
\begin{equation*}
w^{-\epsilon}e^{-it\dot{H}}\chi(\dot{H})w^{-\epsilon}=\frac{1}{2\pi
  i}\frac{1}{(it)^k}\sum_{\pm}\sum_{j=1}^{k+1}\pm
C^{j-1}_k\int\chi_j(\lambda)e^{-it\lambda}w^{-\epsilon}\dot{R}^j(\lambda\pm
i0)w^{-\epsilon}d\lambda
\end{equation*}
with $\chi_j=\chi^{(k+1-j)}$. The estimate then follows from
Proposition \ref{prop5.1}.
\qed 

\begin{proposition}
\label{propPE3}
Assume \eqref{Hyp1}-\eqref{A2}, \eqref{TE1}-\eqref{TE3}. Let
$\epsilon>0$. Then we have for all $\chi\in
C_0^{\infty}(\R\setminus\CS)$:
\begin{equation}
\label{5.10}
\int_{\R}\Vert
w^{-\epsilon}e^{-it\dot{H}}\chi(\dot{H})\varphi\Vert^2_{\dot{\CE}}dt\lesssim
\Vert\varphi\Vert^2_{\dot{\CE}}.
\end{equation}
\end{proposition}
\proof
We write~:
\begin{eqnarray*}
w^{-\epsilon}(\dot{R}(\lambda+\i\delta)-\dot{R}(\lambda-\i\delta))\chi(\dot{H})f=i\int_{\R}w^{-\epsilon}\e^{-\delta\vert
  t\vert}\e^{\i \lambda t}\e^{-\i \dot{H}t}\chi(\dot{H}) f dt.
\end{eqnarray*}
By Plancherel's formula this yields: 
\begin{equation*}
\int_{\R}\Vert
  w^{-\epsilon}(\dot{R}(\lambda+\i\delta)-\dot{R}(\lambda-\i\delta))\chi(\dot{H})
  f\Vert^2_{\dot{\CE}}d\lambda
  =\int_{\R}\e^{-2\delta \vert t\vert}\Vert  w^{-\epsilon}\e^{-\i t\dot{H}}\chi(\dot{H}) f\Vert^2_{\dot{\CE}}dt.
\end{equation*}
The lhs of this equation is uniformly bounded in $\delta$ with
$\delta$ small enough.
\qed
\begin{corollary}
If \eqref{Hyp1}-\eqref{A2}, \eqref{TE1}-\eqref{TE3} hold and
$\lambda$ is a real eigenvalue of $\dot{H}$ then $\lambda\in\CS$.
\end{corollary}
%\proof
%Suppose that $\lambda\in \R$ is an eigenvalue and that $\lambda\in
%\R\setminus \CS$. Let $\chi\in C_0^{\infty}(\R),\,
%\chi(\lambda)=1$ such that \eqref{5.10} holds. Let $u$ be an eigenfunction
%with eigenvalue $\lambda$. Then we have 
%\begin{equation*}
%\int_{\R}\Vert w^{-\epsilon}e^{-it\dot{H}}\chi(\dot{H})u\Vert^2_{\dot{\CE}}dt=\int_{\R}\Vert w^{-\epsilon}\chi(\lambda)u\Vert^2_{\dot{\CE}}dt=\infty,
%\end{equation*}
% which is a contradiction.
%\qed}
\subsection{Estimates on singular points}
\label{sec5.3}
It will be important in applications  to prove that $\dot{H}$ has no singular points.  To do this we will use the following proposition.
\begin{proposition}
\label{thm5.3.1}
Assume \eqref{Hyp1}-\eqref{A2}, \eqref{TE1}-\eqref{TE3}. Then
\[
\CS\subset \CT\cup \CT_{-}\cup \CT_{+}.
\]
\end{proposition} 
\proof  From (\ref{RB}) we obtain for ${\rm Im}z\gg 1$: 
\[
\dot{R}(z)= Q(z)(\one + K(z))^{-1}= Q(z)- Q(z)(\one + K(z))^{-1}K(z),
\]
hence
\begin{align*}
 w^{-\epsilon}\dot{R}(z)=& w^{-\epsilon}Q(z)- w^{-\epsilon}Q(z)(\one + K(z))^{-1}w^{-\epsilon}w^{\epsilon}K(z)\\
 =&w^{-\epsilon}Q(z)- w^{-\epsilon}\dot{R}(z)w^{-\epsilon}w^{\epsilon}K(z).
\end{align*}
Next we write $w^{\epsilon}K(z)= w^{\epsilon}K_{-}(z)+ w^{\epsilon}K_{+}(z)$ and obtain from  the expression (\ref{Kpm}) of $K_{\pm}(z)$ that $w^{\epsilon}K_{\pm}(z)= m_{\epsilon}\dot{R}_{\pm}(z)i_{\pm}$ for $m_{\epsilon}\in B(\dot{\CE})$.   It suffices then to recall the expression (\ref{defdeQ}) of $Q(z)$,  and apply Prop. \ref{prop5.2} and Remark \ref{remdeb}. \qed

\subsection{Additional resolvent estimates}
In this subsection we  make the link between the poles of $\eta p^{-1}(z)\eta$ and those of 
$\eta\dot{R}(z)\eta$ for $\eta\in C_0^{\infty}(\R)$. 

We will  need the following hypothesis:
\begin{equation}
\tag{PE} \label{PE2}
\left\{\begin{array}{ll} 
a) & \psi\in C_0^{\infty}(\R) \Rightarrow h_0^{1/2}\psi(x)h_0^{-1/2}\in \CB(\CH),  \\
b) & \psi\in C_0^{\infty}(\R), \, \psi\geq0, \, \psi= 1 \text{ near 0 } \Rightarrow
\slim_{n\rightarrow \infty} \psi\left(\frac{x}{n}\right)= 1 \text{ in } h_0^{-1/2}\CH.
%\mbox{If in addition $\psi= 1$ in a neighborhood of $0,\,
%  \psi\ge 0,$ then}\\
%& \slim_{n\rightarrow \infty}
%\psi\left(\frac{x}{n}\right)=\one \quad\mbox{in}\quad h_0^{-1/2}\CH. 
\end{array}\right.
\end{equation}
\begin{lemma}
\label{propresestlowfr}
Let $\eta,\tilde{\eta}\in C_0^{\infty}(\R)$ with
$\tilde{\eta}\eta=\eta$. If $z$ is not a pole of $\tilde{\eta}p^{-1}(z)\tilde{\eta}$
then $z$ is not a pole of $\eta R(z)\eta$ nor of $\eta\dot{R}(z)\eta$ and if
$\CP(z):=\Vert\tilde{\eta}p^{-1}(z)\tilde{\eta}\Vert_{\CB(\CH)}$ then we have the
estimates
\begin{eqnarray}
\label{10.2.1}
\Vert \eta R(z)\eta\Vert_{\CB({\CE})}\lesssim \<z\>^2(1+\<z\>\CP^2(z)), \\
\label{10.2.2}
\Vert \eta \dot{R}(z)\eta\Vert_{\CB({\dot{\CE})}}\lesssim  \<z\>^2(1+\<z\>\CP^2(z)).
\end{eqnarray}
\end{lemma}
%\begin{remark}
%The estimates \eqref{10.2.1}-\eqref{10.2.2} are very rough and can be
%improved, but the proof becomes much longer. For our purposes
%\eqref{10.2.1}-\eqref{10.2.2} are sufficient.
%\end{remark}
\proof
We choose functions $\eta_1,\, \eta_0\in C_0^{\infty}(\R)$ with
$\eta_0\eta=\eta,\, \eta_1\eta_0=\eta_0$ and $\tilde{\eta}\eta_1=\eta_1.$
We first notice that \eqref{10.2.2} follows from \eqref{10.2.1} because
\begin{eqnarray*}
\Vert \eta\dot{R}(z)\eta u\Vert_{\dot{\CE}}\lesssim \Vert \eta
R(z)\eta u\Vert_{\CE}\lesssim \<z\>^2(1+\<z\>\CP^2(z))\Vert \eta_0
u\Vert_{\CE}\lesssim \<z\>^2(1+\<z\>\CP^2(z))\Vert u\Vert_{\dot{\CE}}, 
\end{eqnarray*}
where we have used Hardy's inequality, \eqref{TE3}{\it f)}. Now recall from \eqref{resdotH} that
\begin{equation*}
\dot{R}(z):=(\dot{H}-z)^{-1}=\Phi(k)\left(\begin{array}{cc}
    -p^{-1}(z)(k-z) & p^{-1}(z) \\ 1+(k-z)p^{-1}(z)(k-z) &
    -(k-z)p^{-1}(z) \end{array}\right)\Phi(-k).
\end{equation*}
It is therefore sufficient to show:
\begin{eqnarray}
\label{10.2.2a}
\Vert \eta p^{-1}(z)(k-z)\eta u\Vert_{\CH^1}&\lesssim& \<z\>(1+\<z\>^2\CP(z))\Vert
u\Vert_{\CH^1},\\
\label{10.2.3}
\Vert \eta p^{-1}(z)\eta u\Vert_{\CH^1}&\lesssim&(1+\<z\>^2\CP(z))\Vert
u\Vert_{\CH},\\
\label{10.2.4}
\Vert \eta(1+(k-z)p^{-1}(z)(k-z)\eta
u\Vert_{\CH}&\lesssim&\<z\>(1+\<z\>^2\CP^2(z))\Vert u\Vert_{\CH^1},\\
\label{10.2.5}
\Vert \eta (k-z)p^{-1}(z)\eta u\Vert_{\CH}&\lesssim&\<z\>\CP(z)\Vert
u\Vert_{\CH}.
\end{eqnarray}
\eqref{10.2.5} is clear, let us consider \eqref{10.2.3}. By complex
interpolation \eqref{10.2.3} will follow from 
\begin{equation}
\label{10.2.7} 
\Vert \eta p^{-1}(z)\eta u\Vert_{\CH^2}\lesssim (\<z\>^2\CP(z)+1)\Vert u\Vert_{\CH}.
\end{equation}
We compute
\begin{align*}
h_0\eta p^{-1}(z)\eta=&[h_0,\eta]p^{-1}(z)\eta+\eta h_0p^{-1}(z)\eta\\
=&(h_0+1)^{-1}[h_0,\eta]\eta_0(h_0+1)p^{-1}(z)\eta\\
&+(h_0+1)^{-1}[h_0,[h_0,\eta]]\eta_0p^{-1}(z)\eta
+\eta h_0p^{-1}(z)\eta,\\
\eta_0 h_0p^{-1}(z)\eta=&\eta+(k-z)^2\eta_0 p^{-1}(z)\eta.
\end{align*}
Thus we have 
\begin{eqnarray*}
\Vert
\eta_0h_0 p^{-1}(z)\eta u\Vert&\lesssim&(1+\<z\>^2\CP(z))\Vert
u\Vert_{\CH}.
\end{eqnarray*}
Using that $(h_0+1)^{-1}[h_0,[h_0,\eta]]$ is bounded this gives \eqref{10.2.7} and thus \eqref{10.2.3}. Let us now consider
\eqref{10.2.2a}. First note that 
$\Vert (k-z)u\Vert_{\CH^1}\lesssim\<z\>\Vert u\Vert_{\CH^1}$. 
We then estimate using \eqref{10.2.3}
\begin{eqnarray*}
\Vert \eta p^{-1}(z)\eta u\Vert_{\CH^1}&\lesssim& (\<z\>^2\CP(z)+1)\Vert
u\Vert_{\CH^1}.
\end{eqnarray*}
This gives \eqref{10.2.2a}. Let us now show \eqref{10.2.4}. We write
\begin{align*}
&\eta (1+(k-z)p^{-1}(z)(k-z))\eta\\= & \eta p^{-1}(z)\eta_1[h_0,k\eta_0]h_0^{-1/2}h_0^{1/2}\eta_1p^{-1}(z)(k-z)\eta 
 + \eta p^{-1}(z)(k-z)^2\eta.
\end{align*}
We have using \eqref{10.2.2a} :
\begin{align*}
&\Vert \eta
p^{-1}(z)\eta_1[h_0,k\eta_0]h_0^{-1/2}h_0^{1/2}\eta_1p^{-1}(z)(k-z)\eta
u\Vert_{\CH}\\
\lesssim& \CP(z)\Vert
h_0^{1/2}\eta_1p^{-1}(z)(k-z)\eta\Vert_{\CB(\CH)}\Vert u\Vert_{\CH}\\
\lesssim& \<z\>\CP(z)(1+\<z\>^2\CP(z))\Vert u\Vert_{\CH^1}.
\end{align*}
This proves \eqref{10.2.4}.
\qed

\begin{corollary} \label{Cor10.2.4}
If $w^{-\epsilon}p^{-1}(z)w^{-\epsilon}$ has no real poles then $w^{-\epsilon}\dot{R}(z)w^{-\epsilon}$ has no real poles.
\end{corollary}
\proof
By the preceding lemmas $\eta\dot{R}(z)\eta$ has no real poles for all
$\eta\in C_0^{\infty}(\R)$.
Suppose that $w^{-\epsilon}\dot{R}(z)w^{-\epsilon}$ has a pole at
$z=z_0\in \R$. In a neighborhood of $z=z_0$ we have
\begin{equation*}
w^{-\epsilon}\dot{R}(z)w^{-\epsilon}=\sum_{j=1}^m\frac{A_j}{(z-z_0)^j}+H(z),
\end{equation*}
where $H(z)$ is holomorphic and $A_j$ are of finite rank. Let
$\eta_1,\, \eta_2\in C_0^{\infty}(\R)$. We have 
\begin{equation*}
w^{-\epsilon}\eta_1\dot{R}(z)\eta_2w^{-\epsilon}=\sum_{j=1}^m\frac{\eta_1A_j\eta_2}{(z-z_0)^j}+\eta_1H(z)\eta_2.
\end{equation*} 
As $\eta_1\dot{R}(z)\eta_2$ doesn't have a pole at $z=z_0$, we have
\begin{equation*}
\eta_1A_j\eta_2=0,\quad \forall \eta_1,\eta_2\in C_0^{\infty}(\R),\,
j=1,...,m.
\end{equation*}
It follows 
\begin{equation*}
A_j\eta=0,\quad \forall \eta\in C_0^{\infty}(\R),\,
j=1,...,m.
\end{equation*}
%Suppose $A_j\neq 0$. Then there exists $u\in \dot{\CE}$ such that
%$A_ju\neq 0$. Let $\psi\in C_0^{\infty},\, \psi\ge 0,\, \psi= 1$
%in a neighborhood of zero. We have (see the proof of Lemma \ref{lemdens1}):
%\[0=\lim_{n\rightarrow \infty}
%A_j\psi\left(\frac{x}{n}\right)u=A_ju\neq 0,\]
%which is a contradiction.
Using \eqref{PE2} this implies that $A_{j}=0$.
\qed
\section{Boundedness of the evolution 1 : abstract setting}
\label{secbound1}
The aim of this section is to show that the evolution is bounded
outside the complex eigenvalues and the singular points of $\dot{H}$. 
We assume
\begin{equation}
\tag{B} \label{Hyp15}
w^{-1}: \ D(h_{0})\to D(h_{0}), \ 
[-ik,h]\lesssim w^{-1}h_0w^{-1}\mbox{ as quadratic forms on }D(h_0).
\end{equation}
For $\chi\in C^{\infty}(\R)$ and $\mu>0$ we put
$\chi_{\mu}(\lambda)=\chi\left(\frac{\lambda}{\mu}\right).$
\begin{theorem}
\label{thboundedness}
Assume \eqref{Hyp1}, \eqref{A2},
\eqref{TE1}-\eqref{TE3}, \eqref{PE2}, \eqref{Hyp15}. Assume furthermore
$\sigma_{pp}^{\C}(\dot{H})=\emptyset$. Then 
\begin{enumerate}
\item Let $\chi\in C^{\infty}(\R)$ with $\chi= 0$ on $[-1,1]$ and $\chi= 1$
outside $[-2,2]$. Then there are $\mu_0>0, C_1>0$ such that for $\mu\ge \mu_0$,
 and $t\in \R$:
\begin{equation}
\label{unifboundhfr}
\Vert e^{-it\dot{H}}\chi_{\mu}(\dot{H})u\Vert_{\dot{\CE}} \le C_1\Vert
\chi_{\mu}(\dot{H})u\Vert_{\dot{\CE}}, \ u\in \dot{\CE}.
\end{equation}
\item
If $\chi\in
C_0^{\infty}(\R\setminus \CS)$ then there is $C_2>0$ such
that for all $u\in \dot{\CE}$ and $t\in \R$ we have
\begin{equation}
\Vert e^{-it\dot{H}}\chi(\dot{H})u\Vert_{\dot{\CE}}\le C\Vert
u\Vert_{\dot{\CE}}.
\end{equation}
\end{enumerate}
\end{theorem}
\begin{remark}
If $\sigma_{pp}^{\C}(\dot{H})\neq\emptyset$, then  the theorem still holds
for $e^{-it\dot{H}}|_{\CE_{\R}(\dot{H})}$.
\end{remark}
The proof of the Thm. \ref{thboundedness} is divided into a high (part $i)$) and a low
frequency analysis (part $ii)$).

\subsection{High frequency analysis}

\begin{lemma} \label{unifbound1}
Assume \eqref{Hyp1}, \eqref{A2}, \eqref{TE1}-\eqref{TE3}, \eqref{PE2}, \eqref{Hyp15}. 
If $\chi$ is as in the statement of Thm. \ref{thboundedness} then for $\mu>0$ sufficiently
large we have:
\begin{equation*}
\Vert (\chi_{\mu}(\dot{H})u)_0\Vert_{\cH}\lesssim
  \frac{1}{\mu}\Vert \chi_{\mu}(\dot{H}) u \Vert_{\dot{\CE}}.
\end{equation*}
\end{lemma}
\proof
Let $\hat{\chi}$ be as  $\chi$ with  $\hat{\chi}\chi=\chi$. Set
 $\varphi=\hat{\chi}-1$ and observe that $\varphi= -1$ on $(-1,1)$. 
Let $\tilde{\varphi}$ be some (finite order) almost analytic extension of $\varphi$ given by ($N\ge 1$):
\[
\tilde{\varphi}(x+iy)= \sum_{r=0}^N\varphi^{(r)}(x)\frac{(iy)^r}{r!}
\tau\left(\frac{y}{\delta\<x\>}\right)
\]
with $\tau\in C_0^{\infty}(\R),\, \tau(s)=1$ in $|s|\le 1/2,\,
\tau(s)=0$ in $|s|\ge 1.$ Here $\delta$ is chosen such that
$\dot{R}(z)$ has no poles in $|{\rm Im}z|\le \delta\<x\>,\, x\in \supp
\varphi$.  We compute
\begin{align*}
\bar{\partial}\tilde{\varphi}(z)=&\hat{\chi}^{(N+1)}(x)\frac{(iy)^{(N+1)}}{(N+1)!}\tau\left(\frac{y}{\delta \<x\>}\right)+\left(\sum_{r=0}^N\varphi^{(r)}(x)\frac{(iy)^r}{r!}\right)\tau'\left(\frac{y}{\delta\<x\>}\right)\left(\frac{i}{\delta\<x\>}
+\frac{yx}{\delta\<x\>^2}\right)\\
=:&\tilde{\varphi}_1(x+iy)+\tilde{\varphi}_2(x+iy).
\end{align*}
Let $\mu\ge \mu_0= \max\{(1+\varepsilon) \Vert
k\Vert_{\CB(\CH)},\frac{2(1+\varepsilon)}{\delta}\Vert
k\Vert_{\CB(\CH)}\}$. We then have $\supp \bar{\partial}\tilde{\varphi}\subset K:=\{z\in \C;\, |\mu z|\ge
  (1+\varepsilon)\Vert k\Vert_{\CB(\CH))}\}\cap\{z\in \C;\, |z|\ge \min\{1,\frac{1}{2}\delta\}\}.$ Indeed on $\supp
  \tilde{\varphi}_1$ we have $|z|\ge 1$ and thus
\[|\mu z|\ge \mu_0= (1+\varepsilon)\Vert
k\Vert_{\CB(\CH)} .\]
On $\supp \tilde{\varphi}_2$ we have $|z|\ge \frac{\delta}{2}|z|$ and thus 
\[|\mu z|\ge \mu \frac{\delta}{2}\ge (1+\varepsilon)\Vert
k\Vert_{\CB(\CH)}.\]
Note that 
\begin{equation*}
1=-(\chi-1)(0)=\frac{1}{2\pi
  i}\int\bar{\partial}\tilde{\varphi}(z)\frac{1}{z}dz\wedge d\bar{z}.
\end{equation*}
We have
\begin{eqnarray}
\label{chimuH}
\hat{\chi}_{\mu}(\dot{H})&=&\varphi_{\mu}(\dot{H})+1\nonumber\\
&=&\frac{1}{2\pi
  i}\int\bar{\partial}\tilde{\varphi}(z)\left(\left(\frac{\dot{H}}{\mu}-z\right)^{-1}+\frac{1}{z}\right)dz\wedge
d\bar{z}\nonumber\\
&=&-\frac{1}{2\pi i}\int\bar{\partial}\tilde{\varphi}(z)(\dot{H}-\mu
z)^{-1}\frac{\dot{H}}{z}dz\wedge d\bar{z}.
\end{eqnarray}
Let $v^{\mu}=\chi_{\mu}(\dot{H})u$. We compute
\begin{equation*}
\left((\dot{H}-\mu z)^{-1}\frac{\dot{H}}{z}\left(\begin{array}{c} v_0^{\mu}\\
    v_1^{\mu}\end{array}\right)\right)_0=\frac{1}{z}p^{-1}(\mu z)(\mu z
v^{\mu}_1+hv^{\mu}_0).
\end{equation*}
We estimate for $z\in \supp \bar{\partial}\tilde{\varphi}(z)$ using
Proposition \ref{pr:est}:
\begin{eqnarray*}
\Vert p^{-1}(\mu z)z\mu v_1^{\mu}\Vert_{\cH}&\lesssim& \Vert p^{-1}(\mu
z)z\mu (v_1^{\mu}-kv_0^{\mu})\Vert_{\cH}
+\Vert p^{-1}(\mu
z)z\mu kv_0^{\mu}\Vert_{\cH}\\
&\lesssim&\frac{1}{|{\rm Im}z|\mu}\Vert
(v_1^{\mu}-kv_0^{\mu})\Vert_{\cH}+\frac{1}{\mu|{\rm Im}z|}\Vert v_0^{\mu}\Vert_{\cH},\\
\Vert p^{-1}(\mu z) hv_0^{\mu}\Vert_{\cH}&\lesssim& \Vert
p^{-1}(z)h_0v_0^{\mu}\Vert_{\cH}+\Vert p^{-1}(\mu z)k^2v_0^{\mu}\Vert_{\cH}\\
&\lesssim& \frac{1}{|{\rm Im}z|\mu}\Vert
h_0^{1/2}v^{\mu}_0\Vert_{\cH}+\frac{1}{\mu^2|{\rm Im}z|}\Vert v_0^{\mu}\Vert_{\cH}\\
&\lesssim& \frac{1}{|{\rm Im}z|\mu}\Vert
h_0^{1/2}v^{\mu}_0\Vert_{\cH}+\frac{1}{\mu^2|{\rm Im}z|}\Vert v_0^{\mu}\Vert_{\cH}
\end{eqnarray*}
Using \eqref{chimuH} we obtain
\begin{equation*}
\Vert (\chi_{\mu}(\dot{H}^n)u)_0\Vert_{\cH}\lesssim \frac{1}{\mu}\Vert
\chi_{\mu}(\dot{H}^n) u\Vert_{\dot{\CE}}+\frac{1}{\mu^2}\Vert
(\chi_{\mu}(\dot{H}^n)u)_0\Vert_{\cH}.
\end{equation*}
This gives the lemma for $\mu$ sufficiently large. 
\qed

\begin{corollary}
\label{corunifbound1}
Assume \eqref{Hyp1}, \eqref{A2},
\eqref{TE1}-\eqref{TE3}, \eqref{PE2}, \eqref{Hyp15}. Let $\chi$ be as in
Thm.  \ref{thboundedness}. Then for $\mu>0$ sufficiently large there exists
$\varepsilon>0$ such that for all $u\in \CE$:
\begin{equation*}
\<\chi_{\mu}(\dot{H})u, \chi_{\mu}(\dot{H})u\> _0\ge \varepsilon \Vert
\chi_{\mu}(\dot{H})u\Vert^2_{\dot{\CE}}.
\end{equation*}
\end{corollary}
\proof 
By Lemma \ref{unifbound1} we have 
\begin{eqnarray*}
\<\chi_{\mu}(\dot{H})u, \chi_{\mu}(\dot{H})u\>_0&\ge& \Vert \chi_{\mu}(\dot{H})u\Vert_{\dot{\CE}}^2-2\Vert
k(\chi_{\mu}(\dot{H})u)_0\Vert_{\cH}^2\\
&\ge&(1-\frac{C}{\mu^2})\Vert (\chi_{\mu}(\dot{H})u)\Vert_{\dot{\CE}}^2,
\end{eqnarray*}
which gives the corollary for $\mu$ sufficiently large.
\qed

\begin{corollary} \label{corunifbound2}
Assume \eqref{Hyp1}, \eqref{A2},
\eqref{TE1}-\eqref{TE3}, \eqref{PE2}, \eqref{Hyp15}. Let $\chi$ be as in
Thm. \ref{thboundedness}. 
Then there exists $C_1>0$ such that for all $u\in \dot{\CE},\,  t\in \R$ 
\begin{equation*}
\Vert e^{-it\dot{H}}\chi_{\mu}(\dot{H})u\Vert_{\dot{\CE}}\le C_1\Vert
\chi_{\mu}(\dot{H}) u\Vert_{\dot{\CE}}.
\end{equation*}
\end{corollary}
\proof
We use that $\<e^{-it\dot{H}}\chi_{\mu}(\dot{H})u, e^{-it\dot{H}}\chi_{\mu}(\dot{H})u\>_{0}$ is
conserved.  By Corollary \ref{corunifbound1} we have
\begin{eqnarray*}
\Vert e^{-itH}\chi_{\mu}(\dot{H})u\Vert^2_{\dot{\CE}}&\lesssim& \<
e^{-it\dot{H}}\chi_{\mu}(\dot{H})u, e^{-it\dot{H}}\chi_{\mu}(\dot{H})u\>_{0}\\
&=&\<
\chi_{\mu}(\dot{H})u, \chi_{\mu}(\dot{H})u\>_{0}\lesssim \Vert
\chi_{\mu}(\dot{H})u\Vert_{\dot{\CE}}^2,
\end{eqnarray*}
which finishes the proof. 
\qed 

\subsection{Low frequency analysis}

Part $ii)$ of Thm. \ref{thboundedness} follows from the following
\begin{lemma}
\label{unifbound}
Assume \eqref{Hyp1}, \eqref{A2},
\eqref{TE1}-\eqref{TE3}, \eqref{PE2}, \eqref{Hyp15}. Let $\chi\in
C_0^{\infty}(\R\setminus \CS)$. Then there exists $C>0$ such
that:
\begin{equation}
\Vert e^{-it\dot{H}}\chi(\dot{H})u\Vert_{\dot{\CE}}\le C\Vert
u\Vert_{\dot{\CE}}, \ u\in \dot{\CE}, \ t\in \rr.
\end{equation}
\end{lemma}
\begin{proof}
Let $u\in \dot{\CE}$. Let 
\begin{eqnarray*}
\psi(t):=(\psi_0(t),\psi_1(t)):=e^{-itk}\left(\begin{array}{cc} 1 & 0 \\ -k & 1\end{array}\right)e^{-it\dot{H}}\chi(\dot{H})u.
\end{eqnarray*}
Note that
\begin{eqnarray*}
\Vert e^{-it\dot{H}}\chi(\dot{H})u\Vert_{\dot{\CE}}^2=\Vert\psi_1\Vert^2+(h(t)\psi_0(t)|\psi_0(t))
=:\Vert\psi(t)\Vert^2_{\CE(t)}
\end{eqnarray*}
with $h(t)=e^{-itk}(h+k^2)e^{itk}$ and that $\psi(t)$ solves  the wave equation
\begin{eqnarray*}
(\partial_t^2+h(t))\psi_0(t)=0,\quad \psi_1(t)=-i\partial_t\psi_0(t).
\end{eqnarray*}
Thus
\begin{eqnarray*}
\frac{d}{dt}\Vert e^{-it\dot{H}}\chi(\dot{H})u\Vert^2_{\dot{\CE}}=(h'(t)\psi_0(t)|\psi_0(t)),
\end{eqnarray*}
where
\begin{eqnarray*}
h'(t)=e^{-itk}[-ik,h]e^{ikt}\lesssim w^{-1}e^{-ikt}h_0e^{ikt}w^{-1},
\end{eqnarray*}
 using \eqref{Hyp15}. Therefore~:
\begin{eqnarray}
\label{energyderiv}
\frac{d}{dt}\Vert e^{-it\dot{H}}\chi(\dot{H})u\Vert^2_{\dot{\CE}}\lesssim\Vert w^{-1}\psi(t)\Vert^2_{\CE(t)}\lesssim\Vert w^{-1}e^{-it\dot{H}}\chi(\dot{H})u\Vert_{\dot{\CE}}^2,
\end{eqnarray}
where we have used that $w^{-1}$ commutes with
$e^{-itk}\left(\begin{array}{cc} 1 & 0\\ k & 1\end{array}\right)$ (see
\eqref{TE1}).
Integrating \eqref{energyderiv} we obtain~:
\begin{eqnarray*}
\Vert e^{-it\dot{H}}\chi(\dot{H})u\Vert^2_{\dot{\CE}}\lesssim\Vert u\Vert_{\dot{\CE}}^2
+\int_0^{t}\Vert w^{-1}e^{-it\dot{H}}\chi(\dot{H})u\Vert_{\dot{\CE}}^2dt\lesssim\Vert u\Vert_{\dot{\CE}}^2,
\end{eqnarray*}
by Prop.\ref{propPE3}.
\end{proof}
 We end this section by proving a weak convergence result, which will be important in Sect. \ref{secAsympc1}.
 \begin{lemma}
\label{weakconv}
Assume \eqref{Hyp1}-\eqref{A2}, \eqref{TE1}-\eqref{TE3}, \eqref{PE2}-\eqref{Hyp15}.
Then \[
e^{-it\dot{H}}\chi(\dot{H})\rightharpoonup 0, \ \forall \ \chi\in \coinf(\rr\backslash\CS).
\]
\end{lemma}
\proof 
 Since $\e^{-i t\dot{H}}\chi(\dot{H})$ is uniformly bounded in $t$ by Thm. \ref{unifbound1}, it suffices to prove that $\langle v| e^{-i t\dot{H}}\chi(\dot{H})u\rangle_{\dot\cE}\to 0$ for $u, v$ in a dense subspace of $\dot\cE$, where $\langle \cdot | \cdot \rangle_{\dot\cE}$ is the scalar product associated to the norm of $\dot\cE$. By \eqref{PE2}  the space $\{u\in \dot\cE\ : \ u = \chi(x)u, \ \chi\in \coinf(\rr)\}$ is dense in $\dot\cE$. For such $u, v$ the convergence to $0$ follows from Prop. \ref{prop5.2}. \qed

\section{Asymptotic completeness 1: abstract setting}
In this section we prove existence and completeness of wave operators, comparing the full dynamics $\e^{-it\dot{H}}$ with  the two asymptotic dynamics $\e^{-i t\dot{H}_{\pm}}$, for energies away from the set $\mathcal{S}$ of singular points.
\label{secAsympc1}
We first define the spaces of {\em scattering states}.
\begin{definition}
We call $\chi\in C^{\infty}(\R)$ an
{\em admissible}  cut-off function for $\dot{H}$ if 
\begin{itemize}
\item[--] $\chi= 0$ in a neighborhood of $\CS$ and 
\item[--] $\chi= 0$ or $\chi= 1$ on $\R\setminus ]-R, R[$ for
  some $R>0$.
\end{itemize}
We note $\CC^H$ the set of all admissible 
cut-offs for $\dot{H}$.
\end{definition}
\begin{definition}
The spaces of {\em scattering states} are defined as
\begin{eqnarray*}
\dot{\CE}_{scatt}&:=&\{\chi(\dot{H})u;\, u\in \dot{\CE},\, \chi\in \CC^H\},\\
\dot{\CE}_{scatt\pm}&:=&\{\chi(\dot{H}_{\pm})u;\, u\in
\dot{\CE}_{\pm},\, \chi\in \CC^H\}.
\end{eqnarray*}
\end{definition}
We will need the following three lemmas :
\begin{lemma}
\label{lemma6.1}
Assume \eqref{Hyp1}-\eqref{A2} and \eqref{TE1}-\eqref{TE3}. Then
$w[\dot{H},i_{\pm}]w\in \CB(\dot{\CE};\dot{\CE}_{\pm})$.
\end{lemma}
\proof
We have
\begin{eqnarray*}
w[\dot{H},i_{\pm}]w=\left(\begin{array}{cc} 0 & 0\\ w[h,i_{\pm}]w &
    0\end{array}\right)\in B(\dot{\CE}, \dot{\CE}_{\pm}),
\end{eqnarray*}
 by hypothesis \eqref{TE3}{\it e)}.
\qed

\begin{lemma} \label{lemma6.2}
Assume \eqref{Hyp1}-\eqref{A2} and \eqref{TE1}-\eqref{TE3}.
\begin{itemize}
\item[i)] Let $\chi\in C_0^{\infty}(\R)$. Then
\begin{equation*}
i_{\pm}\chi(\dot{H}_{\pm})-\chi(\dot{H})i_{\pm}\in
\CB_{\infty}(\dot{\CE}_{\pm};\dot{\CE}).
\end{equation*}
\item[ii)] Let $\chi\in C^{\infty}(\R)$ such that $\chi= 1$ outside $]-R, R[$ 
for some $R>0$. Then
\begin{equation*}
i_{\pm}\chi(\dot{H}_{\pm})-\chi(\dot{H})i_{\pm}\in
\CB_{\infty}(\dot{\CE}_{\pm};\dot{\CE}).
\end{equation*}
\end{itemize}
\end{lemma}

\proof 
Note first that $ii)$ follows from $i)$, replacing $\chi$ by $1- \chi$.
We therefore only have to prove $i)$.
Let $\tilde{\chi}$ be an almost analytic extension of $\chi$  such that $\supp\tilde{\chi}$ doesn't
intersect the complex poles of $\dot{R}(z)$. We have
\begin{eqnarray*}
i_{\pm}\chi(\dot{H}_{\pm})-\chi(\dot{H})i_{\pm}=\frac{1}{2\pi
  i}\int\bar{\partial}\tilde{\chi}(z)\dot{R}(z)[\dot{H},i_{\pm}]\dot{R}_{\pm}(z)dz\wedge
d\bar{z}.
\end{eqnarray*}
By hypotheses \eqref{TE3}{\it b) }and \eqref{TE3}{\it e)} we have
\[[\dot{H},i_{\pm}]\dot{R}_{\pm}(z)\in
\CB_{\infty}(\dot{\CE}_{\pm};\dot{\CE}).
\]
Then we apply the estimates in Lemma \ref{lemma4.5}.
%We also have the estimate
%\[\Vert
%\dot{R}(z)[\dot{H},i_{\pm}]\dot{R}_{\pm}(z)\varphi\Vert_{\dot{\CE}}\lesssim
%|{\rm Im}z|^{-M}\Vert \varphi\Vert_{\dot{\CE}_{\pm}},\, M>0\]
%for $z\in \supp\tilde{\varphi}$.
\qed

\begin{theorem}
\label{asympcompl}
Assume \eqref{Hyp1}, \eqref{A2},
\eqref{TE1}-\eqref{TE3}, \eqref{PE2}-\eqref{Hyp15}.

(i) For all $\varphi^{\pm}\in \dot{\CE}_{scatt\pm}$ there exist $\psi^{\pm}\in \dot{\CE}_{scatt}$ such that
\begin{eqnarray*}
e^{-it\dot{H}}\psi^{\pm}-i_{\pm}e^{-it\dot{H}_{\pm}}\varphi^{\pm}\rightarrow
0,\, t\rightarrow\infty\quad\mbox{in}\quad \dot{\CE}.
\end{eqnarray*}
(ii)  For all $\psi^{\pm}\in \dot{\CE}_{scatt}$ there exist $\varphi^{\pm}\in \dot{\CE}_{scatt\pm}$ such that
\begin{eqnarray*}
e^{-it\dot{H}_{\pm}}\varphi^{\pm}-i_{\pm}e^{-it\dot{H}}\psi^{\pm}\rightarrow
0,\, t\rightarrow\infty\quad\mbox{in}\quad \dot{\CE}_{\pm}.
\end{eqnarray*}
\end{theorem}
\proof

Let $\chi\in \CC^H$. We only prove $(i)$, the
proof of $(ii)$ being analogous. \\
We first show that the limit
\begin{eqnarray}
\label{Wpmtrunc}
W^{\pm}\varphi:=\lim_{t\rightarrow\infty}e^{it\dot{H}}\chi(\dot{H})i_{\pm}e^{-it\dot{H}_{\pm}}\chi(\dot{H}_{\pm})\varphi
\end{eqnarray}
exists for all $\varphi\in \dot{\CE}_{\pm}$. We first treat the case
$\chi= 0$ on $\R\setminus ]-R, R[$. Using  Thm. \ref{thboundedness} {\it i)}, Lemma \ref{lemma4.3.1} and the fact that $\dot{H}_{\pm}$ are selfadjoint,  we obtain:
\beq\label{toto}
\Vert
e^{it\dot{H}}\chi(\dot{H})i_{\pm}e^{-it\dot{H}_{\pm}}\chi(\dot{H}_{\pm})\varphi\Vert_{\dot{\CE}}\lesssim
\Vert \varphi\Vert_{\dot{\CE}_{\pm}}.
\eeq
By \eqref{toto} and assumption \eqref{PE2} we may assume that 
$\varphi\in D(w)$. We compute
\begin{eqnarray}
\label{HDHHpm}
\frac{d}{dt}e^{it\dot{H}}\chi(\dot{H})i_{\pm}e^{-it\dot{H}_{\pm}}\chi(\dot{H}_{\pm})=e^{it\dot{H}}\chi(\dot{H})[\dot{H},i_{\pm}]e^{-it\dot{H}_{\pm}}\chi(\dot{H}_{\pm}).
\end{eqnarray}
Integrating \eqref{HDHHpm} and
using Lemma \ref{lemma6.1} and Proposition \ref{prop5.2} we
obtain:
\begin{eqnarray*}
\lefteqn{\Vert e^{it\dot{H}}\chi(\dot{H})i_{\pm}e^{-it\dot{H}_{\pm}}\chi(\dot{H}_{\pm})-e^{is\dot{H}}\chi(\dot{H})i_{\pm}e^{-is\dot{H}_{\pm}}\chi(\dot{H}_{\pm})\varphi\Vert}\\
&\lesssim&\int_s^t\Vert
w^{-1}e^{-it'\dot{H}_{\pm}}\chi(\dot{H}_{\pm})\varphi\Vert_{\dot{\CE}_{\pm}}\lesssim\int_s^t\<t'\>^{-2}dt'\Vert
  w\varphi\Vert_{\dot{\CE}_{\pm}}\rightarrow 0,\, s,t\rightarrow \infty.
\end{eqnarray*}
This gives the existence of the limit \eqref{Wpmtrunc}. Let now
$\chi= 1$ on $\R\setminus ]-R, R[$. Let $\hat{\chi}\in
C_0^{\infty}(\rr),\, \hat{\chi}= 1$ in a neighborhood of $0$. Using
that $e^{it\dot{H}}\chi(\dot{H})$ is uniformly bounded  by Thm. \ref{thboundedness} {\it ii)},  
Lemma \ref{lemma6.2} and Lemma \ref{propresestlowfr} we see that
\begin{align*}
&\lefteqn{\slim_{t\rightarrow
    \infty}e^{it\dot{H}}\chi(\dot{H})i_{\pm}e^{-it\dot{H}_{\pm}}\hat{\chi}^2\left(L^{-1}\dot{H}_{\pm}\right)\chi(\dot{H}_{\pm})}\\
=&\slim_{t\rightarrow\infty}e^{it\dot{H}}\chi(\dot{H})\hat{\chi}\left(L^{-1}\dot{H}\right)
i_{\pm}e^{-it\dot{H}_{\pm}}\hat{\chi}\left(L^{-1}\dot{H}_{\pm}\right)\chi(\dot{H}_{\pm})
\end{align*}
exists, since $\hat{\chi}(L^{-1}\cdot)$ is compactly supported. Let $\epsilon>0$. We estimate
\begin{align*}
\Vert
  &e^{it\dot{H}}\chi(\dot{H})i_{\pm}e^{-it\dot{H}_{\pm}}\chi(\dot{H}_{\pm})\varphi^{\pm}-e^{is\dot{H}}\chi(\dot{H})i_{\pm}e^{-is\dot{H}_{\pm}}\chi(\dot{H}_{\pm})\varphi^{\pm}\Vert_{\dot{\CE}}\\
\le&\left\Vert
  e^{it\dot{H}}\chi(\dot{H})i_{\pm}e^{-it\dot{H}_{\pm}}\hat{\chi}^2
  (L^{-1}\dot{H}_{\pm})\chi(\dot{H}_{\pm})\varphi^{\pm}-e^{is\dot{H}}\chi(\dot{H})i_{\pm}e^{-is\dot{H}_{\pm}}\hat{\chi}^2
  (L^{-1}\dot{H}_{\pm})\chi(\dot{H}_{\pm})\varphi^{\pm}\right\Vert_{\dot{\CE}}\\
&+2\left\Vert (1-\hat{\chi}^2 (L^{-1}\dot{H}_{\pm}))\varphi^{\pm}\right\Vert_{\dot{\CE}}<\epsilon,
\end{align*} 
if we choose first $L$ and then $t,s$ large enough. This shows the existence of the limit (\ref{Wpmtrunc}) if $\chi=1$ on $\rr\backslash ]-R, R[$.

Let for $\phi^{\pm}\in \dot{\CE}_{\pm}$
\begin{eqnarray*}
\psi^{\pm}_t=e^{it\dot{H}}\chi(\dot{H})i_{\pm}e^{-it\dot{H}_{\pm}}\chi(\dot{H}_{\pm})\phi^{\pm},\, \psi^{\pm}=\lim_{t\rightarrow\infty}\psi^{\pm}_t.
\end{eqnarray*}
Let us write 
\begin{eqnarray*}
\psi^{\pm}_t=\psi^{\pm}+r(t),\, r(t)\rightarrow 0,\, t\rightarrow\infty.
\end{eqnarray*}
Let $\tilde{\chi}\in \CC^H$ with $\tilde{\chi}\chi=\chi.$
We clearly have $\tilde{\chi}(\dot{H})\psi^{\pm}_t=\psi^{\pm}_t$ and thus 
\begin{eqnarray*}
\tilde{\chi}(\dot{H})\psi^{\pm}+\tilde{\chi}(\dot{H})r(t)=\psi^{\pm}+r(t).
\end{eqnarray*}
Taking the limit $t\rightarrow\infty$ we find~:
\begin{eqnarray*}
\tilde{\chi}(\dot{H})\psi^{\pm}=\psi^{\pm}\quad\mbox{in particular}\quad
\psi^{\pm}\in \dot{\CE}_{scatt},
\end{eqnarray*}
hence $\e^{-i t\dot{H}}\psi^{\pm}$ is uniformly bounded by Thm. \ref{thboundedness}.
It follows 
\begin{equation*}
e^{-it\dot{H}}\psi^{\pm}-\chi(\dot{H})i_{\pm}e^{-it\dot{H}_{\pm}}\chi(\dot{H}_{\pm})\phi^{\pm}\rightarrow
0.
\end{equation*}
Applying Lemma \ref{lemma6.2} we find
\begin{eqnarray}
\label{asympcompl*}
e^{-it\dot{H}}\psi^{\pm}-i_{\pm}e^{-it\dot{H}_{\pm}}\chi^2(\dot{H}_{\pm})\phi^{\pm}\rightarrow 0,\, t\rightarrow\infty.
\end{eqnarray}
Applying once more a density argument we obtain {\it i)}.
\qed

\section{Geometric setting}
\label{SecGS}
We describe in this section an intermediate geometric framework corresponding to our abstract framework. The main example will of course be the Klein-Gordon equation on the Kerr-De Sitter spacetime.
The main part of the section will consist in checking that the geometric hypotheses introduced below imply the abstract hypotheses in Sects. \ref{SecAC1}, \ref{secAsympc1}.

We consider a $d$ dimensional manifold of the form
$\CM=]r_{-}, r_{+}[_{r}\times S^{d-1}_{\omega}$. Let 
\[P=\sum_{ij=1}^{d-1}D^*_i\alpha_{ij}(\omega)D_j\ge 0\] 
be a symmetric elliptic operator on $L^2(S^{d-1}_{\omega};d\omega)$. Then
$(P,H^2(S^{d-1}_{\omega};d\omega))$ is selfadjoint. We assume  that for a suitable
choice of $\theta_1, \,L^2(S^{d-1}_{\omega};d\omega)$ possesses a basis of
eigenfunctions of $D_{\theta_1}$. Let $Y^n$ be the eigenspace
corresponding to the eigenvalue $n$. Then we have
\begin{equation*}
L^2(S^{d-1}_{\omega};d\omega)=\oplus_{n\in \Z}Y^n.
\end{equation*}
Our first assumption is 
\begin{equation}
\tag{G1} \label{G1}
[P, D_{\theta_{1}}]
=0, \hbox{ ie }\alpha_{ij}\hbox{ are independent on }\theta_{1}. 
\end{equation}
Let $q(r):=\sqrt{(r_+-r)(r-r_-)}$. We will need the following function spaces
\begin{equation*}
T^{\sigma}=\{f\in
C^{\infty}(\CM):\quad \partial_r^{\alpha}\partial_{\omega}^{\beta}f\in
\CO(q(r)^{\sigma-2\alpha})\}.  
\end{equation*}
Let $i_{\pm}\in C^{\infty}([r_-,r_+])$, $i_-=0$ in a neighborhood of
$r_+$, $i_+=0$ in a neighborhood of $r_-$ and $i_-^2+i_+^2=1$.
\subsection{Separable Hamiltonians}
We will consider a Hamiltonian $h_{0,s}$ of the following form:
\begin{equation}
\label{sepham}
h_{0,s}=\alpha_1D_r\alpha_2^2D_r\alpha_1+\alpha_3^2P+\alpha^2_4.
\end{equation}
Here $\alpha_i,\, 1\leq i \leq 4$ are smooth functions depending only
on $r$. 
We suppose that there exist $\alpha_j^{\pm}\in \R,\,1\leq j \leq 4$ such that for some $\delta>0$
\begin{equation}
\tag{G2} \label{G2}
   \alpha_j-q(r)(i_-\alpha_j^-+i_+\alpha_j^+)\in
T^{1+\delta},\
\alpha_j\gtrsim q(r).
\end{equation}
Note that $\eqref{G2}$ implies
\begin{equation}
\label{totoblut}
\alpha_{j}\in T^{1}, \ \alpha_{j}\lesssim q(r).
\end{equation}
We will also need an operator $k_s$ of the form
\begin{equation}
\label{sepk}
k_s=k_{s,r}D_{\theta_1}+k_{s,v}.
\end{equation}
Here $k_{s,r}$ and $k_{s,v}$ are smooth functions depending only on $r$.
We suppose that there exist $k_{s,v}^-,\, k_{s,r}^-\in \R$ such that
for some $\delta>0$
\begin{equation}
\tag{G3} \label{G3}
\left\{\begin{array}{rcl}
i_+k_{s,r}, \, i_+k_{s,v}&\in&T^{2},\\
i_-(k_{s,r}-k_{s,r}^-)&\in&T^{2},\\
i_-(k_{s,v}-k_{s,v}^-)&\in&T^{2}.\end{array}\right.
\end{equation}
%\begin{remark}
%It is certainly possible to weaken the hypotheses and to require to
%control only a finite number of derivatives.
%\end{remark}
We put:
\begin{equation*}
h_s=h_{0,s}-k_s^2.
\end{equation*}
The associated separable Klein-Gordon equation is:
\begin{equation}
\label{wavesep}
(\partial_t^2-2ik_s\partial_t+h_s)u=0.
\end{equation}
We introduce the Hilbert spaces
$\CH=L^2(]r_{-}, r_{+}[_{r}\times S^{d-1}_{\omega};drd\omega)$ and
$\CH^n=\CH\cap Y^n$.
\subsection{Perturbed Hamiltonian}
\label{perthamil}
We consider a perturbation  of (\ref{wavesep}) of the form
\begin{equation}
\label{wavenonsep}
(\partial_t^2-2ik\partial_t+h)u=0.
\end{equation}
We first fix an operator $h_0$ of the form :
\beq
\label{perturh}
\begin{array}{rl}
h_0|_{C_0^{\infty}(\CM)}=&h_{0,s}+\sum_{i,j\in
  \{1,...,d-1\}}D_i^*g^{ij}D_j+\sum_{i\in\{1,...,d-1\}}(g^iD_i+D_i^*\bar{g}i)\\[2mm]
&+D_rg^{rr}D_r+g^rD_r+D_r\bar{g}^r+f\\[2mm]
=:&h_{0,s}+h_p.
\end{array}
\eeq 
We assume
\begin{equation}
\tag{G4}\label{intheta1}
\mbox{The functions $g^{ij},\, g^{i},\, g^{rr},\, g^{r},\, f$ are
  independent of $\theta_1$},
\end{equation}
and
\begin{equation}
\tag{G5} \label{bornehp}
\left\{\begin{array}{rcl} h_0&\gtrsim& \alpha_1(r)(D_rq^2(r)D_r+P+1) \alpha_1(r),\\
h_{0,s}&\gtrsim& \alpha_1(r)(D_rq^2(r)D_r+P+1)
\alpha_1(r).\end{array}
\right.
\end{equation}
Note that \eqref{bornehp} implies \eqref{Hyp1}. $(h_0,C_0^{\infty}(\CM))$ is symmetric, we denote by $(h_0,D(h_0))$ its
Friedrichs extension. 
The asymptotic behavior of the various  coefficients is assumed to be as follows:  we require that there exists $\delta>0$ such that
\begin{equation}
\tag{G6} \label{G7}
\begin{array}{l} g^{ij}\in T^{2+\delta},\ g^{rr}\in T^{4+\delta},\\[2mm]
g^{r}\in T^{2+\delta},\
g^i\in T^{2},\
f\in T^2.\end{array}
\end{equation}
We also set
\begin{equation*}
k:=k_rD_{\theta_1}+k_v,\, k_r=k_{s,r}+k_{p,r},\, k_v=k_{s,v}+k_{p,v}
\end{equation*}
and suppose
\begin{equation}
\tag{G7} \label{G8}
k_{p,r},\, k_{p,v}\in T^{2}.
\end{equation} 
Finally we set
\[
h:=h_0-k^2.
\]
All operators have natural restrictions to the space $\CH^n$, which we
denote by a subscript $n$, for example $h_0^n$ is the restriction of $h_0$
to $\CH^n$. The operator $k^n$ is bounded and $(h^n, D(h_0^n))$ is
selfadjoint. We will check the hypotheses for the restrictions of the
operators to $\CH^n$ and the corresponding energy spaces. We will drop the index $n$ in the following.  
\subsection{Asymptotic Hamiltonians}
Let us first introduce the change of variables given by
\[ \frac{dx}{dr}=\alpha_1^{-2}(r).\]
Note that there is a freedom in the choice of the integration
constant. The choice of this constant is however not important for
what follows. For $r\rightarrow r_-$ we find
\[x(r)-x(r_-)=\int_{r_-}^r\frac{1}{\alpha_-^2q^2}dr+\int_{r_-}^rh(r)
dr\]
with $h(r)\in \CO((r-r_-)^{-1+\delta})$. Here we have used
\eqref{G2}. Recalling that $q(r)=\sqrt{(r_+-r)(r-r_-)}$ we find for $r$ close to $r_-$
\[x(r)-x(r_-)\ge \frac{1}{\kappa_-}\ln (r-r_-)-C\]
with $\kappa_-=(\alpha_-)^2(r_+-r_-)$. It follows 
\[r-r_-\lesssim e^{\kappa_-x},\quad r\rightarrow r_-.\]
In a similar way we obtain
\[r_+-r\lesssim e^{-\kappa_+ x},\quad r\rightarrow r_+,\]
where $\kappa_+=(\alpha_+)^2(r_+-r_-)$. Using that
$\partial_x=\alpha_1^{-2}(r)\partial_r$ we find :
\begin{eqnarray*}
f(r)&\in& T^{\sigma}\quad\mbox{iff}\quad f(r(x))\in
T^{\sigma}_x\quad\mbox{for}\\
T^{\sigma}_x&=&\left\{f\in C^{\infty}(\R\times
  S^{d-1}_{\omega});\, \partial^{\alpha}_x\partial^{\beta}_{\omega}f\in
  \left\{\begin{array}{c} \CO (e^{\frac{\sigma \kappa_-}{2}x}),\,
        x\rightarrow -\infty,\\ \CO (e^{-\frac{\sigma
            \kappa_+}{2}x}),\quad x\rightarrow \infty. \end{array}\right\}\right\}.
\end{eqnarray*}
This change of variables gives rise to the unitary transformation
\begin{equation*}
\CU_1:\left. \begin{array}{rcl} L^2(]r_{-}, r_{+}[_{r}\times S^{d-1}_{\omega};drd\omega)&\rightarrow&
    L^2(\R\times S^{d-1}_{\omega};dxd\omega)=:\CH_1\\
v(r,\omega)&\mapsto&\alpha_1(r(x))v(x,\omega).\end{array}\right.
\end{equation*}
We put
\begin{equation*}
\CE_{\pm,1}=(\CU_1\oplus\CU_1)\CE_{\pm},\, h_{\pm}^1:=\CU_1h_{\pm}\CU_1^{-1},\, k^1_{\pm}=\CU_1k_{\pm}\CU_1^{-1}.
\end{equation*} 
We compute
\begin{align*}
h_0^1=\CU_1h_0\CU_1^{-1}=&\ \CU_1h_{0,s}\CU_1^{-1}+\sum_{i,j\in
  \{1,...,d-1\}}D_i^*g^{ij}D_j+\sum_{i\in\{1,...,d-1\}}(g^iD_i+D_i^*\bar{g}i)\\
&+\alpha_1^{-1}(r)D_x g^{rr}\alpha_1^{-2}(r)D_x\alpha_1^{-1}(r)+\alpha_1^{-1}(r)g^rD_x\alpha_1^{-1}(r)+\alpha_1^{-1}(r)D_x\bar{g}^r\alpha_1^{-1}(r)+f,\\
\CU_1h_{0,s}\CU_1=&D_x \alpha_2^2 \alpha_1^{-2}D_x+\alpha_3^2P+\alpha_4^2.
\end{align*}
We will often drop the  exponent $1$ when it is clear which coordinate
system is used. 
\subsubsection{Asymptotic Hamiltonians}\label{subitopresto}
The asymptotic Hamiltonians are now constructed as  in
Subsect. \ref{Harlem Désir}:  we set 
\[
\ell:=nk_{s,r}^-+k_{s,v}^-, 
\]
and define $k_{\pm}$, $h_{+}$, $\tilde{h}_{-}$ as in (\ref{tutu}), (\ref{titi}).  Recall that  since $l$ is fixed, these operators depend only on a cutoff scale $R$.  The following hypothesis is the analog of (\ref{TE2}): 
\begin{equation}
\tag{G8} \label{G9}
(h_{+}, k_{+}), \ (\tilde{h}_{-}, k_{-}-l)\hbox{ satisfy \eqref{bornehp} with }h_{0}\hbox{ replaced by } h_{+}, \ \tilde{h}_{-}.\end{equation}

\subsection{Meromorphic extensions}
In this subsection we will check that $h_+, \tilde{h}_-$ satisfy \eqref{ME2}. To do so we use
a result of Mazzeo and Melrose about the
meromorphic extension of the truncated resolvent for the Laplace
operator on asymptotically hyperbolic manifolds, see \cite{MM87}. We start by briefly recalling
this result.
\subsubsection{A result of Mazzeo-Melrose}
Let $Y$ be a compact
$n-$ dimensional manifold with boundary given by the defining function $y$:
\begin{equation*}
\partial Y=\{y=0\},\, dy|_{\partial Y}\neq 0,\, y|_{Y^0}>0.
\end{equation*} 
Let $g$ be a complete metric  on $Y$ of the form
\begin{equation}
\label{asymphypmetr}
g=\frac{h}{y^2},
\end{equation}
where $h$ is a $C^{\infty}$ metric on $Y$. One is usually interested
in the study of the Laplace-Beltrami operator $\Delta_g$. We have to
consider slightly more general operators. Let 
\begin{equation*}
\cV_0(Y)=\{V\in C^{\infty}(Y;TY);\, V|_{\partial Y}=0\}
\end{equation*}
the space of  vector fields vanishing on the boundary.  
%We are
%interested in operators which look near the boundary like the Laplace
%operator on hyperbolic space. If $\H^n=\{(x,y);\, x\in \R^{n-1},\,
%y\in \R_+\}$ is the hyperbolic space, then $-\Delta_g$ writes
%\begin{equation*}
%-\Delta_g=(yD_y)^2+\sum_{j=1}^{n-1}(yD_{x_j})^2+i(n-1)yD_y
%\end{equation*}
In local coordinates $(y,x)$ near $\partial Y$ the vector fields  $y\partial_y,\, y\partial_{x_j}$ span $\cV_0(Y)$. 

We now need the
definition of the {\em normal operator}. For $p\in\partial Y$ the
tangent space, $T_pY$, is divided into two half-spaces by the
hypersurface $T_p\partial Y$. We will denote by $Y_p$ the half space
on the $"y"$ side (that is spanned by $T_p\partial Y$ and the inward normal
vector at $p$). Then any smooth coefficient polynomial in $\cV_0(Y)$ $Q$
defines a natural constant coefficient operator on $Y_p$:
\begin{equation}
\label{nomop}
N_p(Q)u=\lim_{r\rightarrow 0}R^*_rf^*Q(f^{-1})^*R^*_{\frac{1}{r}}u,
\end{equation}
where $u\in C^{\infty}(Y_p),\, R_r$ is the natural $\R_+$ action on
$Y_p\simeq N_+T_p\partial Y$ given by the multiplication by $r$ on the
fibers and $f$ is a local diffeomorphism from $\Omega\subset Y,\, p\in
\Omega$:
\begin{equation*}
f:\Omega\rightarrow \Omega',\, \Omega'\subset T_pY,\, f(p)=0,\,
df_p=I,\, f(\partial Y)\subset T_p\partial Y.
\end{equation*}  
The definition is independent of $f$. The normal operator freezes the
coefficients at a point $p$, one obtains a polynomial in the elements
of $\cV_0(Y_p)$. The following result is implicit in \cite{MM87}. We
use here the formulation in \cite{SaZw}.
\begin{proposition}[Mazzeo-Melrose 87]
\label{MaMe87}
Let $Q$ be a second order differential operator on $Y$ which is a
polynomial in $\cV_0(Y)$ with coefficients in $C^{\infty}(Y)$. Assume that
\begin{itemize}
\item[i)] the principal part of $Q$ is an elliptic polynomial in the
  elements of $\cV_0(Y)$ uniformly on $Y$,
\item[ii)] for every $p\in \partial Y$ the normal operator of $Q$ defined
  by \eqref{nomop} is given by
\begin{eqnarray*}
N_p(Q)&=&-K\left[z_1^2D_{z_1}^2+i(n-2)z_1D_{z_1}+z_1^2\sum_{i,j=2}^nh_{ij}(p)D_{z_i}D_{z_j}-\left(\frac{n-1}{2}\right)^2\right],\\
Y_p&=&\{z\in \R^n;\, z_1\ge 0\},\quad [h_{ij}]\geq C \one, C>0,
\end{eqnarray*}
where $K<0$ is constant on the components of $\partial Y$. 
\end{itemize}
Then for
any metric $g$ of the form \eqref{asymphypmetr}
\begin{equation*}
R_Q(z)=(Q-z^2)^{-1}:L^2(Y,dvol_g)\rightarrow  L^2(Y,dvol_g)
\end{equation*}
is holomorphic in $\{{\rm Im} z\gg 1\}$. For $N>0$ the operator
$y^NR_Q(z)y^N$ extends to a meromorphic operator in $\{{\rm
  Im}z>-\delta\}$ for some $\delta>0$.
\end{proposition}
\begin{remark}
The width of the strip in which one can extend the truncated resolvent 
is of order $N$ if one removes some special points along the imaginary
axis. At these points which are given by
$(-K)^{1/2}(-i)\left(\frac{2k-1}{2}\right)^{1/2},\, k\in \N$ essential
singularities might occur. If the operator is the Laplacian associated
to an asymptotic hyperbolic metric and the metric is even, then no
essential singularities appear, see \cite{Gui} for a detailed discussion of these questions. In our
case it is sufficient to know  that there exists a
meromorphic extension in some strip and we won't study the type of
singularities. 
\end{remark}
\subsubsection{Meromorphic extensions of the resolvents of
  $h_{+},\tilde{h}_-$} 
\label{secmehpm}
\begin{lemma}
 Assume hypotheses \eqref{G1}---\eqref{G9}. Then $(h_+,k_+,w)$ and $(\tilde{h}_-,k_- -\ell,w)$
satisfy   \eqref{HypME}, \eqref{ME2} for $w= q(r)^{-1}$.
 \end{lemma}
\proof 
 
We will show that  $w^{-\epsilon}(h_{\pm}^n-z^2)^{-1}w^{-\epsilon}$ has  a meromorphic extension to a strip $\{{\rm Im}z>-\delta_{\epsilon}\},\,
\delta_{\epsilon}>0$. Let us start with $h_+$. 

We want to apply Prop.
\ref{MaMe87}, for $Y= \overline{\CM}= [r_{-}, r_{+}]\times S^{d-1}$. The principal part of $h_+^n$ is an elliptic  polynomial in the elements of $\cV_0(Y)$
and hypothesis $i)$ of Prop. \ref{MaMe87} is fulfilled.  Near each boundary component
we put $z=r-r_-$ and $z=r_+-r$ resp. We change the $C^{\infty}$
structure on $Y$ (as a manifold with boundary) and allow a new
smooth coordinate $y=\sqrt{z}$. We will denote the new manifold by
$Y_{1/2}$ and think of $Y_{1/2}$ as a conformally compact
manifold in the sense of having a metric of the form
\eqref{asymphypmetr}. 

The operator $h_+^n$ becomes near $r=r_-$
\begin{eqnarray*}
h_+&=&\frac{1}{4}(\alpha_1^+)^2(\alpha_2^+)^2(r_+-r_-)^4D_yyD_yy+(\alpha_3^+)^2(r_+-r_-)^2y^2\sum_{i,j=1}^{d-1}D_i^*\alpha_{ij}D_j\\
&-&(k^+_{s,r}n+k_{s,v}^+)^2+\CO(y^{\delta})(D_yy)^2+\CO(y^{\delta})y^2\sum_{i,j=1}^{d-1}D_i^*\alpha_{ij}D_j+\CO(y^{\delta}).
\end{eqnarray*}
We now conjugate $h_+$ by a weight function (see \cite{SaZw}) and set:
\begin{equation}
\label{weightQ}
Q=((1-\chi)+\chi y^2)h_+((1-\chi)+\chi y^2)^{-1},
\end{equation}
where $\chi\in C^{\infty}(Y),\, \chi=1$ for
$y<\epsilon<1/2$ and $\chi= 0$ for $y>2\epsilon$. 
It follows that the normal operator becomes
\begin{eqnarray*}
N_p(Q)&=&\frac{1}{4}(\alpha_1^+)^2(\alpha_2^+)^2(r_+-r_-)^4\left(y^2\left(D_y^2+\frac{4(\alpha_3^+)^2}{(\alpha_1^+)^2(\alpha_2^+)^2(r_+-r_-)^2}\sum_{i,j=1}^{d-1}D_i^*\alpha_{ij}(p)D_j\right)\right.\\
&+&\left.iyD_y-1-\frac{4(k_{s,r}^+n+k_{s,v}^+)^2}{(\alpha_1^+)^2(\alpha_2^+)^2(r_+-r_-)^4}\right).
\end{eqnarray*}
This operator is shifted with respect to the model operator in
Proposition \ref{MaMe87} and the points where essential singularities may occur
are now given by
$z^2=(-K)\left(\beta-\left(\frac{1-2k}{2}\right)^2\right),\, k\in \N$, where
$-K=\frac{1}{4}(\alpha_1^+)^2(\alpha_2^+)^2(r_+-r_-)^4,\quad
\beta=-\frac{4(k_{s,r}^-n+k_{s,v})^2}{(\alpha_1^+)^2(\alpha_é^+)^2(r_+-r_-)^4}$. As
$\beta$ is negative all these points have strictly negative imaginary
part. Hence we obtain a meromorphic
continuation of $w^{-\epsilon}(Q-z^2)^{-1}w^{-\epsilon}$, where $Q$ is given by
\eqref{weightQ}. Since for ${\rm Im} z\gg 0$ we have 
\[(h_+-z^2)^{-1}=((1-\chi)+\chi y^2)^{-1}(Q-z^2)^{-1}((1-\chi)+\chi
y),\]
we obtain a meromorphic continuation of $w^{-\epsilon}(h_+-z^2)^{-1}w^{-\epsilon}$.  The proofs near $r=r_+$ and for $\tilde{h}_-$
are similar. \qed

\subsection{Verification of the  abstract hypotheses}
\begin{proposition}\label{propopo}
 Assume  hypotheses \eqref{G1}---\eqref{G9}. Then conditions  \eqref{Hyp1}-\eqref{A2}, \eqref{TE1}-\eqref{TE3}, \eqref{PE2}-\eqref{Hyp15} are satisfied
\end{proposition}
The rest of the subsection is devoted to the proof of Prop. \ref{propopo}. 
We start by some preparations.
\subsubsection{Some useful facts}
By \eqref{bornehp} we have the following
estimates 
\begin{eqnarray}
\label{7.5.1b}
\Vert q(r)D_r\alpha_1(r)u\Vert&\lesssim& \Vert h_0^{1/2}u\Vert,\\
\label{7.5.1c}
\Vert q(r)D_ju\Vert&\lesssim& \Vert h_0^{1/2}u\Vert,\, j=1,...,d-1,\\
\label{7.5.1d}
\Vert q(r)u\Vert&\lesssim& \Vert h_0^{1/2}u\Vert.
\end{eqnarray}
The estimates
\eqref{7.5.1b}-\eqref{7.5.1d} also hold with $h_0$ replaced by $h_+$
or $\tilde{h}_-$. We will also need the following Hardy type estimate.
\begin{lemma}
\label{lemHardy}
We have
\begin{align*}
i)& \ \Vert\<x(r)\>^{-1}u\Vert_{\CH}\lesssim \Vert h_0^{1/2}u\Vert_{\CH},\\
 ii)&\ \Vert fu\Vert_{\CH}\lesssim \Vert h_0^{1/2}u\Vert_{\CH}, \  f\in T^{\delta}, \ \delta>0.
\end{align*}
\end{lemma}
\proof
Since $\langle x(r)\rangle\sim \vert\ln(r-r_{\pm})\vert$ when $r\to r_{\pm}$, {\it ii)} follows from {\it i)}.
 We recall a version of Hardy's inequality:
\begin{equation}
\label{pr4}
\int_0^{\infty}|v(x)|^2x^{-2}dt\le 4\int_0^{\infty}|v'(x)|^2dx, \ v\in \coinf(\rr\backslash\{0\}).
\end{equation}
Let $\chi_1\in C_0^{\infty}(\R),\, \chi_1(0)=1$ and $\chi_2\in
C^{\infty}(\R)$ with $\chi_1+\chi_2=1$. We have 
\begin{equation*}
\Vert \<x\>^{-1}\chi_1u\Vert^2_{\CH_1}\lesssim \Vert \chi_1 u\Vert^2_{\CH_1}\lesssim
((-\partial_x^2+\alpha_1^2)u|u)_{\CH_1}
\end{equation*}
because $\alpha_1^2\gtrsim \chi_1^2$. Now applying \eqref{pr4} to $\chi_2u$
gives :
\begin{eqnarray*}
\Vert \<x\>^{-1}\chi_2u\Vert^2_{\CH_1}&\lesssim& \int_{\R\times S^2}
|\partial_x(\chi_2u)|^2dxd\omega\lesssim \int_{\R\times
  S^2}(\chi'_2)^2|u|^2dxd\omega+\int_{\R\times
  S^2}\chi_2^2|u'|^2dxd\omega\\
&\lesssim& ((-\partial_x^2+\alpha_1^2)u|u)_{\CH_1}
\end{eqnarray*}
because $(\chi_2')^2\lesssim \alpha_1^2$. It follows 
\begin{equation*}
\int\<x\>^{-2}|u|^2dxd\omega\lesssim \int (|D_xu|^2+\alpha_1^2|u|^2)dxd\omega.
\end{equation*}
Changing to coordinates $(r, \omega)$ yields, using $dx=\frac{1}{\alpha_1^2}dr$, $D_x=\alpha_1^2 D_r$:
\begin{equation*}
\int\<x(r)\>^{-4}|u|^2\frac{1}{\alpha_1^2}drd\omega\lesssim
\int (|\alpha_1D_ru|^2+|u|^2) drd\omega.
\end{equation*}
Putting $v=\frac{1}{\alpha_1}u$ gives 
\begin{eqnarray*}
\int\<x(r)\>^{-4}|v|^2drd\omega&\lesssim& \int
(|\alpha_1D_r\alpha_1v|^2+\alpha_1^2|v|^2)dr d\omega,\\
&\lesssim&  \int
(|qD_r\alpha_1v|^2+\alpha_1^2|v|^2)dr d\omega 
\end{eqnarray*}
which using  that $h_0\gtrsim \alpha_1(D_rq^2D_r+1)\alpha_1$ by (\ref{bornehp})  completes the proof of the lemma.
\qed

\begin{lemma}
\label{lemfgcomp}
Let $f,g\in C^{\infty}(\R)$ with
$\lim_{|x|\rightarrow\infty}f(x)=\lim_{|x|\rightarrow\infty}g(x)=0$. Then
$f(x)g(h_+)$ and $f(x)g(\tilde{h}_-)$ are compact.
\end{lemma}
\proof

We only prove the lemma for $h_+$, the proof for $\tilde{h}_-$ being
analogous. We may assume that  $f,g\in C_0^{\infty}(\R)$. Let $\Omega$ be
a bounded domain which contains $\supp f$. Then $f(x)g(h_+)$ sends
$L^2(\CM)$ to $H^2(\Omega)$. But $H^2(\Omega)\hookrightarrow
L^2(\Omega)\hookrightarrow L^2(\CM)$ and the first embedding is
compact.

\qed
\subsubsection{Verification of hypotheses \eqref{Hyp1}, \eqref{A2}.}
We have already noticed that \eqref{bornehp} implies \eqref{Hyp1} (in
particular \eqref{bornehp} implies $0\notin \sigma_{pp}(h_0)$). Let
us check \eqref{A2}. We first check that $h_0^{1/2}kh_0^{-1/2}\in
\CB(\CH)$. This will follow from 
\begin{equation}
\label{7.5.1a}
kh_0k\lesssim h_0, \hbox{ on }C_0^{\infty}(\CM).
\end{equation}
Several terms have to be estimated. 
\begin{eqnarray*}
\lefteqn{i)\quad -k\alpha_1\partial_r\alpha_2^2\partial_r\alpha_1k}\\
&=&
-\alpha_1\partial_rk^2\alpha_2^2\partial_r\alpha_1-\alpha_1\partial_rk\alpha_2^2\alpha_1k'+\alpha_1^2k'^2\alpha_2^2+\alpha_1k'\alpha_2^2k\partial_r\alpha_1\\
&\lesssim&
-\alpha_1\partial_rq^2\partial_r\alpha_1-\alpha_1\partial_rk\alpha_2^2\alpha_1k'+\alpha_1^2k'^2\alpha_2^2+\alpha_1k'\alpha_22k\partial_r\alpha_1.\\
\lefteqn{ii)\quad -k\partial_rg^{rr}\partial_rk}\\
&=&-\alpha_1\partial_r\left(\frac{k}{\alpha_1}\right)^2g^{rr}\partial_r\alpha_1-\alpha_1\partial_rkg^{rr}\left(\frac{k}{\alpha_1}\right)'+\left(\frac{k}{\alpha_1}\right)'g^{rr}k\partial_r\alpha_1\\
&\lesssim&
-\alpha_1\partial_rq^2\partial_r\alpha_1-\alpha_1\partial_rk\left(\frac{k}{\alpha_1}\right)'g^{rr}+\left(\frac{k}{\alpha_1}\right)'g^{rr}k\partial_r\alpha_1.\\
iii)&\quad& kg^rD_rk=\frac{k^2}{\alpha_1}g^rD_r\alpha_1+\frac{1}{i}kg^r\left(\frac{k}{\alpha_1}\right)'\alpha_1.
\end{eqnarray*}
Summarizing and adding the angular terms we find
\begin{eqnarray*}
kh_0k&\lesssim&
\alpha_1D_rq^2D_r\alpha_1-\alpha_1\partial_rk\alpha_2^2\alpha_1k'+\alpha_1^2(k')^2\alpha_2^2+\alpha_1k'\alpha_2^2k\partial_r\alpha_1-\alpha_1\partial_rk\left(\frac{k}{\alpha_1}\right)'g^{rr}\\
&+&\left(\frac{k}{\alpha_1}\right)'kg^{rr}\partial_r\alpha_1+\frac{1}{i}kg^r\left(\frac{k}{\alpha_1}\right)'\alpha_1+\frac{k^2}{\alpha_1}g^rD_r\alpha_1+\alpha_1D_r\frac{k^2}{\alpha_1}\bar{g}^r-\frac{1}{i}\left(\frac{k}{\alpha_1}\right)'k\alpha_1\bar{g}^r\\
&+&\sum_{ij}\alpha_3D^*_i\alpha_{ij}k^2D_j\alpha_3+k^2\alpha_4^2+\sum_{ij}D_i^*k^2g^{ij}D_j+\sum_ig^ik^2D_i+D_i^*k^2\bar{g}^i+k^2f\\
&+&\sum_{ij}D_i^*k\alpha_3^2\alpha_{ij}(D_jk)-\sum_{ij}\alpha_3^2(D_i^*k)\alpha_{ij}kD_j-\sum_{ij}\alpha_3^2(D_i^*k)\alpha_{ij}(D_jk)\\
&+&\sum_{ij}D_i^*kg^{ij}(D_jk)-\sum_{ij}(D_i^*k)g^{ij}kD_j-\sum_{ij}(D_i^*k)g^{ij}(D_jk)\\
&+&\sum_ikg^i(D_ik)-\sum_i(D_i^*k)k\bar{g}^i+2{\rm Im}g^rkk'.
\end{eqnarray*}
We have 
\begin{eqnarray*}
k\alpha_2^2k'\in T^2,\, \alpha_1^2(k')^2\alpha_2^2\in T^4,\,
k'\alpha_2^2k\in T^2,\,
\frac{k}{q}\left(\frac{k}{\alpha_1}\right)'g^{rr}\in T^{\delta},\, k\alpha_1g^r\left(\frac{k}{\alpha_1}\right)'\in T^{\delta},\\
\frac{k^2}{\alpha_1}g^r\frac{1}{q}\in T^{\delta},\, g^ik^2\in
T^{2},\,
k\alpha_3^2\alpha_{ij}(D_ik)\in T^4,\,
\alpha_3^2(D_i^*k)\alpha_{ij}(D_jk)\in T^6,\\
kg^{ij}(D_jk)\in
T^{4+\delta},\, (D_i^*k)g^{ij}(D_jk)\in T^{6+\delta},\, kg^i(D_ik)\in
T^4,\, g^rkk'\in T^{2+\delta}.
\end{eqnarray*}
This gives (\ref{7.5.1a}),  using \eqref{bornehp},
\eqref{7.5.1b}-\eqref{7.5.1d} and  Lemma \ref{lemHardy}. The estimate for 
\[\Vert (k-z)^{-1}\Vert_{\CB(h_0^{-1/2}\CH)}\]
is exactly the same with the derivatives of $k$ replaced by
$\frac{\partial k}{(k-z)^{2}}$. Here we also use that
\begin{eqnarray*}
\Vert (k-z)^{-1}\Vert_{\CB(\CH)}&\le&\frac{1}{\vert {\rm Im} z\vert},\\
\Vert (k-z)^{-1}\Vert_{\CB(\CH)}&\le&\frac{1}{\vert z\vert-\Vert k\Vert_{\CB(\CH)}},
\end{eqnarray*}
the second inequality being valid for $|z|\ge (1+\epsilon)\Vert
k\Vert_{\CB(\CH)},\, \epsilon >0$.
\subsubsection{Verification of hypotheses\eqref{TE1}-\eqref{TE3}}
\eqref{TE1} is obvious, let us check \eqref{TE2}. We check \eqref{TE2}
for $h_+$, the proof for $\tilde{h}_-$ being analogous. First note
that \eqref{bornehp} for $h_+$ implies $0\notin \sigma_{pp}(h_+)$. We
have  $k_+^2\in T^{2}$. This implies the estimate
\[\Vert k_+u\Vert\lesssim \Vert h_+^{1/2}u\Vert.\]
We now check that \eqref{TE3} is fulfilled. Recall that $w=q^{-1}$.
\begin{itemize}
\item[--] \eqref{TE3}{\it a)} follows from \eqref{G3}.
\item[--] \eqref{TE3}{\it b)} is clear.
\item[--] \eqref{TE3}{\it c)}: we have already shown in Sect. \ref{secmehpm} that
  $h_+,\tilde{h}_-$ fulfill \eqref{ME2}. Let us check \eqref{HypME}. 
\begin{itemize}
\item[--] \eqref{HypME}{\it a)} follows from \eqref{G3}.
\item[--] \eqref{HypME}{\it b)} is clear.
\item[--]\eqref{HypME}{\it c)}: let us show that
  $h_+^{-1/2}[h_+,w^{-\epsilon}]w^{\epsilon/2}$ is bounded. We have 
\begin{eqnarray*}
[ih_+,w^{-\epsilon}]&=&\alpha_1D_r\alpha_2^2\alpha_1(w^{-\epsilon})'+D_rg^{rr}(w^{-\epsilon})'+g^r(w^{-\epsilon})'+hc\\
&=&\alpha_1D_r\alpha_2^2\alpha_1(w^{-\epsilon})'+\alpha_1D_r\frac{g^{rr}}{\alpha_1}(w^{-\epsilon})'-\frac{1}{i}\alpha_1\left(\frac{1}{\alpha_1}\right)'g^{rr}(w^{-\epsilon})'+g^r(w^{-\epsilon})'+hc\\
&=&\alpha_1D_rq\alpha+\beta
\end{eqnarray*}
with $\alpha\in T^{\epsilon},\, \beta\in T^{\epsilon+\delta}$. We have
$h_+^{-1/2}\beta w^{\epsilon/2}\in \CB(\CH)$ by 
Lemma \ref{lemHardy}. We have $h_+^{-1/2}\alpha_1D_rq\alpha w^{\epsilon/2}\in
\CB(\CH)$ by \eqref{7.5.1b}. The proof for $\tilde{h}_-$ is analogous.
\item[--] \eqref{HypME}{\it d)} follows from  Lemma \ref{lemHardy}.
\item[--]\eqref{HypME}{\it e)} follows from Lemma \ref{lemfgcomp}.

\end{itemize}  
\item[--] \eqref{TE3}{\it d)}: the proof is exactly the same as for \eqref{A2},
  we omit the details.  
\item[--] \eqref{TE3}{\it e)}.: we start with $w[h,i_+]wh_+^{-1/2}$. We have
\begin{equation*}
w[ih,i_+]w=w(\alpha_1D_r\alpha_2^2\alpha_1i_+'+D_rg^{rr}i_+'+g^ri_+'+hc)w=\alpha
qD_r\alpha_1+\beta 
\end{equation*}
with $\alpha\in T^{\infty},\, \beta\in T^{\infty}.$ This gives
$w[h,i_+]wh_+^{-1/2}\in \CB(\CH)$. The proof for the other operators
is the same, except for $h_0^{-1/2}[w^{-1},h_0]w^{1/2}$ for which it
is analogous to the proof for
$h_+^{-1/2}[h,w^{-\epsilon}]w^{\epsilon/2}$. We omit the details. 
\item[--] \eqref{TE3}{\it f)} follows from Hardy's inequality, Lemma \ref{lemHardy}.
\end{itemize}
\subsubsection{Verification of hypotheses \eqref{PE2},  \eqref{Hyp15}}
\begin{itemize}
\item[--] \eqref{PE2}: thanks to \eqref{7.5.1b}-\eqref{7.5.1d} we see that
  $\Vert h_0^{1/2}u\Vert^2$ is equivalent to 
\[\Vert D_x u\Vert^2+\sum_{j=1}^{d-1}\Vert q(r(x))D_ju\Vert^2+\Vert
q(r(x))u\Vert^2.\]
As $\psi\left(\frac{x}{n}\right)u\rightarrow u$ in $L^2(\R\times
  S^{d-1}_{\omega})$ we only have to show that 
\[\left[iD_x, \psi\left(\frac{x}{n}\right)\right]u\rightarrow 0\]
for $u\in h_0^{-1/2}\CH$. We have 
\[\left[iD_x,\psi\left(\frac{x}{n}\right)\right]u=\frac{x}{n}\psi'\left(\frac{x}{n}\right)\frac{1}{x}h_0^{-1/2}h_0^{1/2}u.\]
By Lemma \ref{lemHardy} it is sufficient to show
that 
\[\frac{x}{n}\psi'\left(\frac{x}{n}\right)v\rightarrow
0\quad\mbox{for}\quad v\in L^2(\R\times S^{d-1}_{\omega})\]
which is obvious.
\item[--] \eqref{Hyp15}:  first note that $w^{-\epsilon}$ clearly sends $D(h_0)$
  into itself. We show that \eqref{Hyp15} is fulfilled. We compute
\begin{equation*}
[ih_{0,s},k]=\alpha_1D_r\alpha_2^2k'\alpha_1+\sum_{ij}D_i^*\alpha_{ij}(\p_jk)\alpha^2_3+hc=:C_r+C_{\omega}.
\end{equation*}
We have
\begin{eqnarray*}
w^{\epsilon}C_rw^{\epsilon}&=&\alpha_1w^{\epsilon}D_r\alpha_2^2k'\alpha_1w^{\epsilon}+hc\\
&=&\alpha_1D_r\alpha_2^2k'\alpha_1w^{2\epsilon}k'\alpha_1-\frac{1}{i}(w^{\epsilon})'\alpha_2^2\alpha_1^2w^{\epsilon}+hc.
\end{eqnarray*}
Using that 
\begin{eqnarray*}
\alpha_2^2w^{2\epsilon}k'\in T^{2-2\epsilon},\,
(w^{\epsilon})'\alpha_2^2\alpha_1^2w^{\epsilon}\in T^{2-2\epsilon},\,
\alpha_3^2\alpha_{ij}(\partial_jk)w^{2\epsilon}\in T^{4-2\epsilon}
\end{eqnarray*}
we find for $\epsilon<1$
\[w^{\epsilon}C_rw^{\epsilon}\lesssim h_0, \ w^{\epsilon}C_{\omega}w^{\epsilon}\lesssim h_0,
\]
 by Lemma \ref{lemHardy}.
We now compute :
\begin{eqnarray*}
w^{\epsilon}[i(h_0-h_{0,s}),k]w^{\epsilon}&=&\sum_{ij}D_i^*q\frac{g^{ij}(iD_jk)w^{2\epsilon}}{q}+\sum_ig^i(D_ik)w^{2\epsilon}\\
&+&\alpha_1D_rw^{2\epsilon}\alpha_1^{-1}g^{rr}k'-\frac{1}{i}\alpha_1\left(\frac{w^{\epsilon}}{\alpha_1}\right)'g^{rr}k'w^{\epsilon}+hc.
\end{eqnarray*}
Noting that
\begin{eqnarray*}
\frac{g^{ij}}{q}(D_jk)w^{2\epsilon}&\in&T^{3+\delta-2\epsilon},\quad
g^i(iD_jk)w^{2\epsilon}\in T^{3+\delta-2\epsilon},\\
g^rk'w^{2\epsilon}&\in&T^{2+\delta-2\epsilon},\quad\frac{w^{2\epsilon}g^{rr}k'}{q\alpha_1}\in T^{2+\delta-2\epsilon},\\
\left(\frac{w^{\epsilon}}{\alpha_1}\right)'\alpha_1g^{rr}k'w^{\epsilon}&\in&T^{2+\delta-2\epsilon} ,
\end{eqnarray*}
we find for $\epsilon>0$ sufficiently small using \eqref{7.5.1b}-\eqref{7.5.1d} and 
Lemma \ref{lemHardy} that
\begin{equation*}
w^{\epsilon}[i(h_0-h_{0,s}),k]w^{\epsilon}\lesssim h_0.
\end{equation*}
Thus \eqref{Hyp15} is fulfilled. 
\end{itemize}
%ici
\section{Asymptotic completeness 2 : geometric setting}
\label{secACgeo}
In this section we will compare the  full dynamics to asymptotic
spherically symmetric dynamics. We put 
\begin{eqnarray*}
h_{+\infty}&:=&h_{0,s},\quad h_{-\infty}:=h_{+\infty}-\ell^2,\\
k_{+\infty}&:=&0,\quad k_{-\infty}:=\ell.
\end{eqnarray*}
The operators 
\[
\dot{H}_{+\infty}:=\left(\begin{array}{cc} 0 & \one\\ h_{+\infty} &
    0\end{array}\right),\
\dot{H}_{-\infty}:=\left(\begin{array}{cc} 0 & \one\\ h_{-\infty} &
    2 \ell\end{array}\right)
\]
are selfadjoint on 
\[
\dot{\CE}_{+\infty}=(h_{+\infty})^{-1/2}\CH\oplus\CH\quad \mbox{ resp. }
\dot{\CE}_{-\infty}=\Phi(\ell)(
(h_{+\infty})^{-1/2}\CH\oplus\CH)
\]with domains 
\[
D(\dot{H}_{+\infty})=(h_{+\infty})^{-1/2}\CH\cap(h_{+\infty})^{-1}\CH\oplus\<h_{+\infty}\>^{-1/2}\CH,\
D(\dot{H}_{-\infty})=\Phi(\ell)D(\dot{H}_{+\infty}).
\]
\begin{remark}
We have $\sigma_{pp}(\dot{H}_{\pm\infty})=\emptyset.$ This follows from
\cite[Lemme 4.2.1]{Ha03}.
\end{remark}
\begin{lemma}
\label{lemma7.6.1}
Assume \eqref{G1}-\eqref{G8}. Then we have $\dot{\CE}_{+\infty}=\dot{\CE}_+,\,
\dot{\CE}_{-\infty}=\dot{\CE}_-$ with equivalent norms. 
\end{lemma}
\proof

We have to show 
\begin{eqnarray}
\label{7.6.1}
h_{+\infty}\lesssim h_+\lesssim h_{+\infty},\
h_{+\infty}\lesssim \tilde{h}_-\lesssim h_{+\infty}.
\end{eqnarray}
Recalling that
$h_{0,s}= h_{+\infty}$, it is sufficient to show \eqref{7.6.1} for
$h_{+\infty}$ replaced by $h_{0,s}$. First note that 
\begin{eqnarray}
\label{7.6.3}
h_{0,s}\lesssim \alpha_1(D_rq^2D_r+P+1)\alpha_1.
\end{eqnarray}
\eqref{G9} then gives 
\[h_{0,s}\lesssim h_+,\quad h_{0,s}\lesssim \tilde{h}_-.\]
By \eqref{bornehp}, \eqref{G7}, the Hardy inequality, Lemma
\ref{lemHardy}, and the estimates \eqref{7.5.1b}-\eqref{7.5.1d} we have 
\begin{equation*}
h_0\lesssim h_{0,s}.
\end{equation*}
Now,
\begin{eqnarray*}
h_+=h_0-k_+^2\lesssim h_0&\lesssim& h_{0,s},\\
\tilde{h}_-=h_0-(k_--\ell)^2\lesssim h_0&\lesssim& h_{0,s},
\end{eqnarray*}
which finishes the proof.
\qed
\begin{lemma}
\label{lem7.6.2}
Assume \eqref{G1}-\eqref{G8}. We have for $\chi\in C_0^{\infty}(\R)$ 
\begin{equation*}
\chi(\dot{H}_{\pm\infty})-\chi(\dot{H}_{\pm})\in
B_{\infty}(\dot{\CE}_{\pm}).
\end{equation*}
\end{lemma}
\proof

Let $\chi\in C_0^{\infty}(\R)$. We prove the lemma for 
$\chi(\dot{H}_{+\infty})-\chi(\dot{H}_{+})$, the proof for 
$\chi(\dot{H}_{-\infty})-\chi(\dot{H}_{-})
$ being analogous. Let us introduce for a positive selfadjoint operator
$\hat{h}$ the transformation
\begin{equation*}
\CU(h):=\frac{1}{\sqrt{2}}\left(\begin{array}{cc} h^{1/2} &
    i\\ h^{1/2} & -i \end{array}\right),\quad
\CU^{-1}(h)=\frac{1}{\sqrt{2}}\left(\begin{array}{cc} h^{-1/2}
  & h^{-1/2} \\ -i & i \end{array}\right).
\end{equation*}
Note that
\begin{eqnarray*}
\CU(h_+):\dot{\CE}_{\pm}\rightarrow\CH\oplus\CH,\
\CU(h_{+\infty}):\dot{\CE}_{+\infty}\rightarrow \CH\oplus\CH
\end{eqnarray*}
are unitary. We set
\begin{eqnarray*}
L_{+\infty}:=\CU(h_{+\infty})\dot{H}_{+\infty}\CU^*(h_{+\infty}),\
L_+:=\CU(h_+)\dot{H}_+\CU^*(h_+).
\end{eqnarray*}
By \cite[Lemmas 6.1.3, A.4.4]{Ha03} we have 
\begin{equation}
\label{7.6.4}
(\one-\CU(h_{+\infty})\CU^*(h_+))\chi(L_+), \ \chi(L_{+\infty})-\chi(L_+)\in
\CB_{\infty}(\CH\oplus\CH).
\end{equation}
We now write
\begin{align*}
\chi(\dot{H}_{+\infty})-\chi(\dot{H}_+)=&\ \CU^*(h_{+\infty})(\one-\CU(h_{+\infty})\CU^*(h_+))\chi(L_+)\CU(h_+)\\
&+\CU^*(h_{+\infty})(\chi(L_{+\infty})-\chi(L_+))\CU(h_{+\infty})\\
&+\CU^*(h_{+\infty})\chi(L_+)(\one-\CU(h_+)\CU^*(h_{+\infty}))\CU(h_{+\infty}),
\end{align*}
which is compact by \eqref{7.6.4}.
\qed

Let 
\[\dot{R}_{\pm\infty}(z):=(\dot{H}_{\pm\infty}-z)^{-1}.\]
In the same way as for $\dot{H}_{\pm}$ we can show:\begin{proposition}
Assume \eqref{G1}-\eqref{G8}. Let $\epsilon>0$. There exists a
discrete and closed set $\CT_{\pm\infty}\subset\R$ such that for
all 
$\chi\in C_0^{\infty}(\R\setminus \CT_{\pm\infty})$ and all $k\in \N$  we have 
\begin{equation}
\label{5.1pminfty}
\sup_{\epsilon>0,\, \lambda\in \R}\Vert
w^{-\epsilon}\chi(\lambda)\dot{R}_{\pm\infty}^k(\lambda\pm i\epsilon)w^{-\epsilon}\Vert_{\CB(\dot{\CE}_{\pm\infty})}<\infty.
\end{equation}
\end{proposition}
Let $\hat{\CT}:=\CS\cup\CT_{\pm\infty}$. The admissible
energy cut-offs for $\dot{H}_{\pm\infty}$ are now defined in exactly
the same manner as for $\dot{H}$ replacing $\CS$ by
$\hat{\CT}$ in the definition. Let $\CC^{H_{\pm\infty}}$ the set of all
admissible cut-off functions for $\dot{H}_{\pm\infty}$. 
We define 
\begin{eqnarray*}
\dot{\CE}_{scatt, \pm\infty}&=&\{\chi(\dot{H}_{\pm\infty})u;\,\chi\in \CC^{H_{\pm\infty}},\,  u\in
\dot{\CE}_{\pm\infty}\},\\
\dot{\CE}_{scatt}&=&\{\chi(\dot{H})u;\,\chi\in \CC^{H_{\pm\infty}},\,  u\in
\dot{\CE}\}.
\end{eqnarray*}
\begin{theorem}
\label{asympcompl2}
Assume \eqref{G1}-\eqref{G8}. 

(i) For all $\varphi^{\pm}\in\dot{\CE}_{scatt, \pm\infty}$ there exist $\psi^{\pm}\in \dot{\CE}_{scatt}$ such that
\begin{eqnarray*}
e^{-it\dot{H}}\psi^{\pm}-i_{\pm}e^{-it\dot{H}_{\pm\infty}}\varphi^{\pm}\rightarrow
0,\, t\rightarrow\infty\quad\mbox{in}\quad \dot{\CE}.
\end{eqnarray*}
(ii)  For all $\psi^{\pm}\in \dot{\CE}_{scatt}$ there exist $\varphi^{\pm}\in\dot{\CE}_{scatt, \pm\infty}$ such that
\begin{eqnarray*}
e^{-it\dot{H}_{\pm\infty}}\varphi^{\pm}-i_{\pm}e^{-it\dot{H}}\psi^{\pm}\rightarrow
0,\, t\rightarrow\infty\quad\mbox{in}\quad \dot{\CE}_{\pm\infty}.
\end{eqnarray*}
\end{theorem}
\proof 

Let $\chi\in \CC^{H_{\pm\infty}}$. By Thm. \ref{asympcompl} it is sufficient to show the existence of
the wave operators
\begin{eqnarray*}
W^{\pm}_{\chi}&=&\slim_{t\rightarrow
  \infty}e^{it\dot{H}_{\pm}}e^{-it\dot{H}_{\pm\infty}}\chi(\dot{H}_{\pm\infty})\quad\mbox{in}\quad
\dot{\CE}_{\pm},\\
\Omega^{\pm}_{\chi}&=&\slim_{t\rightarrow\infty}e^{it\dot{H}_{\pm\infty}}e^{-it\dot{H}_{\pm}}\chi(\dot{H}_{\pm})\quad\mbox{in}\quad
\dot{\CE}_{\pm}
\end{eqnarray*}
and that 
\begin{equation}
\label{7.6.6}
\tilde{\chi}(\dot{H}_{\pm}) W^{\pm}_{\chi}=W^{\pm}_{\chi}, \,
\tilde{\chi}(\dot{H}_{\pm\infty})
\Omega^{\pm}_{\chi}=\Omega^{\pm}_{\chi}.
\end{equation}
for $\tilde{\chi}\in \CC^{H_{\pm\infty}}$ with $\tilde{\chi}\chi=\chi$.
The existence of $W^+_{\chi},\, \Omega^+_{\chi}$ follows directly from
\cite[Thm.  6.2.2]{Ha03}. For the existence of $W^-_{\chi},\,
\Omega^-_{\chi}$ note that 
\begin{eqnarray*}
\Phi(\ell)\dot{H}_{-\infty}\Phi^{-1}(\ell)&=&\dot{H}^{\ell}_{-\infty}+\ell
\one,\\
\Phi(\ell)\dot{H}_{-}\Phi^{-1}(\ell)&=&\dot{H}^{\ell}_{-}+\ell
\one,
\end{eqnarray*}
where
\begin{eqnarray*}
\dot{H}_{-\infty}^{\ell}&=&\left(\begin{array}{cc} 0 & \one \\
h_{0,s} & 0\end{array}\right),\\
\dot{H}_-^{\ell}&=&\left(\begin{array}{cc} 0 & \one \\
    h_0-(k_--\ell)^2 & 2(k_--\ell)\end{array}\right).
\end{eqnarray*}
Again the existence of 
\begin{eqnarray*}
\tilde{W}^{-}_{\chi}&=&\slim_{t\rightarrow
  \infty}e^{it\dot{H}^{\ell}_{-}}e^{-it\dot{H}^{\ell}_{-\infty}}\chi(\dot{H}^{\ell}_{-\infty})\\
\tilde{\Omega}^{-}_{\chi}&=&\slim_{t\rightarrow\infty}e^{it\dot{H}^{\ell}_{-\infty}}e^{-it\dot{H}^{\ell}_{-}}\chi(\dot{H}^{\ell}_{-})
\end{eqnarray*}
follows from \cite[Thm. 6.2.2]{Ha03}. The existence of $W^-_{\chi},\,
\Omega^-_{\chi}$ then follows applying the transformation $\Phi(\ell)$. 
The identity \eqref{7.6.6} follows from Lemma \ref{lem7.6.2}.
\qed
\begin{remark}
\begin{enumerate}
\item[i)]
Note that the results of \cite{Ha03} apply here although the situation
considered in \cite{Ha03} is slightly different. In \cite{Ha03} the
cylindrical manifold $\R\times S^{d-1}_{\omega}$ has one
asymptotically euclidean end and one asymptotically hyperbolic end
whereas we consider here two asymptotically hyperbolic ends. The
situation with two asymptotically hyperbolic ends is simpler, in
particular no gluing of the two conjugate operators for the ends in
the setting of Mourre theory is necessary. 
\item[ii)] As $\sigma_{pp}(\dot{H}_{\pm\infty})=\emptyset$ and
  $\dot{H}_{\pm},\, \dot{H}_{\pm\infty}$ are selfadjoint, the wave
  operators 
\[W^{\pm}=\slim_{t\rightarrow\infty}e^{it\dot{H}_{\pm}}e^{-it\dot{H}_{\pm\infty}}\]
exist. In a similar way we obtain the existence of the wave operators
\[\Omega_{\pm}=\slim_{t\rightarrow\infty}e^{it\dot{H}_{\pm\infty}}e^{-itH_{\pm}}\one^{ac}(\dot{H}_{\pm}),\]
where $\one^{ac}(\dot{H}_{\pm})$ is the projection of the absolutely
continuous subspace of $\dot{H}_{\pm}$.
\end{enumerate}
\end{remark}
\section{The Klein-Gordon equation on the De Sitter Kerr
spacetime}
\label{SecKGDK}
In this section we recall the Klein-Gordon equation  on the De Sitter Kerr spacetime, which will be our main example of the geometric framework from Sect. \ref{SecGS}.
\subsection{The De Sitter Kerr metric in Boyer-Lindquist coordinates}
In Boyer-Lindquist coordinates the De Sitter Kerr space-time is
described by a smooth 4-dimensional Lorentzian manifold
$\CM_{BH}=\R_t\times \R_r\times S^2_{\omega}$, whose space-time metric
is given by
\begin{eqnarray}
\label{Kerrmetric}
g&=&\frac{\Delta_r-a^2\sin^2\theta\Delta_{\theta}}{\lambda^2\rho^2}dt^2+\frac{2a\sin^2\theta((r^2+a^2)^2\Delta_{\theta} -a^2\sin^2\theta\Delta_r)}{\lambda^2\rho^2}dtd\varphi\nonumber\\
&-&\frac{\rho^2}{\Delta_r}dr^2-\frac{\rho^2}{\Delta_{\theta}}d\theta^2-\frac{\sin^2\theta\sigma^2}{\lambda^2\rho^2}d\varphi^2,\\
\rho^2&=&r^2+a^2\cos^2\theta,\quad
\Delta_r=\left(1-\frac{\Lambda}{3}r^2\right)(r^2+a^2)-2Mr,\nonumber\\
\Delta_{\theta}&=&1+\frac{1}{3}\Lambda a^2\cos^2\theta,\,
\sigma^2=(r^2+a^2)^2\Delta_{\theta}-a^2\Delta_r\sin^2\theta,\,
\lambda=1+\frac{1}{3}\Lambda a^2\nonumber.
\end{eqnarray}
Here $\Lambda>0$ is the cosmological constant, $M>0$ is the mass of
the black hole and $a$ its angular momentum per unit mass. The metric
is defined for $\Delta_r>0$, we assume that this is fulfilled on an
open interval $]r_-, r_+[$. (For $a=0$, this is true when
$9\Lambda M^2<1$; it remains true if we take $a$ small enough). 

 Note
that the vector fields $\partial_t$ and $\partial_{\varphi}$ are
Killing. The De Sitter-Schwarzschild metric ($a=0$) is
a special case of the above. The set $\{\rho^2=0\}$ is a true curvature singularity. In contrast to
$\rho^2$ the roots of $\Delta_r$ are mere coordinate
singularities. $r_-$ and $r_+$ represent {\em event horizons} and we will only  be
interested in the region $r_-<r<r_+$. This region is not stationary in
the sense that there exists no global time-like Killing vector
field. In particular there are regions in
$\R_t\times]r_{-}, r_{+}[_{r}\times S^2_{\omega}$ in which $\partial_t$
becomes space-like. Indeed, the function $\Delta'_r$ has a
single zero $r_{max}$ on $]r_-,r_+[$. On $]r_-,r_{max}[$ $\Delta_r$ is
strictly increasing, on $]r_{max},r_+[$ $\Delta_r$ is strictly
decreasing. Therefore there exist $r_1(\theta),\, r_2(\theta)$ defined
on $]0,\pi[$ such
that
\[
\begin{array}{rl}\Delta_r-a^2\sin^2\theta\Delta_{\theta}<0&\hbox{ on }]r_{-}, r_{1}(\theta)[,\\[2mm]
\Delta_r-a^2\sin^2\theta\Delta_{\theta}>0&\hbox{ on }]r_{1}(\theta), r_{2}(\theta)[,\\[2mm]
\Delta_r-a^2\sin^2\theta\Delta_{\theta}<0&\hbox{ on }]r_{2}(\theta), r_{+}[.
\end{array}
\]
As a consequence the vector field $\partial_t$ is 
\begin{itemize}
\item[--] time-like on $\{(t,r,\theta,\varphi):\, r_1(\theta)<r<r_2(\theta)\}$,\\
\item[--] spacelike on $\{(t,r,\theta,\varphi):\,
  r_-<r<r_1(\theta)\}\cup \{(t,r,\theta,\varphi:\,
  r_2(\theta)<r<r_+\}=:\CA_-\cup\CA_+.$
\end{itemize}
The regions $\CA_{\pm}$ are called {\em ergo-spheres}. Of particular
interest are the {\em locally non rotating observers}. These observers have
four velocity
\[u^a=\frac{\nabla^a t}{(\nabla_bt\nabla^bt)^{1/2}}.\]
They rotate with coordinate angular velocity
\begin{equation}
\label{angvel}
\Omega=-\frac{g_{t\varphi}}{g_{\varphi\varphi}}=\frac{a((r^2+a^2)\Delta_{\theta}-\Delta_ra^2\sin^2\theta)}{\sigma^2}.
\end{equation}
Note that this angular velocity has finite limits at both horizons:
\begin{equation}
\label{angvelhor}
\Omega_{\pm}:=\Omega(r_{\pm},\theta)=\frac{a}{r_{\pm}^2+a^2}.
\end{equation}
\subsection{The Klein-Gordon equation on the De Sitter Kerr space-time}
We now  reduce the Klein-Gordon equation on the De Sitter Kerr space time to the  abstract form \eqref{WE}.

A standard computation using $\Box_{g}= |g|^{-\12}\p_{a}|g|^{\12}g^{ab}\p_{b}$ yields for  the  De Sitter Kerr metric:
\begin{eqnarray}
\label{ChKlGor}
\lefteqn{\left(\frac{\sigma^2\lambda^2}{\rho^2\Delta_{\theta}\Delta_r}\partial_t^2-2\frac{a(\Delta_r-(r^2+a^2)\Delta_{\theta})\lambda^2}{\rho^2\Delta_{\theta}\Delta_r}\partial_{\varphi}\partial_t\nonumber\right.}\\
&-&\left.\frac{(\Delta_r-a^2\sin^2\theta\Delta_{\theta})\lambda^2}{\rho^2\Delta_{\theta}\Delta_r\sin^2\theta}\partial_{\varphi}^2-\frac{1}{\rho^2}\partial_r\Delta_r\partial_r-\frac{1}{\sin\theta\rho^2}\partial_{\theta}\sin\theta\Delta_{\theta}\partial_{\theta}+m^2\right)\psi=0.
\end{eqnarray}
We multiply to the left the equation by $c^2=\frac{\rho^2\Delta_r\Delta_{\theta}}{\lambda^2\sigma^2}$
and obtain:
\begin{eqnarray}
\label{eq3.9}
\\
\lefteqn{\left(\partial_t^2-2\frac{a(\Delta_r-(r^2+a^2)\Delta_{\theta})}{\sigma^2}\partial_{\varphi}\partial_t\right.}\nonumber\\
&-&\left.\frac{(\Delta_r-a^2\sin^2\theta\Delta_{\theta})}{\sin^2\theta\sigma^2}\partial_{\varphi}^2-\frac{\Delta_r\Delta_{\theta}}{\lambda^2\sigma^2}\partial_r\Delta_r\partial_r-\frac{\Delta_r\Delta_{\theta}}{\lambda^2\sin\theta\sigma^2}\partial_{\theta}\sin\theta\Delta_{\theta}\partial_{\theta}+\frac{\rho^2\Delta_r\Delta_{\theta}}{\lambda^2\sigma^2}m^2\right)\psi=0.\nonumber
\end{eqnarray}
We now consider the unitary transform
\begin{equation*}
U:\left. \begin{array}{rcl}
    L^2(\CM;\frac{\sigma^2}{\Delta_r\Delta_{\theta}}dr
    d\omega)&\rightarrow&L^2(\CM;
    drd\omega)\\\psi&\mapsto&\frac{\sigma}{\sqrt{\Delta_r\Delta_{\theta}}}\psi \end{array} \right.
\end{equation*}
If $\psi$ solves \eqref{eq3.9}, then $u=U\psi$ solves
\begin{eqnarray}
\label{KGKerrv1}
\lefteqn{\left(\partial_t^2-2\frac{a(\Delta_r-(r^2+a^2)\Delta_{\theta})}{\sigma^2}\partial_{\varphi}\partial_t-\frac{(\Delta_r-a^2\sin^2\theta\Delta_{\theta})}{\sin^2\theta\sigma^2}\partial_{\varphi}^2-\frac{\sqrt{\Delta_r\Delta_{\theta}}}{\lambda\sigma}\partial_r\Delta_r\partial_r
    \frac{\sqrt{\Delta_r\Delta_{\theta}}}{\lambda\sigma}\right.}\nonumber\\
&-&\left.\frac{\sqrt{\Delta_r\Delta_{\theta}}}{\lambda\sin\theta\sigma}\partial_{\theta}\sin\theta\Delta_{\theta}\partial_{\theta}\frac{\sqrt{\Delta_r\Delta_{\theta}}}{\lambda\sigma}+\frac{\rho^2\Delta_r\Delta_{\theta}}{\lambda^2\sigma^2}m^2\right)u=0.
\end{eqnarray}
We introduce a Regge-Wheeler type
coordinate $x$ by the requirement :
\[\frac{dx}{dr}=\lambda\frac{r^2+a^2}{\Delta_r}.\]
We then introduce the unitary transform
\begin{equation*}
\CV:\left.\begin{array}{rcl} L^2(]r_{-}, r_{+}[_{r}\times S^2)&\rightarrow&
    L^2(\R\times S^2,dxd\omega),\\
    v(r,\omega)&\mapsto&\sqrt{\frac{\Delta_r}{\lambda(r^2+a^2)}}v(r(x),\omega).\end{array}\right.
\end{equation*}
Let $u$ be a solution of the Klein-Gordon equation \eqref{KGKerrv1} and
$\psi=\sqrt{\frac{\Delta_r}{\lambda(r^2+a^2)}}u$. Then $\psi$ fulfills: 
\begin{eqnarray}
\label{eqpsiv1}
\\
\lefteqn{\left(\partial_t^2-2\frac{a(\Delta_r-(r^2+a^2)\Delta_{\theta})}{\sigma^2}\partial_{\varphi}\partial_t-\frac{(\Delta_r-a^2\sin^2\theta\Delta_{\theta})}{\sin^2\theta\sigma^2}\partial_{\varphi}^2
  \right.}\nonumber\\
&-&\left.\frac{\sqrt{(r^2+a^2)\Delta_{\theta}}}{\sigma}\partial_x(r^2+a^2)\partial_x
    \frac{\sqrt{(r^2+a^2)\Delta_{\theta}}}{\sigma}
-\frac{\sqrt{\Delta_r\Delta_{\theta}}}{\lambda\sin\theta\sigma}\partial_{\theta}\sin\theta\Delta_{\theta}\partial_{\theta}\frac{\sqrt{\Delta_r\Delta_{\theta}}}{\lambda\sigma}+\frac{\rho^2\Delta_r\Delta_{\theta}}{\lambda^2\sigma^2}m^2\right)\psi=0.\nonumber
\end{eqnarray}
\section{Asymptotic completeness 3 : The De Sitter Kerr case}
\label{SecAC3}
In this section we state the main theorems for the De Sitter Kerr
spacetime. The proofs are given in Sect. \ref{prDSKerr}.

We consider the Klein-Gordon equation \eqref{eqpsiv1} and write it in
the usual form
\begin{equation*}
(\partial_t^2-2ik\partial_t+h)\psi=0.
\end{equation*}
Let 
\begin{equation}\label{caribou}
\CH^n =\{u\in L^{2}(\R\times S^2):\  (D_{\varphi}-n)u=0\}, \ n\in \zz.
\end{equation}
We construct the energy spaces
$\dot{\CE}^n,\, \CE^n$ as well as the Klein-Gordon operators $H^n,\,
\dot{H}^n$ as in Sect. \ref{SecBG}. Also let $i_{\pm}\in C^{\infty}(\R),\,
i_-=0$ in a neighborhood of $\infty$, $i_+=0$ in a neighborhood of
$-\infty$ and $i_-^2+i_+^2=1$. We will use two types of comparison
dynamics :
\begin{enumerate}
\item[--] a separable comparison dynamics,
\item[--] asymptotic profiles.
\end{enumerate}
\subsection{Uniform boundedness of the evolution}
\begin{theorem}
\label{thunifbound}
There exists $a_0>0$ such that for $|a|<a_0$  the following holds:  for all $n\in \Z$, there exists $C_{n}>0$ such that \begin{equation}
\label{equnifbound}
\Vert e^{-it\dot{H}^n} u\Vert_{\dot{\CE}^n}\le C_n \Vert u\Vert_{\dot{\CE}^n}, \ u\in \dot{\CE}^{n}, \ t\in \rr.
\end{equation}
\end{theorem}
Note that  for $n=0$ the Hamiltonian $\dot{H}^n=\dot{H}^0$ is
  selfadjoint, therefore the only issue is $n\neq 0$.

Because of the existence of a zero resonance the evolution is not
expected to be uniformly bounded on the inhomogeneous energy space. This is already the case in the De Sitter Schwarzschild case, ie if $a=0$. In fact from \cite[Thm. 1.3]{BoHa08}, denoting by $r$ the zero resonance state, we have for $\chi\in \coinf(\rr)$:
\begin{eqnarray}
\label{res01}
\chi e^{-itH}\chi u&=&\gamma\left(\begin{array}{c} r\chi(\chi r|u_1) \\
    0\end{array}\right)+E(t),\quad \gamma>0\quad\mbox{and}\\
\label{res02}
\Vert E(t)\Vert_{\CE}&\lesssim& e^{-\epsilon
  t}\Vert\<-\Delta_{S^2}\>u\Vert_{\CE},
\end{eqnarray}
with $\epsilon>0$. Note that in \cite{BoHa08} the norm 
\[\Vert u_1\Vert^2+(h_0u|u)+\int_0^1\int_{S^2}|u_0|^2(x,\omega)dx
d\omega\]
is used, but that the same proof also gives the estimate \eqref{res02}. Now suppose
that $e^{-itH}$ is uniformly bounded on $\CE$. Then we obtain from
\eqref{res01} in the limit $t\rightarrow\infty$ :
\[\Vert r\chi(r\chi|u_1)\Vert_{\CE}\lesssim \Vert u\Vert_{\CE},\quad
u\in C_0^{\infty}(\R\times S^2)\oplus  C_0^{\infty}(\R\times
S^2)\]
and thus 
\[\Vert r\chi(r\chi|u_1)\Vert_{\CH}\lesssim \Vert u\Vert_{\CE},\quad
 u\in \CE\]
by density. Here $\CH=L^2(\R\times S^2, dxd\omega)$. It follows 
\[\Vert r\chi(r\chi|v)\Vert_{\CH}\lesssim \Vert v\Vert_{\CH},\quad
v\in \CH.\]
Thus $\Vert r\chi\Vert_{\CH}\lesssim 1$
uniformly in $\chi$ which implies $r\in \CH$ which is false. Therefore
the evolution is not uniformly bounded on $\CE$, neither on
$\CE^0$. It is however bounded on $\CE^n$ for all $n\neq 0$.
\subsection{Separable comparison dynamics}
Let $\ell_{\pm}=\Omega_{\pm}n$. We put: 
\[
h_{\pm \infty}:=-\ell_\pm^2-\partial_x^2+\frac{\Delta_r}{\lambda^2(r^2+a^2)}P+\Delta_rm^2,\
k_{\pm\infty}:=\ell_\pm,
\]where
\begin{equation*}
P:=-\frac{\lambda^2}{\sin^2\theta}\partial_{\varphi}^2-\frac{1}{\sin\theta}\partial_{\theta}\sin\theta\Delta_{\theta}\partial_{\theta}.
\end{equation*}
In the case $n=0$ $P$ might have a zero eigenvalue and the natural
energy spaces associated to $h_0$ and $h_{\pm \infty}$ may be
different in the massless case. We will therefore consider the case
$n=0$ only in the massive case. Let $\dot{\CE}^n_{\pm\infty},\, \dot{H}_{\pm\infty}^n$ be the
homogeneous energy spaces and operators associated to
$(h_{\pm\infty},k_{\pm\infty})$ according to Sect. \ref{SecBG}.
\begin{theorem}
\label{thACKerrSC}
There exists $a_0>0$ such that for $|a|<a_0$  and $n\in \zz\setminus\{0\}$ the following holds: 
\begin{enumerate}
\item[--] The wave operators 
\begin{equation}
\label{WEKERRSC}
W^{\pm}=\slim_{t\rightarrow\infty}e^{it\dot{H}^n}i_{\pm}e^{-it\dot{H}^n_{\pm\infty}}
\end{equation}
exist as bounded operators $W^{\pm}\in \CB(\dot{\CE}^n_{\pm\infty};\dot{\CE}^n)$.
\item[--] The inverse wave operators 
\begin{equation}
\label{IWEKERRSC}
\Omega^{\pm}=\slim_{t\rightarrow\infty}e^{it\dot{H}^n_{\pm\infty}}i_{\pm}e^{-it\dot{H}^n}
\end{equation}
exist as bounded operators $\Omega^{\pm}\in
\CB(\dot{\CE}^n;\dot{\CE}^n_{\pm\infty})$.
\end{enumerate}
\eqref{WEKERRSC} and  \eqref{IWEKERRSC} also hold for $n=0$ if $m>0$.
\end{theorem}
\subsection{Asymptotic profiles}
We now introduce the Hamiltonians $\dot{H}_r,\, \dot{H}_l$ which
describe the simplest possible asymptotic comparison dynamics. Let 
\[
h^n_{r/l}=-\partial_x^2-\ell_{+/-}^2,\
k_{r/l}=\ell_{+/-},
\]
acting on $\CH^{n}$ defined in (\ref{caribou}).

We associate to these operators the natural homogeneous energy spaces $\dot{\CE}^n_{l/r}$ and Hamiltonians $\dot{H}^n_{l/r}$.Note that the solution of 
\begin{eqnarray}
\label{profilwe}
\left\{\begin{array}{rcl}
    (\partial_t^2-2i\ell_{\pm}\partial_t-\partial_x^2-\ell_{\pm}^2)u&=&0,\\
u|_{t=0}&=&u_0,\\
\partial_tu|_{t=0}&=&u_1\end{array}
\right.
\end{eqnarray}
can be computed explicitly. Indeed if $u$ is the solution of
\eqref{profilwe}, then $v=e^{-i\ell_{\pm}t}u$ fulfills
\begin{eqnarray}
\label{profilwev}
\left\{\begin{array}{rcl}
    (\partial_t^2-\partial_x^2)v&=&0,\\
v|_{t=0}&=&u_0,\\
\partial_tv|_{t=0}&=&u_1-i\ell_{\pm}u_0\end{array}
\right.
\end{eqnarray}
Thus for smooth data the explicit solution of \eqref{profilwe} is
given by
\begin{eqnarray*}
u_0(t,x,\omega)=\frac{e^{i\ell_{\pm}t}}{2}\left(u_0(x+t,\omega)+u_0(x-t,\omega)+\int_{x-t}^{x+t}(u_1(\tau,\omega)-i\ell_{\pm}u_0(\tau,\omega))d\tau\right).
\end{eqnarray*}

Let us  denote the  cutoffs $i_{+/-}$ by $i_{r/l}$.

The spaces $i_l\dot{\CE}^n_l$ and $i_r\dot{\CE}^n_r$ are not
included in $\dot{\CE}^n$ and the group $e^{-it\dot{H}^n_{r/l}}$ doesn't
improve regularity. There is therefore no chance that the limits 
\begin{equation*}
W^+u=\lim_{t\rightarrow\infty}e^{it\dot{H}^n}i_{r/l}e^{-it\dot{H}^n_{r/l}}u
\end{equation*}
exist for all $u\in \dot{\CE}^n_{r/l}$. We will first show the existence of the limits on smaller spaces  and then extend the wave
operators by continuity.  Let $\{\lambda_{q}: q\in \nn\}= \sigma(P)$ and $Z_{q}= \one_{\{\lambda_{q}\}}(P)\cH$. Then 
\begin{equation*}
D(h_0)=D(h_{0,s})=\{u\in \CH: \ \sum_{q\in \N}\| h_{0}^{s,q}\one_{\{\lambda_{q}\}}(P)u\|^{2}<\infty\},
\end{equation*}
where $h_0^{s,q}$ is the restriction of $h_{0,s}$ to $L^2(\R)\otimes
Z_q$. Let
\begin{eqnarray*}
W_q&:=&(L^2(\R)\otimes Z_q)\oplus(L^2(\R)\otimes Z_q), \
\CE^{q,n}_{l/r}:=\CE^n_{r/l}\cap W_q,\\
\CE^{fin,n }_{l/r}&:=&\left\{u\in \CE_{l/r}^n: \,\exists Q>0,\, u\in \oplus_{q\le Q} \CE^{q,n}_{l/r}\right\}.
\end{eqnarray*}
\begin{theorem}
\label{Asympcompl2}
There exists $a_0>0$ such that for all $|a|<a_0$  and $n\in \zz\setminus\{0\}$ the following holds:
\begin{enumerate}
\item[i)]
For all $u\in \CE^{fin,n}_{r/l}$ the limits
\begin{equation*}
W_{r/l}u=\lim_{t\rightarrow\infty}e^{it\dot{H}^n}i^2_{r/l}e^{-it\dot{H}^n_{r/l}}u
\end{equation*}
exist in $\dot{\CE}^n$. The operators $W_{r/l}$ extend to bounded
operators $W_{r/l}\in \CB(\dot{\CE}^n_{r/l};\dot{\CE}^n)$. 
\item[ii)] The inverse wave operators 
\begin{equation*}
\Omega_{r/l}=\slim_{t\rightarrow\infty}e^{it\dot{H}^n_{r/l}}i^2_{r/l}e^{-it\dot{H}^n}
\end{equation*}
exist in $\CB(\dot{\CE}^n;\dot{\CE}^n_{r/l})$.
\end{enumerate}
$i),\, ii)$ also hold for $n=0$ if $m>0$.
\end{theorem}
\section{Proof of the main theorems for the De Sitter Kerr spacetime} 
\label{prDSKerr}
We want to apply the geometric setting developed in Sect. \ref{SecGS}. To
do so, we have to reduce ourselves to $\ell_+=0$ by a change of coordinates. We introduce the new coordinate
\[
\tilde{\varphi}=\varphi-\frac{a}{r_+^2+a^2}t,
\] the other coordinates
remain unchanged. We will denote $\tilde{\varphi}$ again by $\varphi$
in the following. In the new coordinates (\ref{eqpsiv1}) writes :
\begin{eqnarray}
\lefteqn{\left(\left(\partial_t-\frac{a}{r_+^2+a^2}\partial_{\varphi}\right)^2-2\frac{a(\Delta_r-(r^2+a^2)\Delta_{\theta})}{\sigma^2}\partial_{\varphi}\left(\partial_t-\frac{a}{r_+^2+a^2}\partial_{\varphi}\right)-\frac{(\Delta_r-a^2\sin^2\theta\Delta_{\theta})}{\sin^2\theta\sigma^2}\partial_{\varphi}^2\right.}\nonumber\\
&-&\left.\frac{\sqrt{\Delta_r\Delta_{\theta}}}{\lambda\sigma}\partial_r\Delta_r\partial_r
    \frac{\sqrt{\Delta_r\Delta_{\theta}}}{\lambda\sigma}-\frac{\sqrt{\Delta_r\Delta_{\theta}}}{\lambda\sin\theta\sigma}\partial_{\theta}\sin\theta\Delta_{\theta}\partial_{\theta}\frac{\sqrt{\Delta_r\Delta_{\theta}}}{\lambda\sigma}+\frac{\rho^2\Delta_r\Delta_{\theta}}{\lambda^2\sigma^2}m^2\right)\psi=0,
\end{eqnarray}
i.e.
\begin{eqnarray}
\label{equ}
\lefteqn{\left(\partial_t^2-2\left(\frac{a}{(r_+^2+a^2)}+\frac{a(\Delta_r-(r^2+a^2)\Delta_{\theta})}{\sigma^2}\right)\partial_{\varphi}\partial_t\right.}\nonumber\\
&+&\left.\left(\frac{a^2}{(r_+^2+a^2)^2}+2\frac{a^2(\Delta_r-(r^2+a^2)\Delta_{\theta})}{\sigma^2(r_+^2+a^2)}-\frac{(\Delta_r-a^2\sin^2\theta\Delta_{\theta})}{\sin^2\theta\sigma^2}\right)\partial_{\varphi}^2\right.\nonumber\\
&-&\left.\frac{\sqrt{\Delta_r\Delta_{\theta}}}{\lambda\sigma}\partial_r\Delta_r\partial_r
    \frac{\sqrt{\Delta_r\Delta_{\theta}}}{\lambda\sigma}-\frac{\sqrt{\Delta_r\Delta_{\theta}}}{\lambda\sin\theta\sigma}\partial_{\theta}\sin\theta\Delta_{\theta}\partial_{\theta}\frac{\sqrt{\Delta_r\Delta_{\theta}}}{\lambda\sigma}+\frac{\rho^2\Delta_r\Delta_{\theta}}{\lambda^2\sigma^2}m^2\right)\psi=0.
\end{eqnarray}
Let us put:
\begin{eqnarray*}
k&:=&\left(\frac{a}{(r_+^2+a^2)}+\frac{a(\Delta_r-(r^2+a^2)\Delta_{\theta})}{\sigma^2}\right) D_{\varphi},\\
h&:=&\left(\frac{a^2}{(r_+^2+a^2)^2}+2\frac{a^2(\Delta_r-(r^2+a^2)\Delta_{\theta})}{\sigma^2(r_+^2+a^2)}-\frac{(\Delta_r-a^2\sin^2\theta\Delta_{\theta})}{\sin^2\theta\sigma^2}\right)\partial_{\varphi}^2\\
&-&\frac{\sqrt{\Delta_r\Delta_{\theta}}}{\lambda\sigma}\partial_r\Delta_r\partial_r
    \frac{\sqrt{\Delta_r\Delta_{\theta}}}{\lambda\sigma}-\frac{\sqrt{\Delta_r\Delta_{\theta}}}{\lambda\sin\theta\sigma}\partial_{\theta}\sin\theta\Delta_{\theta}\partial_{\theta}\frac{\sqrt{\Delta_r\Delta_{\theta}}}{\lambda\sigma}+\frac{\rho^2\Delta_r\Delta_{\theta}}{\lambda^2\sigma^2}m^2.
\end{eqnarray*}
Noting that the coordinate change $\varphi\rightarrow \tilde{\varphi}$
corresponds to the unitary transform $e^{-it\Omega_+D_{\varphi}}$ and
using Subsect. \ref{SecGT} we see that it is sufficient to show the
corresponding theorems of Sect. \ref{SecAC3} for the operators $h,k$. We put $h_0:=h+k^2.$ A tedious calculation gives :
\begin{eqnarray}
\label{h0kerr}
h_0&=&-\frac{\rho^4\Delta_r\Delta_{\theta}}{\sigma^4\sin^2\theta}\partial_{\varphi}^2\nonumber\\
&-&\frac{\sqrt{\Delta_r\Delta_{\theta}}}{\lambda\sigma}\partial_r\Delta_r\partial_r
    \frac{\sqrt{\Delta_r\Delta_{\theta}}}{\lambda\sigma}-\frac{\sqrt{\Delta_r\Delta_{\theta}}}{\lambda\sin\theta\sigma}\partial_{\theta}\sin\theta\Delta_{\theta}\partial_{\theta}\frac{\sqrt{\Delta_r\Delta_{\theta}}}{\lambda\sigma}+\frac{\rho^2\Delta_r\Delta_{\theta}}{\lambda^2\sigma^2}m^2.
\end{eqnarray}
We put 
\begin{eqnarray}
\label{h0kerrn}
h^n_0&:=&\frac{(\rho^4-\sigma^2)\Delta_r\Delta_{\theta}}{\sigma^4\sin^2\theta}n^2\nonumber\\
&-&\frac{\sqrt{\Delta_r\Delta_{\theta}}}{\lambda\sigma}\partial_r\Delta_r\partial_r
    \frac{\sqrt{\Delta_r\Delta_{\theta}}}{\lambda\sigma}+\frac{\sqrt{\Delta_r\Delta_{\theta}}}{\lambda\sigma}P\frac{\sqrt{\Delta_r\Delta_{\theta}}}{\lambda\sigma}+\frac{\rho^2\Delta_r\Delta_{\theta}}{\lambda^2\sigma^2}m^2,\\
\label{h0kerrkn}
k^n&:=&\left(\frac{a}{(r_+^2+a^2)}+\frac{a(\Delta_r-(r^2+a^2)\Delta_{\theta})}{\sigma^2}\right)n.
\end{eqnarray}

We will drop in the following the index $n$ which is implicit in the
operators. 
\subsection{Verification of the geometric hypotheses}
Let us recall that
\[P=-\frac{\lambda^2}{\sin^2\theta\Delta_{\theta}}\partial_{\varphi}^2-\frac{1}{\sin
    \theta}\partial_{\theta}\sin\theta\Delta_{\theta}\partial_{\theta}.\]
With this choice of $P$ \eqref{G1} is clearly fulfilled. 
We now put 
\begin{eqnarray*}
h_{0,s}&:=&-\frac{\sqrt{\Delta_r}}{\lambda(r^2+a^2)}\partial_r\Delta_r\partial_r
\frac{\sqrt{\Delta_r}}{\lambda(r^2+a^2)}+\frac{\sqrt{\Delta_r}}{\lambda(r^2+a^2)}P\frac{\sqrt{\Delta_r}}{\lambda(r^2+a^2)}+\Delta_rm^2.
\end{eqnarray*}
Recall that $q(r):=\sqrt{(r_+-r)(r-r_-)}$. We write
$\Delta_r=q^2(r)P_2(r)$, where $P_2$ is a polynomial of degree 2. It is easy to see that \eqref{G2} is fulfilled with
\begin{equation*}
\alpha_1^{\pm}=\alpha_3^{\pm}:=\frac{\sqrt{P_2(r_{\pm})}}{\lambda(r_{\pm}^2+a^2)},\, \alpha_2^{\pm}:=\sqrt{P_2(r_{\pm})},\,
 \alpha_4^{\pm}:=m^2\frac{\sqrt{P_2(r_{\pm})}}{(r_{\pm}^2+a^2)\lambda^2}.
\end{equation*}
We also put
\begin{eqnarray*}
k_{s,v}&:=&k_s^n:=\left(\frac{a}{r_+^2+a^2}-\frac{a}{r^2+a^2}\right)n,\\
k^-_{s,v}&:=&\frac{an}{(r_+^2+a^2)(r_-^2+a^2)}(r_--r_+)(r_-+r_+),\\
k_{s,r}&=&k_{s,r}^-=0.
\end{eqnarray*}
With these choices \eqref{G3} is clearly fulfilled.  \eqref{intheta1} follows
from \eqref{h0kerrn}. Let
\begin{eqnarray*}
g_1&:=&\frac{\sqrt{\Delta_r\Delta_{\theta}}}{\lambda\sigma}\in
T^1,\quad g_0:=\frac{\sqrt{\Delta_r}}{\lambda(r^2+a^2)}\in T^1,\\
p_1&:=&\frac{\sqrt{\Delta_r\Delta_{\theta}}}{\sigma}\in T^1,\quad p_0:=\frac{\sqrt{\Delta_r}}{r^2+a^2}\in T^1.
\end{eqnarray*}
We have
\begin{eqnarray*}
g_1-g_0&\in&T^3,\quad p_1-p_0\in T^3.
\end{eqnarray*}
An elementary calculation gives :
\begin{eqnarray*}
g^{rr}&=&(g_0-g_1)(g_0+g_1)\Delta_r\in T^5,\\
g^r&=&i((\partial_rg_1)g_1-(\partial_rg_0)g_0)\Delta_r\in T^3,\\
g^{\theta\theta}&=&(p_0-p_1)(p_0+p_1)\in T^3,\\
g^{\theta}&=&i((\partial_{\theta}p_1)p_1-(\partial_{\theta}p_0)p_0)\in T^2,\\
f&=&((\partial_rg_0)^2-(\partial_rg_1)^2)\Delta_r+((\partial_{\theta}p_0)^2-(\partial_{\theta}p_1)^2)+\frac{m^2\Delta_r
  \rho^2\Delta_{\theta}}{\lambda^2\sigma^2}-\Delta_r m^2\\
&+&\left(\frac{\rho^4\Delta_r\Delta_{\theta}}{\sigma^4\sin^2\theta}-\frac{\Delta_r}{(r^2+a^2)^2\sin^2\theta}\right)n^2\in
T^2,\\
g^{\varphi\varphi}&=&g^{\varphi}=0.
\end{eqnarray*}
Note that because of the diagonalization w.r.t.  $D_{\varphi}$, we can put
$g^{\varphi\varphi}$ and $g^{\varphi}$ into $f$. It follows that
hypothesis \eqref{G7} is fulfilled. Let us now check
\eqref{bornehp}. We consider the case $n=0$ only if $m>0$. Also
\eqref{bornehp} will only be satisfied if $|a|<a_1$ for some $a_1$
independent of $n$.
We first show that \eqref{bornehp} is fulfilled for $h_{0,s}^n$. We
have 
\begin{eqnarray*}
h_{0,s}^n&=&\alpha_1(D_r\Delta_rD_r+P+(r^2+a^2)^2m^2)\alpha_1\\
&\gtrsim&\alpha_1(D_rq^2D_r+P+1)\alpha_1,
\end{eqnarray*}
with $\alpha_1=\frac{\sqrt{\Delta_r}}{\lambda(r^2+a^2)}$ because $P\gtrsim 1$ for $n\neq 0$ and we suppose $m>0$ for $n=0$. Let now
\begin{eqnarray*}
\tilde{h}_{0}^n&=&\frac{\rho^4\Delta_r\Delta_{\theta}}{\sigma^2\Delta_{\theta}\sin^2\theta}n^2-\frac{\sqrt{\Delta_r}}{\lambda(r^2+a^2)}\partial_r\Delta_r\partial_r
\frac{\sqrt{\Delta_r}}{\lambda(r^2+a^2)}\\
&-&\frac{\sqrt{\Delta_r}}{\lambda(r^2+a^2)\sin\theta}\partial_{\theta}\sin
\theta\Delta_{\theta}\partial_{\theta}\frac{\sqrt{\Delta_r}}{\lambda(r^2+a^2)}+\frac{\rho^2\Delta_rm^2\Delta_{\theta}}{\lambda^2\sigma^2}\\
&\gtrsim&\alpha_1(D_rq^2D_r+P+1)\alpha_1.
\end{eqnarray*}
We then compute
\begin{eqnarray*}
h_0^n-\tilde{h}_{0}^n&=&-\left(\frac{1}{\sigma}-\frac{1}{(r^2+a^2)\sqrt{\Delta_{\theta}}}\right)\frac{\sqrt{\Delta_r\Delta_{\theta}}}{\lambda}\partial_r\Delta_r\partial_r
\frac{\sqrt{\Delta_r\Delta_{\theta}}}{\lambda\sigma}\\
&-&\frac{\sqrt{\Delta_r}}{\lambda(r^2+a^2)}\partial_r\Delta_r\partial_r
\frac{\sqrt{\Delta_r\Delta_{\theta}}}{\lambda}\left(\frac{1}{\sigma}-\frac{1}{(r^2+a^2)\sqrt{\Delta_{\theta}}}\right)\\
&-&\left(\frac{1}{\sigma}-\frac{1}{(r^2+a^2)\sqrt{\Delta_{\theta}}}\right)
\frac{\sqrt{\Delta_r\Delta_{\theta}}}{\lambda\sin\theta}\partial_{\theta}\sin\theta\Delta_{\theta}\partial_{\theta}\frac{\sqrt{\Delta_r\Delta_{\theta}}}{\lambda\sigma}\\
&-&\frac{\sqrt{\Delta_r}}{(r^2+a^2)\lambda\sin
  \theta}\partial_{\theta}\sin \theta\Delta_{\theta}\partial_{\theta}\frac{\sqrt{\Delta_r}}{\lambda}\left(\frac{1}{\sigma}-\frac{1}{(r^2+a^2)\sqrt{\Delta_{\theta}}}\right).
\end{eqnarray*}
We compute
\begin{eqnarray*}
\left(\frac{1}{\sigma}-\frac{1}{(r^2+a^2)\sqrt{\Delta_{\theta}}}\right)&=&\frac{a^2\Delta_r\sin^2\theta}{\sigma^2(r^2+a^2)^2\Delta_{\theta}}\left(\frac{1}{\sigma}+\frac{1}{(r^2+a^2)\Delta_{\theta}}\right)^{-1}=:a^2g_a.
\end{eqnarray*}
We have $ g_a\in T^2$ uniformly in $a$ meaning that 
\begin{equation*}
\forall\, \alpha,\beta\in \N,\quad\vert \partial^{\alpha}_r\partial_{\omega}^{\beta}f_a\vert\le C_{\alpha\beta}q(r)^{2-2\alpha}
\end{equation*}
with $C_{\alpha,\beta}$ independent of $a$. We then compute
\begin{eqnarray*}
-a^2g_a \frac{\sqrt{\Delta_r\Delta_{\theta}}}{\lambda}\partial_r\Delta_r\partial_r
\frac{\sqrt{\Delta_r\Delta_{\theta}}}{\lambda\sigma}&=&-a^2
\frac{\sqrt{\Delta_r\Delta_{\theta}}}{\lambda\sigma}\partial_r\tilde{g}_a\Delta_r\partial_r\frac{\sqrt{\Delta_r\Delta_{\theta}}}{\lambda\sigma}\\
&+&a^2\tilde{g}'_a
\frac{\sqrt{\Delta_r\Delta_{\theta}}}{\lambda\sigma}\Delta_r\partial_r\frac{\sqrt{\Delta_r\Delta_{\theta}}}{\lambda\sigma}\\
&\gtrsim&-a^2h^n_0,
\end{eqnarray*}
where $\tilde{g}_a=g_a\sigma\in T^2$ uniformly in $a$. Using similar
arguments for the other terms we find 
\[h_0^n-\tilde{h}_0^n\gtrsim-a^2h_0^n-a^2\tilde{h}_0^n\]
and thus for $a$ small enough (independently of $n$):
\begin{equation}
\label{G5h0}
h_{0}^n\gtrsim \tilde{h}_{0}^n\gtrsim
\alpha_1(D_rq^2D_r+P+1)\alpha_1.
\end{equation}
This is \eqref{bornehp} for $h_0^n$. We also have
\begin{equation*}
k_{p,v}=\left(\frac{a\Delta_r}{\sigma^2}+\frac{a^3\sin^2\theta\Delta_r}{(r^2+a^2)\sigma^2}\right)n\in
T^2,\quad k_{p,r}=0.
\end{equation*}
This implies \eqref{G8}. Let us now check \eqref{G9}. Recall that
$\ell=k_{s,v}^-$. We construct $h_{+}$, $\tilde{h}_{-}$ and $k_{\pm}$ as in Subsect. \ref{subitopresto}. We have 
\[
h_+=h_0-k_+^2,\ k^2_+\le
C_+a^2n^2(r_+-r)(r-r_-).
\] 
Observing that $P\ge \frac{n^2}{\sin^2\theta}$ we obtain using
\eqref{G5h0}:
\begin{equation*}
h_0^n\gtrsim \alpha_1(D_rq^2D_r+P+1)\alpha_1+n^2(r_+-r)(r-r_-). 
\end{equation*}
Thus fixing a cut-off scale $\epsilon>0$ (see Section
\ref{asympham}) there exists $a_{2}>0$ (independent of $n$) such that for all
$|a|<a_2$ and $n\in \Z$ we have :
\[
h_+, \  \tilde{h}_{-}\gtrsim \alpha_1(r)(D_rq^2(r)D_r+\tilde{P}+1)\alpha_1(r).
\]
In particular \eqref{bornehp} is fulfilled for $h_+,\, \tilde{h}_-$ if
$a$ is small enough. Thus \eqref{G9} is fulfilled. In the sequel we assume 
$|a|<a_0$, where $a_0$ is such that all the geometric hypotheses are
fulfilled for all $n\in \Z\setminus \{0\}$ ($m=0$) resp. all
$n\in \Z$ ($m>0)$ if $|a|<a_0$.
\subsection{Proof of Thm. \ref{thunifbound}}
Thm. \ref{thunifbound}  will follow from Thm. \ref{thboundedness}, provided we show that the set $\CS$ of singular points is empty.
 We recall that the sets  $\CS$, $\CT$, $\CT_{\pm}$ were defined in Subsect. \ref{fai} and that we showed in Prop. \ref{thm5.3.1} that  $\CS\subset \CT\cup \CT_{-}\cup \CT_{+}$. Therefore the proof of Thm. \ref{thunifbound} will  follow from
 \begin{proposition}\label{allegromanotropo}
\begin{enumerate}
\item[i)] There exists $a_{1}>0$ such that for  $|a|<a_{1},\, n\neq
  0$: 
\[
\sigma^{\cc}_{\rm pp}(\dot{H})= \cT=\emptyset.
\]
\item[ii)] One has $\CT_{\pm}= \emptyset$ for $n\neq 0$.
\end{enumerate}
\end{proposition}
\proof
{\it i)} essentially follows from the work of
Dyatlov \cite{Dy11_01}.
Let us first prove that $\sigma^{\cc}(\dot{H})=\emptyset$. 
By \cite[Thm. 4]{Dy11_01} we have
$\rho(h,k)\cap\{ {\rm Im}z>0\}=\emptyset$ for $a>0$ sufficiently
small. Then we apply Prop.  \ref{lemmaresHdot}.

Let us now prove that $\mathcal{T}= \emptyset$, i.e. that 
$r(z): =w^{- \epsilon}p^{-1}(z)w^{-\epsilon}$
has no real poles. We replace the weight $w^{-\epsilon}$ by $\cosh(\epsilon x)^{-1}$
which is equivalent and holomorphic in a neighborhood of the real
axis. We will note this new weight again by $w^{-\epsilon}$. 
We know that $r(z)$ has a meromorphic extension to $\{{\rm Im}
z>-\delta_{\epsilon}\}$ for some $\delta_{\epsilon} >0$. We note again
$r(z)$ this meromorphic extension. Let
\[
\tilde{p}(z)= \tilde{p}(z, x, \p_{x}):= w^{\epsilon} p(z)w^{\epsilon}.
\]
This is an elliptic second order operator with analytic
coefficients. We clearly have
\beq\label{e1}
r(z)\circ \tilde{p}(z)= \tilde{p}(z)\circ r(z)=\one,
\eeq
first for ${\rm Im} z$ sufficiently large and then in $\{ {\rm
  Im}z>-\delta_{\epsilon}\}$ by meromorphic extension. Let 
$K_{z}(x, x')$ the distribution kernel of $r(z)$. We have : 
\[
 \tilde{p}(z)(x, \p_{x})K_{z}(x, x')= \delta(x, x'), \ z\in \Omega,
\]
\[
\tilde{p}(z)^{t}(x', \p_{x'})K_{z}(x, x')= \delta(x, x'),\ z\in \Omega,
\]
where  $\tilde{p}(z)^{t}$ is the transpose of $\tilde{p}(z)$, and  is also elliptic with analytic
coefficients. By the Morrey-Nirenberg theorem \cite[Thm. 7.5.1]{Ho},
$\tilde{p}(z)$ and $\tilde{p}(z)^{t}$ are {\em analytic hypo-elliptic}, 
from which we obtain that 
 $K_{z}(x, x')$ is analytic in $x,x'$ outside the diagonal, for $\{{\rm Imz}>-\delta_{\epsilon}\}$.

Recall from \cite{Dy11_01} that there exists $\delta_r>0$ such that for
all $\eta\in C_0^{\infty} ((r_-+\delta_r,r_+-\delta_r)$ $\eta
p^{-1}(z)\eta$ has no poles in $\{{\rm Im}z>-\delta_0\}$ for some $\delta_0>0$. 
Let now $z_{0}\in \{{\rm Im} z>-\delta_0\}$ be a possible pole of
$r(z)$. We write:
\[
r(z)= \sum_{i=1}^{N}P_{j}(z- z_{0})^{-j}+ H(z),
\] 
where $P_{j}$ are finite rank operators and $H(z)$ is holomorphic close
to  $z_{0}$ and $P_{N}\neq 0$. We want to show that all the $P_{j}$
are zero. Clearly it is sufficient to show that $P_{N}=0$.

We have
\[
P_{N}= \frac{1}{2\i \pi}\ointctrclockwise_{\gamma}(z-z_{0})^{N-1} r(z)dz,
\]
which shows that the kernel $P_{N}(x, x')$ of $P_{N}$ is analytic
outside the diagonal. But as $\eta p(z)^{-1}\eta$ has no poles, we
necessarily have $P_{N}(x, x')=0$ for $x\neq x'$, $x, x'\in \supp
\eta$. By analytic continuation  we therefore have $P_{N}(x, x')=0$ for
$x\neq x'$. We then have
 \[
\begin{array}{rl}
\tilde{p}(z_{0})P_{N}=&\frac{1}{2\i \pi}\ointctrclockwise_{\gamma}\tilde{p}(z_{0})(z-z_{0})^{N-1} r(z)dz\\[2mm]
=&\frac{1}{2\i \pi}\ointctrclockwise_{\gamma}(z-z_{0})^{N-1}(\tilde{p}(z_{0})- \tilde{p}(z)) r(z)dz\\[2mm]
&+\frac{1}{2\i \pi}\ointctrclockwise_{\gamma}(z-z_{0})^{N-1}\tilde{p}(z) r(z)dz.
\end{array}
\]
As $\tilde{p}(z)- \tilde{p}(z_{0})= (z- z_{0})T(z)$ with $T(z)$
holomorphic  close to $z_{0}$, the first term is zero, the second is zero because $\tilde{p}(z) r(z)=\one$.
It follows that $\tilde{p}(z_{0})P_{N}=0$.

Let us show that this implies $P_{N}=0$. Let $u\in L^{2}(\rr\times
S^{2})$ with compact support. As the distribution kernel of  $P_{N}$
is supported on the diagonal, $v= P_{N}u$ has also compact support and
 $\tilde{p}(z_{0})v=0$. Again by analytic hypo-ellipticity of
$\tilde{p}(z_{0})$, we obtain that $v$ is analytic with compact
support, thus $v=0$. By a density argument we obtain $P_{N}=0$. This completes the proof of {\it i)}.

Let us now prove {\it ii)}. 
By  \cite [Proposition 9.3]{GGH2} we know that $\CT_{\pm}\cap \rr\backslash\{0\}= \emptyset$.
By Corollary \ref{Cor10.2.4} it is sufficient to show that $0$ is not a resonance
of $w^{-\epsilon}p^{-1}_{\pm}(z)w^{-\epsilon}$. We treat the $+$ case, the $-$
case being analogous. Suppose that $0$ is a resonance. First note that
$p_+(0)$ is an elliptic operator with
$p_+(0)\gtrsim n^2 w^{-2}$. In particular $w^{\epsilon}p_+(0)w^{\epsilon}v=0$ implies $v=0$. Let
$r(z)=w^{-\epsilon}p_+^{-1}(z)w^{-\epsilon}$. Suppose that $r(z)$ has a
pole at $z=0$:
\begin{equation*}
r(z)=\sum_{j=1}^N\frac{P_j}{z^j}+H(z),\quad P_N\neq 0.
\end{equation*}
Here $P_j$ are of finite rank and $H(z)$ is holomorphic. 
Let $u\in \CH$ such that $P_Nu\neq 0$. We have 
\begin{equation*}
z^Nu=\sum_{j=1}^Nz^{N-j}w^{\epsilon}p_+(z)w^{\epsilon}P_ju+z^Nw^{\epsilon}p_+(z)w^{\epsilon}H(z)u.
\end{equation*}
In the limit $z\rightarrow 0$ we obtain
$w^{\epsilon}p_+(z)w^{\epsilon}P_Nu=0$ and thus $P_Nu=0$ which is a
contradiction.  This completes the proof of the proposition. 
\qed

\subsection{Proof of Thm. \ref{thACKerrSC}}
We will apply here the results of Sect.  \ref{secACgeo}. First note
that in our new setting (i.e. after rotation) we have to consider
\begin{eqnarray*}
h_{-\infty}&:=&-\ell^2-\partial_x^2+\frac{\Delta_r}{r^2+a^2}P+\frac{\Delta_rm^2}{\lambda^2(r^2+a^2)},\\
k_{-\infty}&:=&\ell,\\
h_{+\infty}&:=&-\partial_x^2+\frac{\Delta_r}{r^2+a^2}P+\frac{\Delta_rm^2}{\lambda^2(r^2+a^2)},\\
k_{+\infty}&:=&0,\\
\ell&:=&\left(\frac{a}{r_+^2+a^2}-\frac{a}{r_-^2+a^2}\right)n.
\end{eqnarray*}
We associate to these operators the operators $H_{\pm\infty},\,
\dot{H}_{\pm\infty}$ and spaces $\CE_{\pm\infty},\,
\dot{\CE}_{\pm\infty}$ as in Sect.  \ref{SecBG}. Let
$\CT_{\pm\infty}$ be the set of singular points of
$\dot{H}_{\pm\infty}$. 
We have : 
\begin{lemma}
\label{AbspHpminfty}
For $n\neq 0$ we have $\CT_{\pm\infty}=\emptyset$.
\end{lemma}
\proof

As $\dot{H}_{\pm\infty}$ is selfadjoint we can use the Kato theory of
$H-$ smoothness. The proof for the absence of real resonances is
analogous to the proof of Prop. \ref{allegromanotropo} {\it ii)}, we omit the details. 
\qed

\subsection{Proof of Thm. \ref{thACKerrSC}}
We first consider the case $n\neq 0$. By Prop. \ref{allegromanotropo} we know that
$\sigma_{pp}^{\C}(\dot{H})=\CS=\CT_{\pm\infty}=\emptyset$. Thus
$\one=\one_{\R}(\dot{H})$ is an admissible energy cut-off. Using in
addition that $e^{-it\dot{H}},\, e^{-it\dot{H}_{\pm\infty}}$ are
uniformly bounded, the theorem follows from Thm. \ref{asympcompl2}. In
the case $n=0$ all operators are selfadjoint. This case then follows from
\cite{Ha03}, we omit the details. 
\qed
\subsection{Proof of Thm. \ref{Asympcompl2}}
We first write the comparison dynamics which we obtain after rotation:
\begin{eqnarray*}
h_r&=&-\partial_x^2,\, h_l=-\partial_x^2-\ell^2,\\
k_r&=&0,\, k_l=\ell.
\end{eqnarray*}
We associate to these operators the natural homogeneous  energy spaces
$\dot{\CE}_{l/r}$. Let 
\begin{equation*}
\dot{H}_r=\left(\begin{array}{cc} 0 & \one \\ h_r &
    2k_r \end{array}\right),\quad \dot{H}_l=\left(\begin{array}{cc} 0 & \one \\ h_l &
    2k_l \end{array}\right).
\end{equation*}
We now further analyze the energy spaces. Note that
\begin{eqnarray*}
\dot{\CE}_l&=&\Phi(\ell)(H^1(\R);L^2(S^2))\oplus L^2(\R\times S^2)),\\
\dot{\CE}_r&=&H^1(\R;L^2(S^2))\oplus L^2(\R\times S^2).
\end{eqnarray*}
We will need the following subspaces
\begin{eqnarray*}
\dot{\CE}_l^L&=&\{(u_0,u_1)\in \dot{\CE}_l;\, u_1-i\ell u_0\in L^1(\R;L^2(S^2)),\,
\int (u_1-i\ell u_0) (x,\omega)dx=0\quad \mbox{a.e. in}\quad \omega\},\\
\dot{\CE}_r^L&=&\{(u_0,u_1)\in \dot{\CE}_r;\, u_1\in L^1(\R;L^2(S^2)),\,
\int u_1(x,\omega)dx=0\quad \mbox{a.e. in}\quad \omega\}.
\end{eqnarray*}
We define the spaces of {\em incoming /outgoing} initial data:
\begin{eqnarray*}
\dot{\CE}_l^{\rm in}&=&\{u\in \dot{\CE}^L_l;\,
u_1=\partial_xu_0+i\ell u_0\},\\
\dot{\CE}_l^{\rm out}&=&\{u\in \dot{\CE}^L_l;\,
u_1=-\partial_xu_0+i\ell u_0\},\\
\dot{\CE}_r^{\rm in}&=&\{u\in \dot{\CE}^L_r;\,
u_1=\partial_xu_0\},\\
\dot{\CE}_r^{\rm out}&=&\{u\in \dot{\CE}^L_r;\,
u_1=-\partial_xu_0\}.
\end{eqnarray*}
If $(u_0,u_1)\in \dot{\CE}^{\rm in}_l$, then the solution of
\eqref{profilwe} is given by
\[u_0(t,x,\omega)=e^{i\ell t}u_0(x+t,\omega),\]
which is clearly incoming. We have the following
\begin{lemma}
\label{lem12.2.1}
We have 
\begin{equation*}
\dot{\CE}^L_l=\dot{\CE}^{\rm in}_l\oplus \dot{\CE}^{\rm out}_l,\quad \dot{\CE}^L_r=\dot{\CE}^{\rm in}_r\oplus\dot{\CE}^{\rm out}_r.
\end{equation*}
\end{lemma}
\proof 
We only show the lemma for $\dot{\CE}_l^L$, $\dot{\CE}_r^L$ being the special case
$\ell=0$.
We define for $u=(u_0,u_1)\in \dot{\CE}^L_l$:
\begin{eqnarray}
\label{equation12.6}
\begin{array}{rcl}
u^{\rm in}_{0}&=&\frac{1}{2}\int_x^{\infty}(-\partial_xu_0-(u_1-i\ell u_0))(\tau,\omega)d\tau,\\[2mm]
u^{\rm in}_{1}&=&\frac{1}{2}(u_1-i\ell
u_0+\partial_xu_0)+\frac{i\ell}{2}\int_x^{\infty}(-\partial_xu_0-(u_1-i\ell
u_0))(\tau,\omega)d\tau,\\[2mm]
u^{\rm out}_{0}&=&\frac{1}{2}\int_{-\infty}^x(\partial_xu_0-(u_1-i\ell
u_0))(\tau,\omega)d\tau,\\[2mm]
u^{\rm out}_{1}&=&\frac{1}{2}(u_1-i\ell
u_0-\partial_xu_0)+\frac{i\ell}{2}\int_{-\infty}^x(\partial_xu_0-(u_1-i\ell
u_0))(\tau,\omega)d\tau,\\[2mm]
u^{{\rm in}/{\rm out}}&=& (u_{0}^{{\rm in}/{\rm out}}, u_{1}^{{\rm in}/{\rm out}}).
\end{array}
\end{eqnarray}
It is easy to check that
\[
u= u^{\rm in}+ u^{\rm out}, \ u^{{\rm int}/ {\rm out}}\in \dot{\CE}_{\ell}^{{\rm in}/ {\rm out}},
\]
\def\ino{{\rm in}}\def\outo{{\rm out}}\def\doti{\dot{\CE}}
which shows that $\doti^{\ino}_{\ell}+ \doti^{\outo}_{\ell}= \doti_{\ell}$. Next if $v\in \doti^{\ino}_{\ell}\cap \doti^{\outo}_{\ell}$ we have $\p_{x}v_{0}=0$, $v_{0}\in L^{2}$ hence $v_{0}=0$, hence $v_{1}=0$.  \qed
\begin{remark}
\label{suppinout}
If $(u_0,u_1)\in \dot{\CE}^L_l$ or $(u_0,u_1)\in \dot{\CE}^L_r,\, \supp u_0,\, \supp
u_1\subset]R_1,R_2[\times S^2$, then we have
\begin{equation}
\label{suppin}
\supp u^{\rm in}_{0,1}\subset]-\infty,R_2[\times S^2,\quad
\supp u^{\rm out}_{0,1}\subset ]R_1,\infty[\times S^2.
\end{equation}
\end{remark}
The spaces $\CE^q_{l/r},\, \CE^{fin}_{l/r}$ are defined as before but
starting with our now slightly modified operators (due to the
rotation). Let
\begin{equation*}
\CD^{fin}_{r/l}=C_0^{\infty}(\R\times S^2)\times C_0^{\infty}(\R\times
S^2)\cap\CE^{fin}_{r/l}\cap\dot{\CE}_{r/l}^L.
\end{equation*}
\begin{lemma}
\label{lem12.2.2}
$\CD^{fin}_{r/l}$ is dense in $\CE^{fin}_{r/l}$.
\end{lemma}   
\proof
We prove the lemma in two steps.
\begin{enumerate}
\item[--] $C_0^{\infty}(\R\times S^2)\times C_0^{\infty}(\R\times S^2)\cap\CE^{fin}_{r/l}$ is
  dense in $\CE^{fin}_{r/l}$. This follows easily from the usual
  regularization procedures.
\item[--] 
\begin{equation*}
C_0^{\infty}(\R\times S^2)\times C_0^{\infty}(\R\times S^2)\cap
\CE_{r/l}^{fin}\cap\dot{\CE}_{r/l}^L\,\mbox{is dense in}\,\,
C_0^{\infty}(\R\times S^2)\times C_0^{\infty}(\R\times
S^2)\cap\CE^{fin}_{r/l}.
\end{equation*}
We can clearly replace $\CE^{fin}_{r/l}$ by $\CE^q_{r/l}$ in the
statement. We only treat the $l-$ case. Let 
\begin{equation*}
u=(u_0,u_1)\in C_0^{\infty}(\R\times S^2)\times C_0^{\infty}(\R\times
S^2)\cap\CE^q_{l}. 
\end{equation*}
We will consider $u$ as a function of $x$ alone. We put 
\[v=\Phi(-\ell)u.\]
Let $\psi\in C_0^{\infty}(\R),\, \psi\ge 0,\, \psi= 1$ in a
neighborhood of zero and $\int\psi(x)dx=1$. We put 
\begin{equation*}
v_0^n=v_0,\quad
v_1^n=v_1-n^{-1}\psi(n^{-1}x)\int v_1(x)dx,
\end{equation*}
so that 
$\int v_1^n(x)dx=0$. We then estimate
\begin{eqnarray*}
\| v_{1}- v_{1}^{n}\|_{L^{2}}\leq n^{-\12}\| v_{1}\|_{L^{1}}\| n^{-\12}\psi(n^{-1}\cdot)\|_{L^{2}}\leq Cn^{-\12}
\|v_1\|_{L^1}\rightarrow 0,
\end{eqnarray*}
which  completes the proof. 
\end{enumerate}
\qed

We need an additional
\begin{lemma}
\label{normcomp}
There exists $C>0$ such that 
\[
\Vert i_{r/l}u\Vert_{\dot{\CE}_{r/l}}\le C\Vert
u\Vert_{\dot{\CE}}, \ u\in \doti.
\]

\end{lemma}
\proof 
We have:
\begin{eqnarray*}
\Vert i_ru\Vert^2_{\dot{\CE}_r}&=&\Vert
i_ru_1\Vert_{\CH}^2+(i_rh_{+\infty}i_ru_0|u_0)\\
&\lesssim&\Vert (u_1-ku_0)\Vert^2_{\CH}+(i_r(h_{+\infty}+k^2)i_ru_0|u_0)\\
&\lesssim&\Vert (u_1-ku_0)\Vert^2_{\CH}+(h_0u_0|u_0)=\Vert
u\Vert^2_{\dot{\CE}}.
\end{eqnarray*}
Now recall that
\[\tilde{h}_{-\infty}=-\partial_x^2+\frac{\Delta_r}{r^2+a^2}P+m^2\Delta_r.\]
We then estimate
\begin{eqnarray*}
\Vert i_lu\Vert^2_{\dot{\CE}_l}&=&\Vert i_l(u_1-\ell
u_0)\Vert^2+(i_l\tilde{h}_{-\infty}i_lu_0|u_0)\\
&\lesssim&\Vert
i_l(u_1-ku_0)\Vert^2_{\CH}+(i_l(\tilde{h}_{-\infty}+(k-\ell)^2)i_lu_0|u_0)\\
&\lesssim& \Vert
(u_1-ku_0)\Vert^2_{\CH}+(h_0u_0|u_0)=\Vert u\Vert^2_{\dot{\CE}}.
\end{eqnarray*}
\qed
\proof[Proof of Thm. \ref{Asympcompl2}]
We first show for $u\in \CE^{fin}_{r/l}$ the existence of the limit 
\[\tilde{W}_{r/l}u=\lim_{t\rightarrow\infty}e^{it\dot{H}_{\pm\infty}}i_{r/l}e^{-it\dot{H}_{r/l}}u\]
in $\dot{\CE}_{\pm\infty}$. Let $u\in \oplus_{|q|\le Q}\CE^q$. Using the
estimate 
\[\Vert e^{it\dot{H}_{\pm\infty}}i_{r/l}e^{-it\dot{H}_{r/l}}u\Vert_{\dot{\CE}_{\pm\infty}}\le
C(Q)\Vert u\Vert_{\dot{\CE}_{r/l}}\]
as well as the spherical symmetry of the problem, it is sufficient to
show for all $|q|\le Q$ the existence of the limits
\[\lim_{t\rightarrow\infty}e^{it\dot{H}_{\pm\infty}^q}i_{r/l}e^{-it\dot{H}^q_{r/l}}u^q,\]
where $u^q\in \CE^q$ and $\dot{H}^q_{\pm\infty}$ resp. $\dot{H}^q_{r/l}$ are the
restrictions of $\dot{H}_{\pm\infty}$ resp. $\dot{H}^l_{r/l}$
 to
$\dot{\CE}^q_{r/l}$. The existence of this limit follows from standard
arguments using the exponential decay of $\Delta_r$ at
$\pm\infty$. Using Thm. \ref{thACKerrSC} we obtain the existence
of the limit
\[\lim_{t\rightarrow\infty}e^{it\dot{H}}i^2_{r/l}e^{-it\dot{H}_{r/l}}u=W_{r/l}u.\]
We now want to show that there exists $C>0$ such that for all $u\in
\CE^{fin}_{r/l}$
\begin{equation}
\label{estWrl}
\Vert W_{r/l}u\Vert_{\dot{\CE}}\le C\Vert u\Vert_{\dot{\CE}_{r/l}}.
\end{equation}
We first consider $W_l$. By Lemma \ref{lem12.2.2} we can suppose $(u_0,u_1)\in \CD^{fin}_l$. Let $\supp u_0,\, \supp u_1\subset
(R_1,R_2)$. We decompose $(u_0,u_1)$ in incoming and outgoing
solutions according to the discussion at the beginning of this
subsection :
\begin{eqnarray*}
u_0=u_{0,l}^{\rm in}+u_{0,l}^{\rm out},\quad
u_1=u_{0,l}^{\rm in}+u_{0,l}^{\rm out}.
\end{eqnarray*}
By Remark \ref{suppinout} we have
\begin{eqnarray*}
\supp u_{0,l}^{\rm in},\, \supp u_{1,l}^{\rm in}\subset ]-\infty, R_2[\times S^2,\\
\supp u_{0,l}^{\rm out},\, \supp u_{1,l}^{\rm out}\subset ]R_1,\infty[\times S^2.
\end{eqnarray*}
Let $u_l^{\rm in}=(u_{0,l}^{\rm in},u_{1,l}^{\rm in}),\,
u_l^{\rm out}=(u_{0,l}^{\rm out},u_{1,l}^{\rm out})$. We have 
\[W_lu_l^{\rm out}=0,\]
because 
\[i^2_{l}e^{-it\dot{H}_l}u^{\rm out}_l=0\]
for $t$ sufficiently large. We have 
\[\supp e^{-it\dot{H}_l}u_l^{\rm in}\subset (]-\infty, R_2-t[\times S^2)\times(
]-\infty, R_2-t[\times S^2).\]
We then estimate for $t$ large
\begin{eqnarray*}
\Vert e^{it\dot{H}}i^2_le^{-it\dot{H}_l}u^{\rm in}\Vert_{\dot{\CE}}&\lesssim&\Vert
i_l^2e^{-it\dot{H}_l}u^{\rm in}\Vert_{\dot{\CE}}\\
&\lesssim&\Vert
u^{\rm in}\Vert^2_{\dot{\CE}_l}+\left(\left(\frac{\Delta_r}{r^2+a^2}P+\Delta_rm^2\right)\left(e^{-it\dot{H}_l}u^{\rm in}\right)_0\big\vert\left(e^{-it\dot{H}_l}u^{\rm in}\right)_0\right)\\
&\lesssim& \Vert
u^{\rm in}\Vert^2_{\dot{\CE}_l}+e^{-\kappa_-t}(Q+1)\Vert u^{\rm in}\Vert^2_{\CH}\\
&\rightarrow&\Vert u^{\rm in}\Vert^2_{\dot{\CE}_l},\quad t\rightarrow \infty.
\end{eqnarray*}
It follows 
\begin{equation*}
\Vert W_lu\Vert_{\dot{\CE}}\le C\Vert u\Vert_{\dot{\CE}_l},
\end{equation*}
which is the required estimate. The proof for $W_r$ is analogous. Part
$ii)$ is shown in the same way. The required estimate follows from
Lemma \ref{normcomp}.
\qed

\end{document}